\documentclass[11pt,a4paper,reqno]{amsart}
\usepackage{amssymb,amsmath,graphics,verbatim}
\usepackage{latexsym}
\usepackage{eucal}
\usepackage{a4wide}
\usepackage[colorlinks]{hyperref}
\usepackage[normalem]{ulem}
\usepackage{soul}
\usepackage{cancel}
\newcommand{\stkout}[1]{\ifmmode\text{\sout{\ensuremath{#1}}}\else\sout{#1}\fi}

\newcommand{\p}{\partial}

\newcommand{\cjd}{\rangle}
\newcommand{\cjg}{\langle}

\newcommand\R{\rr}

\newcommand\NN{\mathbb{N}}
\newcommand\ZZ{\mathbb{Z}}

\newcommand\Id{\operatorname{Id}}

\newcommand\Sym{\mathrm{Sym}}

\newcommand\xin{\overline{\xi}_0}

\newcommand{\pl}{\partial}
\newcommand{\la}{\lambda}
\newcommand{\mc}{\mathcal}
\newcommand{\rr}{\mathbb{R}}

\newcommand{\zz}{\mathbb{Z}}

\newcommand{\eps}{\epsilon}

\newcommand{\x}{\times}

\newcommand{\til}{\widetilde}
\newcommand{\bbar}{\overline}
\newcommand{\vol}{\mathrm{v}}

\newtheorem{theorem}{Theorem}
\newtheorem{lemma}[theorem]{Lemma}
\newtheorem{proposition}[theorem]{Proposition}
\newtheorem{corollary}[theorem]{Corollary}

\theoremstyle{definition}
\newtheorem{definition}[theorem]{Definition}
\newtheorem{remark}[theorem]{Remark}

\numberwithin{equation}{section}
\numberwithin{theorem}{section}

\DeclareMathOperator \Vol {Vol}

\title[Asymptotically Euclidean metrics without conjugate points are flat]{Asymptotically Euclidean  metrics\\ without conjugate points are flat}

\author{Colin Guillarmou}
\address{Colin Guillarmou\newline\indent CNRS, Universit\'e Paris-Sud, D\'epartement de Math\'ematiques, 91400
Orsay, France}
\email{cguillar@math.cnrs.fr}

\author{Marco Mazzucchelli}
\address{Marco Mazzucchelli\newline\indent CNRS, \'Ecole Normale Sup\'erieure de Lyon, UMPA, 69364 Lyon Cedex 07, France}
\email{marco.mazzucchelli@ens-lyon.fr}

\author{Leo Tzou}
\address{Leo Tzou\newline\indent School of Mathematics and Statistics, University of Sydney, NSW 2006, Australia}
\email{leo@maths.usyd.edu.au}

\date{September 4, 2019}

\begin{document}
\begin{abstract}
We prove that any asymptotically Euclidean metric on $\R^n$ with no conjugate points must be isometric to the Euclidean metric.
\end{abstract}

\maketitle

\section{Introduction}

The study of manifolds without conjugate points is a deep subject in Riemannian geometry and dynamical systems. A celebrated theorem due to Hopf \cite{Hopf:1948aa}, which extended previous work of Morse-Hedlund \cite{Morse:1942}, implies that the only Riemannian metrics with no conjugate points on a 2-dimensional torus are the flat ones. The analogous statement for higher dimensional tori had been a long-standing conjecture, finally proved by Burago-Ivanov \cite{Burago:1994aa}. 
Similarly, the flat metrics are rigid within a large class of Riemannian manifolds without conjugate points that are flat outside a compact set and have a ``simple'' topology. More specifically, any compact perturbation of the Euclidean metric on $\R^n$ must be either flat or have pairs of conjugate points, as was shown by Green-Gulliver \cite{Green:1985} in dimension $2$ and Croke \cite{Croke:1991aa} in higher dimension.
In this paper, we extend this rigidity result to the class of asymptotically Euclidean metrics on $\R^n$ that are short-range perturbations of the Euclidean metric.

We denote by $x_1,\dots,x_n$ the Cartesian coordinates on $\R^n$, and by $g_0=dx_1^2+...+dx_n^2$ the Euclidean metric. 
For an integer $m\geq 2$, we will say that a Riemannian metric $g$ on $\R^n$ is \emph{asymptotically Euclidean to order $m\geq1$} if, outside a compact neighborhood of the origin, it has the form
\begin{equation}\label{g-g_0orderm}
 g = g_0 + \frac{1}{|x|^m}\sum_{i,j=1}^n a_{ij}\Big(\frac{1}{|x|},\frac{x}{|x|}\Big)dx_idx_j
 \end{equation}
for some smooth functions $a_{ij}\in C^\infty ([0,1]\x \mathbb{S}^{n-1})$. This condition is actually independent of the choice of Cartesian coordinates $x_i$, and is formulated in Section \ref{s:geometricback} in terms of Melrose's scattering bundle. 
Our main result is the following rigidity theorem.
\begin{theorem}
\label{t:main}
Let $g$ be a Riemannian metric on $\R^n$ that is asymptotically Euclidean to order $m>2$ and without conjugate points. Then, there exists a diffeomorphism $\psi:\R^n\to\R^n$ such that $\psi^*g$ is the Euclidean metric.
\end{theorem}
Actually, the proof of Theorem~\ref{t:main} only needs the coefficients $a_{ij}$ to be $C^k$ for some $k>n+1-m$ that can be computed explicitly,  i.e. the Riemannian metric $g$ needs to have an asymptotic expansion in integer powers of $|x|^{-1}$ up to  $o(|x|^{-n-1})$.

We stress that the assertion of Theorem~\ref{t:main} is not true if we only assume $g$ to be without conjugate points and with sectional curvature converging to $0$ at infinity at speed $\mc{O}(1/|x|^2)$. Indeed, one can construct non-flat non-positively curved metrics on $\R^n$ asymptotic to a metric cone $dr^2+r^2h$ (in polar coordinates), where $h$ is a metric with constant curvature $\kappa \in (0,1)$ on $\mathbb{S}^{n-1}$, see \cite[Section 2.3]{Guillarmou:2019aa}.

The proof of Theorem \ref{t:main} will be carried over in four steps:

\vspace{5pt}

\noindent\emph{Step 1}. We show that an asymptotically Euclidean metric $g$ to order $m>2$ with no conjugate points must have the same scattering map and travel time as the Euclidean metric. Namely, if we denote by $S\R^n$ and $S^{0}\R^n=\R^n\times\mathbb{S}^{n-1}$ the unit tangent bundles of $g$ and $g_0$ respectively, the geodesic flows of $g$ and $g_0$ are conjugated by a diffeomorphism $\theta:S\R^n\to S^0\R^n$ tending to the identity at infinity, with an order of convergence comparable to $|x|^{-m+1}$.

\vspace{5pt}

\noindent\emph{Step 2}. The conjugacy $\theta$ of the geodesic flows allows us to compare the asymptotic volumes of Riemannian balls for the two metrics: if we denote by $B_{g}(x_0,R)$ and $B_{g_0}(x_0,R)$ the Riemannian balls of center $x_0$ and radius $R$ for $g$ and $g_0$ respectively, then for each $x_0\in\R^n$ we have
\begin{equation}\label{volestintro}
{\rm Vol}_g(B_{g}(x_0,R))={\rm Vol}_{g_0}(B_{g_0}(0,R))+\mc{O}(R^{n-m+1})
\quad\mbox{as }R\to\infty.
\end{equation} 

\vspace{5pt}

\noindent\emph{Step 3}. By employing Jacobi tensors in a similar way as in Croke's \cite{Croke:1991aa} and in Berger's proof of the Blaschke conjecture for $\mathbb{S}^n$ \cite{Besse:1978aa}, the asymptotics \eqref{volestintro} and the absence of conjugate points force $g$ to be flat provided it is asymptotically Euclidean to order $m>n+1$.

\vspace{5pt}

\noindent\emph{Step 4}. The particular structure of the Riemannian metric $g$ near infinity allows us to describe very precisely the behavior of the $g$-geodesics that do not enter the regions $|x|>R$ for large $R$: if we compactify $\R^n$ radially by adding a sphere $\mathbb{S}^{n-1}$ at infinity, we see that as $R\to\infty$ those geodesics converge as non-parametrised curves to geodesics of length $\pi$ on the unit round sphere $\mathbb{S}^{n-1}$ (namely, half great circles). Since $g$ and $g_0$ have the same scattering map, we show that the tensors in the asymptotic expansion of $g$ in powers of $1/|x|$ satisfy suitable equations. Such equations, together with suitable properties of the Funk transform (also known as the X-ray transform), imply that all the terms in the asymptotic expansion of $g$ must vanish, and therefore showing that $g$ is asymptotically Euclidean to any order $m$. This, together with Step 3, completes the proof of Theorem~\ref{t:main}.

\vspace{5pt}

Steps 1 and 3 build on the work of Croke \cite{Croke:1991aa}. As for Step 4, the behavior of the geodesics near infinity was studied by Melrose-Zworski in their work \cite{Melrose:1996} on the quantum scattering operator for the Laplacian of asymptotically Euclidean metrics. The recovery of the asymptotic tensors is inspired by the work of Joshi-S\'a Barreto \cite{Joshi:1999aa} on the quantum scattering operator for the Laplacian; nevertheless, our arguments are quite different from the ones of Joshi-S\'a Barreto, as we cannot employ the amplitudes that arise in the Fourier integral representation of the quantum scattering operator, a key tool in \cite{Joshi:1999aa} for the recovery of the terms in the asymptotic expansion of the Riemannian metric.

We conclude the introduction by mentioning some related rigidity results in the literature. A remarkable theorem of Croke \cite{Croke:1992} states that, on the universal cover $(M,g)$ of any closed Riemannian manifold without conjugate points, the volume of the Riemannian balls satisfy
\[
\liminf_{R\to \infty}\frac{{\rm Vol}_g(B_g(x_0,R))}{{\rm Vol}_{g_0}(B_{g_0}(0,R))}\geq 1,
\qquad
\forall x_0\in M,
\]
where, as above, $g_0$ is the Euclidean metric on $\R^n$; moreover, if this limit is equal to $1$ for some $x_0$, then $g$ is flat. Building on the  work of Hopf \cite{Hopf:1948aa}, Bangert-Emmerich \cite{Bangert:2013aa} showed that the same assertion holds in dimension $2$ under the only assumption that $M$ is a complete surface without conjugate points. In the case of non-simply connected manifolds, beside the already mentioned results of Hopf \cite{Hopf:1948aa} and Burago-Ivanov \cite{Burago:1994aa}, Burns-Knieper \cite{Burns:1991aa} proved that every Riemannian metric on the $2$-dimensional cylinder $\R\x\mathbb{S}^1$ that is without conjugate points, bounded cross-section, and curvature bounded from below must be flat; Croke-Kleiner \cite{Croke:1998ab} proved that compact perturbations of  complete flat Riemannian metrics on any non-compact manifold must have conjugate points or be flat.
As for the case of non-compact perturbations of the Euclidean metric, we mention the rigidity result of Innami \cite{Innami:1986} for Riemannian metrics on $\R^n$ with no conjugate points, integrable Ricci curvature, and vanishing integral scalar curvature. Compared to these results, our Theorem~\ref{t:main} has a rather strong assumption on the geometry at infinity, but on the other hand does not require the invariance of the Riemannian metric under a group action, nor global curvature assumptions, nor the fact that the Riemannian metric is a compactly supported perturbation of a good model.

\subsection*{Acknowledgements.} This project has received funding from the European Research Council (ERC) under the European Union’s Horizon 2020 research and innovation programme (grant agreement No. 725967). This work was written while the first and second authors were visiting the University of Sydney, we thank the institution for its support and hospitality.  
We finally thank Gerhard Knieper and Gabriel Paternain for helpful discussions and references.

\section{Geometric background}\label{s:geometricback} 
  
\subsection{Asymptotically Euclidean metrics}
We denote by $g_0$ the Euclidean metric on $\R^n$, and by $(r,y)$
the polar coordinates on $\R^n$, where $r=|x|=|x|_{g_0}$,  $y=x/r\in \mathbb{S}^{n-1}$, and $x=(x_1,\dots,x_n)$ are the Cartesian coordinates on $\R^n$. The Euclidean metric $g_0=dx_1^2+...+dx_n^2$ can be written in polar coordinates as $g_0=dr^2+r^2h_0$, where $h_0$ is the Riemannian metric of constant curvature $+1$ on $\mathbb{S}^{n-1}$. We set $\rho_0:=r^{-1}$ in the region $r\geq 1$ and extend it smoothly as a positive function on $\R^n$.
We compactify $\R^n$ radially by adding an $(n-1)$-sphere at infinity, i.e. at $\rho_0=0$. 
We denote by $\bbar{\R^n}$ this compactification, and equip it with the smooth structure that extends the one of $\R^n$ and is such that $(\rho_0,y)$ are smooth coordinates near $\pl \bbar{\R^n}=\mathbb{S}^{n-1}$. The Euclidean metric $g_0$ can  be written on $\R^n\setminus\{0\}$ as 
\begin{equation}\label{defg0} 
g_0=\frac{d\rho_0^2}{\rho_0^4}+\frac{h_0}{\rho_0^2}.
\end{equation}
Following Melrose \cite{Melrose:1994aa} (see also \cite[Chapter 2]{Melrose:1996ij} about boundary fibration structures), there is a smooth bundle, called the \emph{scattering tangent bundle} and denoted by ${^{\rm sc}T}\bbar{\R^n}\to\bbar{\R^n}$, whose space of smooth sections is identified with the space of smooth vector fields of the form
$\rho_0 V$, where $V$ are smooth vector fields on $\bbar{\R^n}$ tangent to the boundary 
$\pl\bbar{\R^n}$. By means of this identification, in local coordinates $(\rho_0,y_1,\dots,y_{n-1})$ near $\pl\bbar{\R^n}$, a local frame of ${^{\rm sc}T}\bbar{\R^n}\to\bbar{\R^n}$
is given by $\rho_0^2\pl_{\rho_0},\rho_0\pl_{y_1},\dots,\rho_0\pl_{y_{n-1}}$. 
Using polar coordinates, it is straightforward to check that the vector fields $\pl_{x_i}$ are smooth sections of ${^{\rm sc}T}\bbar{\R^n}$. Equivalently,  smooth scattering vector field $W$ are those vector fields on $\R^n$ that can be written outside a compact set as 
\[W=\sum_{j=1}^na_j\Big(\frac{1}{|x|},\frac{x}{|x|}\Big)\pl_{x_j}.\]
for some $a_j\in C^\infty([0,1)\x \mathbb{S}^{n-1})$.
The vector bundle dual to ${^{\rm sc}T}\bbar{\R^n}$ is denoted by ${^{\rm sc}T}^*\bbar{\R^n}$. Near $\partial\bbar{\R^n}$, it admits the smooth frame $\rho_0^{-2}d\rho_0$ and 
$\rho_0^{-1}\omega$, where $\omega$ is any smooth $1$-form on $\bbar{\R^n}$ such that $\omega(\pl_{\rho_0})|_{\pl\bbar{\R^n}}=0$. As before, the 1-forms $dx_i$  extend as smooth sections of ${^{\rm sc}T}^*\bbar{\R^n}$ on the whole $\bbar{\R^n}$. 
Therefore, the Euclidean metric $g_0$ is a smooth section of the symmetric tensor bundle $S^2({^{\rm sc}T}^*\bbar{\R^n})$, i.e. $g_0\in C^{\infty}(\bbar{\R^n};S^2({^{\rm sc}T}^*\bbar{\R^n}))$.

\begin{definition}\label{defAE} 
A Riemannian metric $g$ on $\R^n$ is \emph{asymptotically Euclidean} to order $m\geq 1$
when $g-g_0\in \rho_0^mC^{\infty}(\bbar{\R^n};S^2({^{\rm sc}T}^*\bbar{\R^n}))$ and $\rho_0^{-2}|d\rho_0|_{g}=1+\mc{O}(\rho_0^2)$.
When $m=1$, we simply say that $(\R^n,g)$ is asymptotically Euclidean.
\end{definition}
This condition implies that, near $\pl\bbar{\R^n}$, we have
\begin{equation}\label{g-g0}
g-g_0=\rho_0^m\Big(a\frac{d\rho_0^2}{\rho_0^4}+\frac{\sum_j b_jd\rho_0\, dy_j}{\rho_0^3}+
\frac{\sum_{ij}k_{ij}dy_i\,dy_j}{\rho_0^2}\Big)
\end{equation} 
where $a,b_j,k_{ij}\in C^{\infty}(\bbar{\R^n})$ with
$\rho_0^ma=\mc{O}(\rho_0^{\max(2,m)})$. In Cartesian coordinates, the condition $g-g_0\in \rho_0^mC^{\infty}(\bbar{\R^n};S^2({^{\rm sc}T}^*\bbar{\R^n}))$ means that for $|x|>1$
\[ g_{ij}(x)=\delta_{ij}+\frac{1}{|x|^m}\sum_{i,j=1}^na_{ij}\Big(\frac{1}{|x|},\frac{x}{|x|}\Big)dx_idx_j\]
for some $a_{ij}\in C^\infty([0,1]\x \mathbb{S}^{n-1})$, and 
this further implies that, for each multi-index $\beta\in \NN^{n}$, 
\begin{align*}
|\pl_x^{\beta}(g-g_0)_{ij}|\leq C_{\beta}|x|^{-m-|\beta|},
\qquad
\forall x\in \R^n\mbox{ with }|x|\geq 1
\end{align*}
for some $C_\beta>0$. The condition $\rho_0^{-2}|d\rho_0|_{g}=1+\mc{O}(\rho_0^2)$ is only relevant if $m=1$ (as otherwise automatically satisfied) and corresponds to 
the condition $\sum_{ij}a_{ij}x_ix_j=\mc{O}(1)$ as $|x|\to \infty$.
It was proved by Joshi-Sa Barreto \cite{Joshi:1999aa} that an asymptotically Euclidean Riemannian metric admits an approximate normal form near the boundary $\pl\bbar{\R^n}$. In fact, it also admits an exact normal form. We recall the following lemma which can be viewed as a sort 
of polar coordinates representation of $g$ and is extracted from \cite[Lemma 2.2.]{Guillarmou:2019aa}.
\begin{lemma}[]\label{normalform}
Let $g$ be asymptotically Euclidean to order $m\geq 1$. 
There exists a boundary defining function $\rho\in C^\infty(\bbar{\R^n})$  that satisfies 
\begin{equation}
\label{|drho|=1}
\frac{|\nabla^g\rho|_g}{\rho^2}=1,
\qquad
\rho=\rho_0(1+\mc{O}(\rho_0^m)).
\end{equation}
If $m\geq 2$, the function $\rho$ is uniquely determined near $\pl\bbar{\R^n}$ by \eqref{|drho|=1}. 
If $m=1$, the function $\rho$ is not unique: for each $\omega_0\in C^\infty(\pl\bbar{\R^n})$ there exists a function $\rho=\rho_0+\omega_0\rho_0^2+\mc{O}(\rho_0^3)$ such that 
$\rho^{-2}|\nabla^g\rho|_g=1$.

For every such boundary defining function $\rho$, there is a smooth diffeomorphism \[\psi:[0,\eps)_s \x \pl\bbar{\R^n}\to U\subset \bbar{\R^n}\] onto a collar neighborhood $U$ of $\pl\bbar{\R^n}$ such that $\psi(0,\cdot)|_{\partial\bbar{\R^n}}=\Id$, $\psi^*\rho=s$ and 
\[\psi^*g=\frac{ds^2}{s^4}+\frac{h_s}{s^2},\] 
where $h_s$ is a smooth family of Riemannian metrics on $\pl\bbar{\R^n}$ satisfying 
\[h_s-h_0\in s^m C^\infty([0,\eps)\x\pl\bbar{\R^n};S^2(T^*\pl\bbar{\R^n})).\] 
The diffeomorphism $\psi$ is uniquely determined by $\rho$ by means of the formula $\psi(s,y)=e^{sZ_\rho}(y)$, where $Z_\rho:=\rho^{-4}\nabla^g\rho\in C^\infty(\bbar{\R^n};T\bbar{\R^n})$.
\end{lemma}

Let us consider a diffeomorphism $\psi$ as in the previous lemma, and the associated function $\rho$. We extend $\rho$ arbitrarily to a non-vanishing smooth function on the whole $\R^n$. We consider the corresponding diffeomorphism $\psi_0: [0,\infty)\x \mathbb{S}^{n-1}\to \overline{\R^n}\setminus \{0\}$, $\psi_0(\rho_0,y)= y/\rho_0$ associated with the Euclidean metric $g_0$. Up to pulling back $g$ by any diffeomorphism which is equal to 
$\psi\circ \psi_0^{-1}$ outside a compact subset of $\R^n$, we can always assume that $g$ and $g_0$ are in the same normal form near $\pl\bbar{\R^n}$. Namely, $\rho(x)=\rho_0(x)=|x|^{-1}$ outside a compact subset of $\R^n$.

\begin{lemma}\label{curvdecay}
Let $g$ be an asymptotically Euclidean Riemannian metric to order $m$. Its Christoffel symbols $\Gamma_{ij}^k$ with respect to the Cartesian coordinates on $\R^n$ satisfy
\begin{equation}
\label{Gammaijk}
\begin{split}
\Gamma_{ij}^k(x) & =\mc{O}(\rho(x)^{m+1}),\\
\max_{|v|_g=1} d\Gamma_{ij}^k(x)v & =\mc{O}(\rho(x)^{m+2}),
\end{split} 
\end{equation}
and its Riemann tensor $R$ satisfies
\begin{equation}\label{curvtensor}
\max_{|v_1|_g=...=|v_4|_g=1} R_x(v_1,v_2,v_3,v_4)=\mc{O}(\rho(x)^{m+2}).
\end{equation} 
\begin{proof}
The vector fields $\pl_{x_k}$ are smooth sections of $^{\rm sc}T\bbar{\R^n}$. Using Koszul formula and writing $g=g_0+\rho^{m}q$ for some $q\in C^\infty(\bbar{\R^n};S^2({^{\rm sc}T}^*\bbar{\R^n}))$, we get
\begin{align*}
2\Gamma_{ij}^k & = \pl_{x_i}g(\pl_{x_j},\pl_{x_k})+\pl_{x_j}g(\pl_{x_i},\pl_{x_k})-
\pl_{x_k}g(\pl_{x_i},\pl_{x_j})\\
 & = \pl_{x_i}(\rho^mq(\pl_{x_j},\pl_{x_k}))+\pl_{x_j}(\rho^mq(\pl_{x_i},\pl_{x_k}))-
\pl_{x_k}(\rho^mg(\pl_{x_i},\pl_{x_j})).
\end{align*}
This, together with $\rho^m q(\pl_{x_j},\pl_{x_k})\in \rho^m C^\infty(\bbar{\R^n})$ for all $j,k$, implies the decays in \eqref{Gammaijk}. As for \eqref{curvtensor}, we proceed similarly by writing 
$R^g(\pl_{x_i},\pl_{x_j},\pl_{x_k},\pl_{x_\ell})$ in terms of the the Christoffel symbols and of the second derivatives of $g$.   
\end{proof}

\end{lemma}

\subsection{Rescaled geodesic flow} 
\label{s:rescaled_geodesic_flow}
Let $g$ be an asymptotically Euclidean Riemannian metric to order $m\geq 1$ on $\R^n$ without conjugate points.
The geodesic flow associated with $g$ on the unit cotangent bundle 
$S^*\R^n=\big\{(x,\xi)\in T^*\R^n\ \big|\ |\xi|_g=1\big\}$ 
is the flow of the its geodesic vector field $X$, which is the Hamiltonian vector field of the Hamiltonian $H(x,\xi)=\frac{1}{2}|\xi|^2_g$. It is shown in \cite[Section 2.4]{Guillarmou:2019aa} that $X$ can be better described on a non-compact manifold with boundary, denoted by $\bbar{S^*\R^n}$, whose interior is $S^*\R^n$ and whose boundary consists of two connected components diffeomorphic to $T^*\pl\bbar{\R^n}$. The manifold $\bbar{S^*\R^n}$ is defined as 
follows. The unit scattering tangent bundle is an $(n-1)$-sphere bundle over $\bbar{\R^n}$ whose total space is the compact smooth manifold with boundary   
\[ 
{^{\rm sc}S}^*\bbar{\R^n}:=\big\{ (x,\xi)\in {^{\rm sc}T}^*\bbar{\R^n}\ \big|\ |\xi|_{g_x}=1\big\}.
\]
The coordinates $(\rho,y)$ on $\bbar{\R^n}$ induce local coordinates $(\rho,y,\bbar{\xi}_0,\bbar{\eta})$ on ${^{\rm sc}T}^*\bbar{\R^n}$ as follows: every covector $\xi\in{^{\rm sc}T}^*\bbar{\R^n}$ can be written as 
\begin{equation}\label{scatxi}
\xi=\bbar{\xi}_0\frac{d\rho}{\rho^2}+\sum_{i=1}^{n-1}\bbar{\eta}_i\frac{dy_i}{\rho}.
\end{equation} 
Therefore, $\xi\in{^{\rm sc}S}^*\bbar{\R^n}$ if and only if $\bbar{\xi}_0^2+|\bbar{\eta}|_{h_\rho}^2=1$. We consider the (radial) blow-up 
\[ [{^{\rm sc}S}^*\bbar{\R^n}; L_\pm] \]
of ${^{\rm sc}S}^*\bbar{\R^n}$ at the submanifolds $L_\pm:=\{ \bbar{\xi}_0=\pm 1, \rho=0\}\subset {^{\rm sc}S}^*\bbar{\R^n}$
obtained by removing $L_\pm$ and gluing in the inward pointing spherical normal bundle to $L_\pm$; see \cite[Chapter 5]{Melrose:1996ij} for details about blow-ups and \cite[Section 2.4]{Guillarmou:2019aa} for this particular case. The smooth structure on the blow-up is the one that makes the polar coordinates around $L_\pm$ smooth. 
The space $[{^{\rm sc}S}^*\bbar{\R^n}; L_\pm]$ is a manifold with codimension $2$ corners. The new boundary faces obtained from the blow-up are half-sphere bundles whose interior is denoted by $\pl_\pm S^*\R^n$, and are isomorphic to $T^*\pl\bbar{\R^n}$: using that in the region $\pm\xi_0>0$, 
$L_\pm=\{\rho=0,\bbar{\eta}=0\}$, we can use,
in a neighborhood of the interior of these new boundary faces, the projective coordinates    
\begin{equation}\label{projcoord} 
\rho, y,\eta:=\bbar{\eta}/\rho 
\end{equation}
are smooth coordinates and $\pl_\pm S^*\R^n=\{\rho=0\}$ in this neighborhood using the projective coordinate. Notice that $(y,\eta)$ restricted to $\pl_\pm S^*\R^n$ provide a diffeomorphism with $T^*\pl \R^n$. The variable $\bbar{\xi}_0$ is determined by $(\rho,y,\eta)$ near $\pl_\pm S^*\R^n$ by the equation
\begin{equation}\label{bbarS^*M} 
\bbar{\xi}_0^2+ \rho^2 |\eta|_{h_\rho}^2=1.
\end{equation}
The other boundary hypersurface of $[{^{\rm sc}S}^*\bbar{\R^n}; L_\pm]$ 
corresponds to pull-back of $\{\rho=0,\bbar{\eta}\not=0\}$ to $[{^{\rm sc}S}^*\bbar{\R^n}; L_\pm]$ by the blow-down map $\beta:[{^{\rm sc}S}^*\bbar{\R^n}; L_\pm]\to {^{\rm sc}S}^*\bbar{\R^n}$ and is denoted by ${\rm bf}$. 
We then define the non-compact manifold with boundary
\[\bbar{S^*\R^n}:= [{^{\rm sc}S}^*\bbar{\R^n}; L_\pm] \setminus {\rm bf}. \]
Note that the coordinates $(\rho,y,\xi_0,\eta)$ satisfying   
the condition \eqref{bbarS^*M} provide well-defined smooth coordinates on $\bbar{S^*\R^n}$, since the only region where $\eta=\bbar{\eta}/\rho$ was not defined on $[{^{\rm sc}S}^*\bbar{\R^n}; L_\pm]$ is ${\rm bf}$. We also note that $\rho$ is a smooth boundary defining function of $\pl_\pm S^*\R^n$
in $\bbar{S^*\R^n}$.
More informaly, this blown-up manifold leading to $\bbar{S^*\R^n}$ can be considered as the following process: using Lemma \ref{normalform}, when $g$ is asymptotically Euclidean to order $m\geq 2$, the coordinate $\rho$  uniquely defined by \eqref{|drho|=1} gives a decomposition  
$\bbar{\R^n}\setminus K\simeq [0,\eps)_\eps\x \pl\bbar{\R^n}$ for some compact set $K\subset \R^n$, inducing a decomposition of $T^*\bbar{\R^n}$ over $\bbar{\R^n}\setminus K$ under the form (using \eqref{scatxi})
\begin{equation}\label{defmcM}
T^*_{\bbar{\R^n}\setminus K}\bbar{\R^n}\simeq \mc{M}:=[0,\eps)_\rho \x \rr_{\bbar{\xi}_0}\x (T^*\pl\bbar{\R^n})_{y,\eta},
\end{equation}
and $\bbar{S^*\R^n}$ is just the fiber bundle given by $S^*\R^n$ over $K$, while $\bbar{S^*\R^n}\setminus S_K^*\R^n$ is diffeomorphic outside $K$ to the subset
\[\{(\rho,\bbar{\xi}_0,y,\eta)\in \mc{M}\ |\ \bbar{\xi}_0^2+\rho^2|\eta|_{h_\rho}^2=1\}.\]
The blown-up picture is an invariant way (not depending on $\rho$) to define this smooth manifold.
In the case $m=1$, $\rho$ is not uniquely defined by \eqref{|drho|=1}, and two such functions $\rho$ and $\hat\rho$ are related by $\hat{\rho}=\rho+\omega_0\rho^2+\mc{O}(\rho^3)$ for some $\omega_0\in C^\infty(\pl\bbar{\R^n})$; the induced coordinates $(\hat{\rho},\hat{\bbar{\xi}}_0,\hat{y},\hat{\eta})$ are related to $(\rho,y,\bbar{\xi}_0,\eta)$ by (for $\eta$ in a compact set)
\begin{equation}\label{changeofcoord}
\hat{y}=y+\mc{O}(\rho), \quad \hat{\bbar{\xi}}_0=\bbar{\xi}_0+\mc{O}(\rho),\quad \hat{\eta}=\eta+d\omega_0+\mc{O}(\rho^2).
\end{equation}
In \cite[Lemma 2.4]{Guillarmou:2019aa}, it is shown that the rescaled geodesic vector field 
\begin{align}
\label{Xbar}
\bbar{X}=:\rho^{-2}X 
\end{align}
extends as a smooth vector field on the whole $\bbar{S^*\R^n}$ that is transverse to the boundary $\pl_\pm S^*\R^n$ of $\bbar{S^*\R^n}$. Its restriction to $\pl_\pm S^*\R^n$ is given by
\begin{align}
\label{e:bbarX_local_coordinates}
 \bbar{X}|_{\pl_\pm S^*\R^n}=\mp \pl_{\rho}+Y,
\end{align}
where $Y$ is the Hamiltonian flow of $\frac{1}{2}|\eta|^2_{h_0}$ on $T^*\pl\bbar{\R^n}$. Moreover, $d\rho(\bbar{X})=\pm 1+\mc{O}(\rho)$ near $\pl_{\mp}S^*\R^n$, which implies that on compact sets $\Omega$ of $\bbar{S^*\R^n}$, $\bbar{X}$ is transverse to the hypersurfaces $\{\rho=\eps\}$ for each $\eps>0$ small enough depending on $\Omega$.
In the local coordinates $(\rho,y,\bbar{\xi}_0,\eta)$ (subject to the condition $\bbar{\xi}_0^2+\rho^2|\eta|^2_{h_\rho}=1$),   
\begin{equation}\label{formbarX}
\bbar{X} =    \xin  \p_\rho   +  \sum_{i,j}h^{ij}_\rho\eta_i
\p_{y_j} -  \big(|\eta|^2 + \frac{1}{2}\rho \p_\rho |\eta|^2_{h_\rho}\big) \rho\p_{\xin} -
\frac 12  \sum_{j}\p_{y_j}(|\eta|^2_{h_\rho}) \p_{\eta_j}.
\end{equation}
We denote by $\bbar{\varphi}_\tau$ the flow of 
$\bbar{X}$, which is smooth on $\bbar{S^*\R^n}$. Its flow lines are mapped by the base projection $\pi_0:T^*\R^n\to \R^n$ to the geodesics of $g$. By \eqref{formbarX}, for each $(y,\eta)\in \pl_-S^*\R^n\simeq T^*\pl\bbar{\R^n}$ the geodesic $\bbar{\gamma}(\tau)=\pi_0(\bbar{\varphi}_\tau(y,\eta))=(\rho(\tau),y(\tau))$ is parametrized so that
\begin{equation}\label{behavbargamma}
\rho(\tau)=\tau+\mc{O}(\tau^3),  \quad y(\tau)=y+\tau \eta^\sharp +\mc{O}(\tau^2),\quad\mbox{as }\tau\to0
\end{equation}
locally uniformly in $\eta$. Here, $\eta^\sharp$ is the tangent vector to $\pl\bbar{\R^n}=\mathbb{S}^{n-1}$ so that $\eta=h_0(\eta^\sharp,\cdot)$, where $h_0$ is the round Riemannian metric on the unit sphere $\mathbb{S}^{n-1}$.

The following lemma follows from \cite[Section 2.7]{Guillarmou:2019aa}.
\begin{lemma}\label{behavgeod}
For each unit-speed $g$-geodesic $\gamma:\R\to \R^n$, there exist $y_\pm\in \pl\bbar{\R^n}$ and $\eta_\pm\in T^*_{y_\pm}\pl\bbar{\R^n}$ such that  
\[ \lim_{t\to \pm\infty}(\gamma(t),\dot{\gamma}(t)^\flat)= \Big(y_\pm, \mp \frac{d\rho}{\rho^2}+ \eta_\pm\Big)\in \pl_\pm S^*\R^n,\]
where $\dot{\gamma}(t)^\flat=g(\dot\gamma,\cdot)$.
In particular, $\lim_{t\to \pm\infty}\gamma(t)= y_\pm$ in $\bbar{\R^n}$. 
\end{lemma}
\begin{proof}
Since $g$ has no-conjugate points, its exponential map at any $x_0\in \R^n$ is a diffeomorphism. This readily implies that $d(x_0,\gamma(t))\to \infty$ as $t\to \pm \infty$, and therefore $\rho(\gamma(t))\to 0$ as $t\to \pm\infty$. Let $T>0$ be large enough so that, for all $t\geq T$, $\gamma(t)$ is contained a neighborhood of $\bbar{\R^n}$ where $(\rho,y)$ are local coordinates. Up to further increasing $T$, we have by \eqref{e:bbarX_local_coordinates} that $d\rho(\gamma(t))\dot{\gamma}(t)<0$ for all $t\geq T$. Since $\bbar{X}=\rho^{-2}X$, there exists a diffeomorphism $t:[0,\tau_+)\to[T,\infty)$ such that, if we set $\bbar{\gamma}(\tau):=\gamma(t(\tau))$, then  $(\bbar{\gamma}(\tau),\rho^2(\bbar{\gamma}(\tau))\dot{\bbar{\gamma}}(\tau)^\flat)$ is an integral curve of $\bbar{X}$ with $d\rho(\dot{\bbar{\gamma}}(\tau))<0$ for all $\tau\geq0$. We write this flow line of $\bbar X$ in the local coordinates $(\rho,y,\bbar{\xi}_0,\eta)$ as $(\rho(\tau),y(\tau),\bbar{\xi}_0(\tau),\eta(\tau))$. By  \cite[Lemma 2.6]{Guillarmou:2019aa}, we have that $|\eta(\tau)|_{h_0}\leq C|\eta_0|_{h_0}$ 
for some $C>0$ independent of  $\tau\in [0,\tau^+)$. This means that the trajectories stay in a compact subset of $\bbar{S^*\R^n}$. Since $\bbar{X}$ is transverse to the boundary $\pl_+S^*\R^n$ and non-vanishing, then $\tau^+<\infty$ and there is $(y_+,\eta_+)\in T^*\pl\bbar{\R^n}$ so that
\begin{equation}\label{limtau-} 
\lim_{\tau\to \tau^+}(\rho(\tau),y(\tau),\bbar{\xi}_0(\tau),\eta(\tau))=\bbar{\varphi}_{\tau^+}(\gamma(0),\dot{\gamma}(0)^\flat)=(0,y_+,-1,\eta_+).
\end{equation}
The argument for $t\to-\infty$ is analogous.
\end{proof} 
Since $\bbar X=\rho^{-2}X$, for any $g$-geodesic $\gamma(t)$ with 
\[ \lim_{t\to \pm\infty}(\gamma(t),\dot{\gamma}(t)^\flat)= \Big(y_\pm, \mp \frac{d\rho}{\rho^2}+ \eta_\pm\Big)\in \pl_\pm S^*\R^n,\]
the curves $\bbar{\gamma}(\tau):=\pi_0(\varphi_{\tau} (y_-,\eta_-))$ satisfies
\begin{equation}\label{formulataut}
\gamma(t) = \bbar\gamma(\tau(t)),\ {\rm where }\ \tau(t)=\int_{-\infty}^t\rho^2(\gamma(s))ds.
\end{equation}
The inverse function of $\tau(t)$ is given by 
\begin{eqnarray}
\label{formulattau}
t(\tau)=\int_{\tau(0)}^\tau \rho^{-2}(\bbar{\gamma}(s))ds.
\end{eqnarray}

\begin{corollary}\label{limitgeod}
For each unit-speed $g$-geodesic $\gamma$, in the projective coordinates~\eqref{projcoord} we have
$\rho(\gamma(t))=1/|t|+\mc{O}(|t|^{-2})$  as $t\to\pm \infty$, and there exist Euclidean geodesics $\gamma_0^\pm$  such that
\begin{eqnarray}
\label{geodesics at the two limits}
\lim_{t\to\pm\infty}
d_{g_0}(\gamma(t),\gamma^\pm_0(t))= 0, \quad \lim_{t\to\pm\infty}\big(\dot{\gamma}(t)-\dot{\gamma}_0^\pm(t)\big)= 0.
\end{eqnarray}
If we denote by $S^0\R^n$ and $S\R^n$ the unit tangent bundles associated with $g_0$ and $g$, there exists a diffeomorphism $\theta: S^0\R^n \to S\R^n$ satisfying
\begin{align*}
{ \theta(\gamma_0^-(t),\dot\gamma_0^-(t))=(\gamma(t),\dot\gamma(t)),}
\end{align*}
\end{corollary}
\begin{proof}
Let $X_0$ be the geodesic vector field on the unit tangent bundle  $S^0\R^n$ associated with the Euclidean metric $g_0$. We define the rescaled Euclidean geodesic vector field by $\bbar X_0:=\rho^{-2} X_0$, and its flow by $\varphi_t^0$. We fix a $g$-geodesic $\gamma(t)$ parametrized by unit-speed, and consider  $(y_-,\eta_-)\in \pl_-S^*\R^n$ given by Lemma \ref{behavgeod}. Using the implicit function theorem ($\bbar{X}$ is transverse to $\pl_-S^*\R^n$), we deduce that the map $(\gamma(0), \dot \gamma(0))\mapsto (y_-,\eta_-)$ is smooth. 
Consider also the rescaled Euclidean geodesic 
\[\bbar{\gamma}_0(\tau):=\pi_0(\bbar{\varphi}_\tau^{0}(y_-,\eta_-))=
(\rho_0(\tau),y_0(\tau)),\] 
where $\pi_0:S^0\R^n\to \R^n$, and similarly the rescaled $g$-geodesic 
\[\bbar{\gamma}(\tau):=\pi_0(\bbar{\varphi}_\tau(y_-,\eta_-))=(\rho(\tau),y(\tau)).\] 
By \eqref{behavbargamma} applied with the metric $g$ and $g_0$, for all sufficiently small $\tau>0$ we have
\begin{equation}\label{formrhoy} 
\rho(\tau)=\tau(1+\mc{O}(\tau^2)), 
\qquad 
y(\tau)=y_-+\tau \eta_{-}^\sharp +\mc{O}(\tau^2).
\end{equation}
The analogous relations holds for $\rho_0(\tau)$ and $y_0(\tau)$.
Define $\tau(t,\eta)$ and $t(\tau,\eta)$ by \eqref{formulataut} and \eqref{formulattau} respectively. By \eqref{behavbargamma}, we have that $t(\tau)=\tau^{-1}(-1+f(\tau))$ with $f\in C^{\infty}$  satisfying $f(\tau)=c_0\tau+\mc{O}(\tau^2)$ as $\tau\to 0$ for some $c_0\in \rr$, therefore
\begin{equation}\label{rote}
\tau(t)=\frac{1}{t}+\frac{c_0}{t^2}+\mc{O}(t^{-3})
\quad \mbox{as }t\to -\infty.
\end{equation}
We also note that $c_0$ is smooth as a function of $(y_-,\eta_-)$, thus as a function of $(\gamma(0),\dot\gamma(0))$. We recall that $\rho=1/|x|$ near $\pl\bbar{\R^n}$. We set 
\[\tau_0(t):=\int_{-\infty}^t \rho^2(\gamma_0^-(s))ds\] 
where $\gamma_0^-(s)$ is a unit-speed Euclidean geodesic equal as a curve to $\bbar{\gamma}_0$.  In particular there exists a unique $v_0\in \R^n$, $|v_0|_{g_0}  = 1$ and $x_0\in \R^n$ with $g_0(x_0,v_0) = 0$ such that 
$$ \gamma_0^-(t) = x_0 + t v_0 + t_0 v_0$$
for some $t_0\in \R$. The choice of $x_0$ and $v_0$ depends smoothly on $(y_-,\eta_-)$ which in turn depends smoothly on $(\gamma(0),\dot \gamma(0))$. We now show that there is a unique choice of $t_0$ depending smoothly on $(y_-,\eta_-)$ such that the limits in the statement of this Corollary is satisfied.

Since $\gamma_0^-(s)$ belongs to the region where $\rho=|x|^{-1}$ if $|s|$ is large enough, we get the asymptotic as $t\to -\infty$
\[\tau_0(t)=\int_{-\infty}^t |x_0+sv_0 + t_0 v_0|^{-2}ds=\frac{1}{t}-\frac{2t_0}{t^2}+\mc{O}(t^{-3})
\]
and we can choose $t_0=-c_0/2$. Note that, by construction, $t_0$ then depends smoothly on $(y_-,\eta_-)$ and thus on $(\gamma(0),\dot \gamma(0))$.
Making this choice of $t_0$ implies that
$\tau(t)=\tau_0(t)(1+\mc{O}(t^{-2}))$.
The Euclidean distance between $\gamma_0^-(t)=\bbar{\gamma}_0(\tau_0(t))$ and $\gamma(t)=\bbar{\gamma}(\tau(t))$ is thus 
\[ \Big| \frac{y_0(\tau_0(t))}{\rho_0(\tau_0(t))}-\frac{y(\tau(t))}{\rho(\tau(t))}\Big|=\mc{O}(1/|t|).\]
where we used \eqref{formrhoy} for $g$ and $g_0$. This proves the first limit for $\gamma_0^-$.
We also have $\dot{\gamma}(t)=\xi_0(\tau(t))\rho(\tau(t))^2\pl_{\rho}+\rho(\tau(t))^2\eta^\sharp(\tau(t))$ with $\eta^\sharp=h_\rho(\eta,\cdot)$, and similarly for $\dot{\gamma}_0^{-}(t)$. Since $\tau(t)=\tau_0(t)(1+\mc{O}(t^{-2}))$ and  
\begin{align*}
\lim_{\tau\to0} \xi(\tau)=\lim_{\tau_0\to0} \xi_0(\tau_0),
\qquad
\lim_{\tau\to0} \eta(\tau)=\lim_{\tau_0\to0} \eta_0(\tau_0),
\end{align*}
we conclude that $\dot{\gamma}(t)-\dot{\gamma}_0^-(t)\to 0$ as $t\to -\infty$.
The same argument applies as $t\to +\infty$ for another Euclidean geodesic $\gamma_0^+$. 
We notice that the unit-length parametrized Euclidean geodesic $\gamma_0^-(t)$ satisfying \eqref{geodesics at the two limits} is unique. Furthermore, we saw in our construction that all parameters determining $\gamma_0^-(t)$ depends smoothly on the value of $(\gamma(0),\dot\gamma(0))$, thus the map 
\begin{eqnarray}
\label{theta inverse}
(\gamma(0),\dot\gamma(0)) \mapsto (\gamma_0^-(0), \dot \gamma_0^-(0))
\end{eqnarray}
is smooth from $S\R^n \to S^0\R^n$. 
This map is also clearly injective by construction: two points mapping to the same image must be on the same $g$-geodesic $\gamma$, and $\theta$ maps the unit-length parametrized $\gamma$ to the unit-length parametrized $\gamma_0^-$. 
It is surjective since for $(\gamma_0^-(0),\dot\gamma_0^-(0))\in S^0\R^n$, the Euclidean geodesic 
$\cup_{t\in\R} (\gamma_0^-(t),\dot{\gamma}_0^-(t))$ is the image of the unique geodesic 
$\cup_{t\in\R}(\gamma(t),\dot{\gamma}(t))$ having the same backward endpoint at $\pl_-S^*\R^n$ for the rescaled flows. Let us show that the inverse is smooth: for $(\gamma_0^-(0),\dot\gamma_0^-(0))\in S^0\R^n$, one can find $x_0\in \R^n$, $|v_0|_{g_0} = 1$, and $g_0(x_0,v_0) = 0$ such that $\gamma_0^-(t + t_0) = x_0 + tv_0$ for some $t_0\in \R$, and all of $x_0$, $v_0$, and $t_0$  depends smoothly on $(\gamma_0^-(0),\dot\gamma_0^-(0))$. This geodesic has the same trajectory as $\pi_0(\bbar \varphi_\tau^0(y_-, \eta_-))$ for some $(y_-,\eta_-) \in T^*\partial \R^n$. Let $\epsilon>0$ be small, then the unit-speed $g$-geodesic $\gamma(t) := \pi_0(\varphi_t(\bbar\varphi_{\epsilon}(y_-,\eta_-)))$ depends smoothly on $(x_0,v_0)$. By the above argument showing the smoothness of the map \eqref{theta inverse}, there exists a unique $t_1\in \R$ depending smoothly on $(\gamma_0^-(0),\dot\gamma_0^-(0))$ such that $(\gamma(t),\dot\gamma(t))$ gets mapped to $(x_0 + (t+ t_1) v_0, v_0)$ by \eqref{theta inverse}. Therefore, 
$$\theta : (\gamma_0^-(0),\dot\gamma_0^-(0))\mapsto (\gamma(-t_1 - t_0), \dot \gamma(-t_1-t_0))$$
is precisely the inverse map of \eqref{theta inverse} and is smooth.
\end{proof}
The following estimate  was given in \cite[Lemma 2.6]{Guillarmou:2019aa}.
\begin{lemma}\label{lemmeGLT}
There is $C>1$ such that for all $\eps>0$ small enough and all unit-speed $g$-geodesics $\gamma$ satisfying $\rho(\gamma(0))=\eps$ and $d\rho(\dot{\gamma}(0))\leq 0$, we have 
\begin{equation}\label{firstboundrho}
\frac{\eps}{1+\eps t}\leq \rho(\gamma(t))\leq \frac{C\eps}{1+\eps t},
\qquad
e^{-C\eps}\leq \frac{|\eta(\gamma(t))|_{h_0}}{|\eta(\gamma(0))|_{h_0}}\leq e^{C\eps},
\qquad
\forall t\geq0.
\end{equation}
The analogous estimates hold for $t\leq 0$ if $d\rho(\dot{\gamma}(0))\geq 0$. 
\hfill\qed
\end{lemma}

\begin{corollary}\label{rholeq1/R}
There is $C>0$ and $R>0$ such that for all unit-speed $g$-geodesics $\gamma$ such that
\[ 
\lim_{t\to -\infty}(\gamma(t),\dot{\gamma}(t)^\flat)
= 
(y,\eta)\in \pl_-S^*\R^n\simeq T^*\pl\bbar{\R^n},\]
with $|\eta|_{h_0}\geq R$, we have
\begin{equation}
\label{uniformboundrho}
\rho(\gamma(t)) \leq \frac{C }{|t|+|\eta|_{h_0}},
\qquad \forall t\in\R.
\end{equation} 
\end{corollary}

\begin{proof}
Let $\epsilon_0>0$ be small enough so that the sublevel set $\{\rho\leq \epsilon_0\}$ is in the domain where~\eqref{projcoord} are well defined coordinates. Assume that $\gamma(t)$ is such that $\rho(\gamma(t))=\eps_0$ for some $t$ and let $t_0=\inf\{t\in\R\ |\ \rho(\gamma(t))=\eps_0\}$. 
Then the second inequality in \eqref{firstboundrho} implies that if $\eps_0$ is small enough, for each $t\leq t_0$
\[ \rho(\gamma(t_0))\leq |\eta(\gamma(t_0))|_{h_\rho}^{-1}\leq 
2|\eta(\gamma(t))|_{h_\rho}^{-1}.\]
Taking the limit as $t\to -\infty$, we get $\rho(\gamma(t_0))\leq 2/R$, thus we must have $2/R\geq \eps_0$. Then if $R$ is large enough, $\rho(\gamma(t))<\eps_0$ for all $t$ and the same argument as  above tells us that $\rho(\gamma(t))\leq 2/|\eta|_{h_0}$ for all $t\in\R$. 
It is easy to see that for $R$ large enough, and up to changing the parametrization $t\to t+t_0$ for some $t_0$, 
$\max_{t\in\R}\rho(\gamma(t))$ is attained at $t=0$ so that $d\rho(\dot{\gamma}(0))=0$. This means that $\rho(\gamma(0))=|\eta(\gamma(0))|_{h_\rho}^{-1}$ and the second inequality of \eqref{firstboundrho} gives that $\rho(\gamma(0))\geq \lim_{t\to -\infty}|\eta(\gamma(t))|_{h_\rho}/2\geq |\eta|_{h_0}/2$. The first inequality of  \eqref{firstboundrho} then yields \eqref{uniformboundrho}.
\end{proof}

\subsection{Scattering map}
Let $g$ be an asymptotically Euclidean Riemannian metric to order $m\geq 2$ on $\R^n$ with no conjugate points. We consider its rescaled geodesic vector field $\bbar{X}$ on $\bbar{S^*\R^n}$, which we defined in~\eqref{Xbar}, and we denote by $\bbar{\varphi}_\tau$ its flow on $\bbar{S^*\R^n}$. By Lemma \ref{behavgeod}, for each $(y,\eta)\in \pl_-S^*\R^n$ there exists $\tau^+(y,\eta)\in(0,\infty)$ such that $\bbar{\varphi}_{\tau}(y,\eta)\in S^*\R^n$ for all $\tau\in(0,\tau^+(y,\eta))$, and $\bbar{\varphi}_{\tau^+(y,\eta)}(y,\eta)\in \pl_+S^*\R^n$. The \emph{scattering map} of $(\R^n,g)$ is the $C^2$ map 
\begin{align*}
S_g:\pl_-S^*\R^n\to \pl_+S^*\R^n,
\qquad
S_g(y,\eta)=\bbar{\varphi}_{\tau^+(y,\eta)}(y,\eta).
\end{align*}
Since $m\geq 2$, we can identify $\pl_\pm S^*\R^n$ with $T^*\pl\bbar{\R^n}$ using the decomposition \eqref{defmcM} and the choice of $\rho=\rho_0+\mc{O}(\rho_0^2)$ giving the normal form of Lemma \ref{normalform}. Therefore, the scattering map can be seen as a map of the form $S_g:T^*\pl\bbar{\R^n}\to T^*\pl\bbar{\R^n}$. However, if we only had $m=1$, $\rho$ would not be uniquely defined by $\rho_0$, and given another boundary defining function $\hat{\rho}=\rho+\omega_0\rho^2+\mc{O}(\rho^3)$, its scattering map $\hat{S}_g:T^*\pl\bbar{\R^n}\to T^*\pl\bbar{\R^n}$ would only be conjugate to $S_g$ via the diffeomorphism $(y,\eta)\mapsto (y,\eta+d\omega_0(y))$.

The next lemma describes the scattering map of the Euclidean metric $g_0$ on $\R^n$.

\begin{lemma}\label{scatR^n}
The scattering map of the Euclidean space $(\R^n,g_0)$ is given by
\[ 
S_{g_0}:T^*\mathbb{S}^{n-1}\to T^*\mathbb{S}^{n-1},
\qquad
S_{g_0}(y,\eta)=(\Theta(y),(d\Theta(y)^{-1})^{T}\eta),\]
where $\mathbb{S}^{n-1}=\big\{y\in \R^n\ \big|\ |y|=1\big\}$ and $\Theta:\mathbb{S}^{n-1}\to \mathbb{S}^{n-1}$ is the antipodal map $\Theta(y)=-y$.
\end{lemma}
\begin{proof}
We denote by $\rho_0=|x|^{-1}$ and $y=x/|x|\in \pl\bbar{\R^n}$ the coordinates near the boundary at infinity $\pl\bbar{\R^n}$, where $\R^n$. Any unit-speed Euclidean geodesic $\gamma(t)=x_0+tv$, with $x_0\in \R^n$ and $v\in \mathbb{S}^{n-1}$,  satisfies 
\[ 
(\rho_0(\gamma(t)), y(\gamma(t)))=\Big(\frac{1}{|t|}+\mc{O}(t^{-2}),
v\frac{t}{|t|}+\frac{x_0}{|t|}-\frac{v(x_0.v)}{|t|}+\mc{O}(t^{-2}))
\Big)
\quad
\mbox{as }t\to\pm\infty.\]
This shows that $y(\gamma(t))\to \pm v$ as $t\to \pm\infty$. Now 
$\dot{\gamma}(t)^\flat=\sum_{i=1}^n v_idx_i$ if $v=(v_1,\dots,v_n)$, and in the $(\rho_0,y)$ coordinates (writing $\rho_0(t)=\rho_0(\gamma(t))$ and $y(t)=y(\gamma(t))$) we have as $t\to \pm \infty$
\[\begin{split}
\dot{\gamma}(t)^\flat= \sum_{i=1}^n v_idx_i=-(v.y(t))\frac{d\rho_0}{\rho_0(t)^2}+\sum_{i=1}^n\frac{v_idy_i}{\rho_0(t)}
\to \mp \frac{d\rho_0}{\rho_0^2}+\sum_{i=1}^n(-x_0+(x_0.v)v)_idy_i
\end{split}\]
using that $\sum_{i=1}^n y_i(t)dy_i=0$, where the limit is understood in the manifold $\bbar{S^*\R^n}$. This shows that $S_{g_0}(v,\eta)=(\Theta(v),(d\Theta(v)^{-1})^{T}\eta)$ if 
$\eta=\sum_{i=1}^n(-x_0+(x_0.v)v)_idy_i$. Since each $(v,\eta)\in T^*\mathbb{S}^{n-1}$ can be obtained this way, this achieves the proof.
\end{proof}

\section{Rigidity of the Euclidean plane}

In dimension 2, a result stronger than our Theorem \ref{t:main} holds. The 2-dimensional case is special, as it can be treated by means of Hopf's celebrated argument \cite{Hopf:1948aa} that shows that $2$-dimensional Riemannian tori 
 without conjugate points are flat. The statement that we obtain is probably well-known to the specialists, and seems weaker than an analogous rigidity theorem of Bangert-Emmerich \cite{Bangert:2013aa}. Since the argument is short and self-contained, we nevertheless include a full proof here.

\begin{proposition}
\label{rigidity_dim2}
Let $(\R^2,g)$ be a complete simply connected Riemannian plane without conjugate points, and $B_1\subset B_2\subset B_3\subset ...$ a sequence of compact subsets of $\R^2$ whose union is the whole $\R^2$, such that the following conditions hold:
\begin{itemize}
\item[(i)] Each $B_j$ is strictly geodesically convex.
\item[(ii)] The Gaussian curvature $K_g$ decays at infinity as follows
\begin{align*}
\lim_{j\to\infty} j^2 \,\, \max_{\partial B_j} |K_g| =0.
\end{align*}
\item[(iii)] The boundary of $B_j$ has length $\ell_j=\mathcal O(j)$ as $j\to\infty$ and, if we parametrize it with arc-length, its geodesic curvature $k_j:\rr/\ell_j\zz\to\rr$ satisfies
\begin{align*}
\lim_{R\to\infty} \int_0^{\ell_j} k_j(t)\,dt = 2\pi.
\end{align*}
\end{itemize}
Then $(\R^2,g)$ is isometric to the Euclidean $(\R^2,g_0)$.
\end{proposition}

\begin{proof}
We denote by $S\R^2$ the unit tangent bundle of $(\R^2,g)$.
For each $(x,v)\in S\R^2$, we denote by $\gamma_{x,v}:\rr\to \R^2$ the unit-speed $g$-geodesic with initial condition $\gamma(0)=x$ and $\dot\gamma(0)=v$. For each $T\in\rr$, we denote by $y_{x,v,T}:\rr\to\rr$ the smooth function satisfying the Jacobi equation with boundary conditions
\begin{align*}
\left\{
  \begin{array}{l}
    \ddot y_{x,v,T}(t) + K_g(\gamma_{x,v}(t))y_{x,v,T}(t)=0, \\ 
  	y_{x,v,T}(0)=1,\\
	y_{x,v,T}(T)=0.
  \end{array}
\right.
\end{align*}
Namely, if $j_{x,v}$ denotes the unit normal vector field to $\gamma_{x,v}$, the vector field $Y_{x,v,T}:= y_{x,v,T} j_{x,v}$ is a Jacobi vector field satisfying $Y_{x,v,T}(0)=j_{x,v}(0)$ and $Y_{x,v,T}(T)=0$. Since $(\R^2,g)$ is without conjugate points, a classical argument due to Hopf \cite[page~48]{Hopf:1948aa} implies that the limit
\begin{align*}
u(x,v) := \lim\limits_{T\to\infty} \frac{\dot y_{x,v,T}(0)}{y_{x,v,T}(0)}
\end{align*}
exists, and defines a smooth solution $u:S\R^2\to\rr$ of the Riccati equation
\begin{align}
\label{e:Riccati}
Xu + u^2 + K_g=0,
\end{align}
where $X$ denotes the geodesic vector field on $S\R^2$.

By assumption (i), each $B_j$ is geodesically convex. This, together with the fact that $(\R^2,g)$ has no conjugate points, implies that the geodesics leaving $B_j$ do not come back to it, i.e.
\begin{align*}
\gamma_{x,v}(\pm t)\not\in B_j,
\qquad
\forall (x,v)\in\partial_{\pm}SB_j,\ t>0,
\end{align*}
where $\pl_\pm SB_j:=\{(x,v)\in SB_j\ |\ \pm g(v,\nu)>0\}$ and $\nu$ is the outward unit normal to $\pl B_j$.
If we set
\[
c_j:=-\min\left\{0,\inf_{\R^2\setminus B_j} K_g\right\}\geq0,
\]
a well-known Sturmian argument \cite[pages~49,50]{Hopf:1948aa} implies that
\begin{align}
\label{e:Sturm}
\inf u|_{\partial_+SB_j} \geq - c_j^{1/2},
\qquad
\sup u|_{\partial_-SB_j} \leq  c_j^{1/2},
\end{align}

We denote by $\alpha$ the Liouville contact form on $S\R^2$, which is defined as 
\[\alpha_{x,v}(w)=g_x(v,d\pi(x,v)w),\] 
where $\pi:S\R^2\to \R^2$ is the base projection.
The exterior product $\alpha\wedge d\alpha$ is a volume form on $S\R^2$, which defines the orientation of $S\R^2$. The measure associated to $\alpha\wedge d\alpha$ disintegrates locally as the product of the Riemannian measure $\vol_g$ on $\R^2$ and of the Lebesgue measure on the unit circle $\mathbb{S}^1$. This, together with Gauss-Bonnet's theorem and assumption~(ii), implies
\begin{equation}
\label{e:integral_curvature}
\begin{split}
\lim_{j\to \infty}\int_{SB_j} K_g\,\alpha\wedge d\alpha
=
\lim_{j\to\infty} 2\pi\int_{B_j} \!\!\! K_g\, d\vol_g
=
4\pi^2 - 2\pi\lim_{j\to\infty} \int_{0}^{\ell_j} k_j(t)\,dt =0. 
\end{split}
\end{equation}
The geodesic vector field $X$ is the Reeb vector field associated to the contact form $\alpha$, i.e.\ $\alpha(X)\equiv1$ and $d\alpha(X,\cdot)\equiv0$. The geodesic vector field $X$ points toward the interior of $SB_j$ along $\partial_-SB_j$, and towards the exterior of $SB_j$ along $\partial_+SB_j$. Therefore, the 2-form $d\alpha=i_X(\alpha\wedge d\alpha)$ restricts to a volume form on $\partial_{\pm}SB_j$ 
which is positive on $\partial_+SB_j$ and negative on $\partial_-SB_j$. More precisely, let us denote by $\rho_\theta:S\R^2\to S\R^2$ the positive rotation of angle $\theta$ in the fibers of $S\R^2$, and by $\gamma_j:\rr/\ell_j\zz\to \partial B_j$ an arc-length parametrization of the boundary $\partial B_j$ oriented in such a way that $\rho_{\pi/2}(\dot\gamma_j(t))$ points outside $B_j$. For each $f\in C^{\infty}(\partial SB_j)$, we have
\begin{align*}
\int_{\partial SB_j}\!\!\!\!\!\! f\,d\alpha = \int_0^{\ell_j}  \!\!\int_{0}^{2\pi} f(\gamma_j(t),\rho_\theta(\dot\gamma_j(t)))\,\cos(\theta)\,d\theta\,dt.
\end{align*}
This, together with Stokes theorem and~\eqref{e:Sturm}, implies
\begin{align*}
\int_{SB_j}\!\!\! Xu\,\alpha\wedge d\alpha
& =
\int_{\partial SB_j}\!\!\! u\, d\alpha 
 \geq c_j^{1/2} \left( \int_{\partial_- SB_j}\!\!\! d\alpha  - \int_{\partial_+ SB_j}\!\!\! d\alpha \right)
 =- 4\ell_j c_j^{1/2}.
\end{align*}
Since $c_j=o(j^{-2})$ and $\ell_j=\mathcal O(j)$ as $j\to\infty$, we infer
\begin{align}
\label{e:integral_Xu}
\liminf_{j\to\infty} \int_{SB_j}\!\!\! Xu\,\alpha\wedge d\alpha \geq 0.
\end{align}
By~\eqref{e:Riccati}, \eqref{e:integral_curvature}, and \eqref{e:integral_Xu}, we infer
\begin{align*}
\limsup_{j\to\infty} \int_{SB_j}\!\!\! u^2 \alpha\wedge d\alpha
=-\lim_{j\to \infty}\int_{SB_j}\!\!\! K_g\,\alpha\wedge d\alpha - \liminf_{j\to\infty} \int_{SB_j}\!\!\! Xu\,\alpha\wedge d\alpha \leq 0,
\end{align*}
and thus $u\equiv0$. By plugging this into the Riccati equation~\eqref{e:Riccati}, we conclude that $K_g\equiv0$, and thus $(\R^2,g)$ is isometric to the Euclidean space $(\R^2,g_0)$.
\end{proof}

\begin{corollary}
\label{c:rigidity_dim2} 
Any asymptotically Euclidean Riemannian metric on $\R^2$ without conjugate points is isometric to the Euclidean metric.
\end{corollary}
\begin{proof}
Writing the asymptotically Euclidean metric in normal form associated to some boundary defining function $\rho>0$, we readily see that $B_j:=\{x\in \R^n\ |\ \rho(x)\geq j^{-1}\}$ is strictly geodesically convex for all $j>0$ large enough. Lemma \ref{curvdecay} gives 
that $\sup_{\pl B_j}|K_g|=\mc{O}(j^{-3})$. If we parametrize $\pl B_j$ by arc-length, and denote by $\ell_j$ its length and by $k_j:\R/\ell_j\ZZ\to\R$ its curvature, a straightforward computation also gives $\ell_j=\mc{O}(j)$ and $\int_{0}^{\ell_j} k_j(t) \to 2\pi$ as $j\to\infty$. Therefore Theorem \ref{rigidity_dim2} implies the result.
\end{proof}

\begin{remark}
The proof of Theorem \ref{rigidity_dim2} also shows (with a slight modification) that an asymptotically conic metric on $\R^2$ with asymptotic behavior 
given by $g= d\rho^2/\rho^4+ h_\rho/\rho^2$ on $(0,\eps)_\rho\x (\R/\ell \ZZ)_\theta$ with $h_\rho$ a smooth family of metrics on $\R/\ell \ZZ$ for some $\ell<2\pi$ and $h_0=d\theta^2$, then there must be conjugate points. This case corresponds to cone ends with angle less than the Euclidean one. On the other hand, asymptotically conic metrics on $\R^2$ with asymptotic cone angle $\ell>2\pi$ possess many metrics with negative curvature, thus have no conjugate points, see \cite[Section 2.3]{Guillarmou:2019aa}. \hfill\qed
\end{remark}

\section{Determination of the scattering map}
\label{s:scattering}

In this section we shall prove the following statement.
\begin{proposition}\label{samescat}
Any asymptotically Euclidean Riemannian metric to order $m>2$ on $\R^n$ with no conjugate points has the same scattering map as the Euclidean metric.
\end{proposition}

Let $g$ be an asymptotically Euclidean metric to order $m>2$ on $\R^n$, so that
\begin{align}
\label{e:g-g_0}
g - g_0 = \rho^{m-2} h_\rho,
\end{align} 
where $h_\rho$ is a smooth family of symmetric tensors on $\pl\bbar{\R^n}$. Given an absolutely continuous curve $\gamma:[a,b]\to\R^n$, we denote its $g$-length and its $g_0$-length respectively by
\begin{align*}
\ell_g(\gamma)=\int_a^b |\dot\gamma(t)|_g\,dt,
\qquad
\ell_{g_0}(\gamma)=\int_a^b |\dot\gamma(t)|_{g_0}\,dt.
\end{align*}
Corollary \ref{limitgeod} tells us that for each $g$-geodesic $\gamma:\R\to\R^n$ with $|\dot\gamma|_g\equiv1$ there exist two Euclidean geodesics $\gamma_0^\pm(t)=x_{\pm}+tv_{\pm}$ with $|v_\pm|_{g_0}=1$, such that
\begin{equation}\label{limitinggeod}
\lim_{t\to\pm\infty} d_{g_0}(\gamma(t),\gamma_0^\pm(t))=0,\qquad
\lim_{t\to\pm\infty} \dot\gamma(t)=v_\pm.
\end{equation}
In order to prove Proposition \ref{samescat},  it suffices to show that $x_+=x_-$ and $v_+=v_-$, so that $\gamma_0^-=\gamma_0^+$. This will be done in Lemma~\ref{same_scattering} using an argument inspired by Croke \cite[Section~6]{Croke:1991aa}. We will infer as a consequence that the Riemannian distance $d_g$ converges with a suitable rate to the Euclidean one $d_{g_0}$ at infinity (Lemma~\ref{dg0=dg_at_infinity}), a useful fact for the next section.
 
\begin{lemma}
\label{l:asymptotic_slope}
The asymptotic slopes at $t\to\infty$ and $t\to-\infty$ of any $g$-geodesic $\gamma$ are the same, i.e. $v_+=v_-$. In other words,  $\gamma$ has the same endpoints at infinity
as the Euclidean geodesic $\gamma_0^-(t)$ defined in Corollary \ref{limitgeod}. 
\end{lemma}

\begin{proof}
Let us assume by contradiction that $v_+=-v_-$. This implies that   
\begin{equation}\label{dg0gamma0^pm}
d_{g_0}(\gamma_0^+(T),\gamma_0^-(-T))= d_{g_0}(x_+,x_-)
\end{equation}
is independent of $T\in\R$. 
Since the Euclidean geodesic 
\[\til{\zeta}_T:[0,1]\to\R^n,\qquad \til\zeta_T(t)=t\gamma_0^+(T)+(1-t)\gamma_0^-(-T)\] escapes uniformly to infinity as $T\to\infty$ and has Euclidean length $\ell_{g_0}(\til{\zeta}_T)=d_{g_0}(x_+,x_-)$
independent of~$T$, we have using \eqref{e:g-g_0}
\begin{equation}\label{limzetaT}
\lim_{T\to\infty}\ell_{g}(\til{\zeta}_T)-\ell_{g_0}(\til\zeta_T)=0.
\end{equation}
On the other hand, 
\begin{align*}
\infty &=\lim_{T\to\infty} \ell_g(\gamma|_{[-T,T]})=
\lim_{T\to\infty} d_g(\gamma(T),\gamma(-T)) \\
&\leq
\lim_{T\to\infty}
d_g(\gamma(T),\gamma_0^+(T)) 
+
d_g(\gamma_0^+(T),\gamma_0^-(-T)) 
+
d_g(\gamma_0^-(-T),\gamma(-T))<\infty,
\end{align*}
which gives a contradiction by using \eqref{limzetaT} and \eqref{dg0gamma0^pm}.

Let us now assume by contradiction that $v_+\neq v_-$, so that $v_+$ and $v_-$ must be linearly independent. We break the proof into three steps.

\vspace{5pt}

\noindent\emph{Step 1.} For each $T>0$, we consider the Euclidean geodesic $\zeta_T:[0,1]\to\R^n$, $\zeta_T(t)=t\gamma(T)+(1-t)\gamma(-T)$. We claim that if $T>0$ is large enough, then $\zeta_T$ does not intersect the Riemannian ball $B_g(R)=\{x\in\R^n\ |\ d_g(0,x)<R\}$ for each $R>0$ fixed. In order to prove this, we first consider the Euclidean geodesic
\[
\til\zeta_T:[0,1]\to\R^n,
\quad
\til\zeta_T(t)=t\gamma_0^+(T)+(1-t)\gamma_0^-(-T)=T \underbrace{(t v_+ +(1-t)v_-)}_{=:v_t} + tx_++(1-t)x_-.\]
Since $v_+$ and $v_-$ are linearly independent, every vector $v_t$ is non-zero. This proves that $|\til\zeta_T(t)|\to\infty$ uniformly in $t\in[0,1]$ as $T\to\infty$ for all $t\in[0,1]$, and in particular $\til\zeta_T$ will not intersect the compact ball $B_g(R)$. Since $d_{g_0}(\gamma_0^-(-T),\gamma(-T))\to0$ and $d_{g_0}(\gamma_0^+(T),\gamma(T))\to 0$ as $T\to\infty$, we have that 
$d_{g_0}(\zeta_T(t),\til\zeta_T(t))\to 0$ uniformly in $t\in[0,1]$ as $T\to\infty$, and in particular $\zeta_T$ will not intersect $B_g(R)$ for all large enough $T>0$.

\vspace{5pt}

\noindent\emph{Step 2.} We claim that there is $C>0$ such that for each $T>0$, we have 
\[\ell_g(\zeta_T)\leq\ell_{g_0}(\zeta_T)+C,
\qquad
\ell_{g_0}(\gamma|_{[-T,T]})\leq\ell_{g}(\gamma|_{[-T,T]})+C.
\] 
Indeed, \eqref{e:g-g_0} implies
\begin{align*}
\ell_g(\gamma_{[-T,T]}) = \ell_{g_0}(\gamma_{[-T,T]}) +  \int_{-T}^T \Big(\sqrt{|\dot \gamma(t)|_{g_0}^2 + \rho^{m-2}(\gamma(t))  |\dot \gamma(t)|_{h}^2} - |\dot \gamma(t)|_{g_0} \Big) dt.
\end{align*}
We see that it suffices to show that for any fixed $T_0\in (0,T)$,
\begin{eqnarray}
\label{difference}
 \left |\int_{T_0 \leq |t| \leq T} \Big(\sqrt{|\dot \gamma(t)|_{g_0}^2 + \rho^{m-2}(\gamma(t))  |\dot \gamma(t)|_{h}^2} - |\dot \gamma(t)|_{g_0} \Big)  dt \right|.
\end{eqnarray}
is uniformly bounded in $T$. Observe that $\dot \gamma(t) =\bbar{\xi}_0(t)\rho^2(t) \partial_\rho +\rho^2(t) \sum_i\eta_i(t)\partial_{y_i}$ and, by Lemma \ref{behavgeod}, for all $\epsilon >0$ we may choose $T_0 >0$ such that, $\rho(t) \leq \epsilon$, $|\bbar\xi_0(t)| \geq 1/2$ for all $|t|\geq T_0$. By Lemma \ref{lemmeGLT}, we have $|\eta(t)|_h \leq C$ for all $|t|\geq T_0$. With these observations, \eqref{difference} can be estimated using Corollary \ref{limitgeod} 
by 
\[C\int_{T_0\leq |t|} \rho(t)^m dt \leq C\int_{T_0 \leq |t|} t^{-m} dt\leq C_{T_0}.\]
This proves the desired estimate for the $g_0$-length of $\gamma$. The same argument done by reversing $g$ and $g_0$ shows the desired estimate for the $g$-length of $\zeta_T$.

\vspace{5pt}

\noindent\emph{Step 3.} We claim that $\ell_{g_0}(\gamma|_{[-T,T]})-\ell_{g_0}(\zeta_T)\to+\infty$ as $T\to\infty$. Indeed, using \eqref{limitinggeod} there exists $c>0$ such that 
$d_{g_0}(\gamma(t), tv_\pm)\leq c$ for all $\pm t\geq 0$ and $\ell_{g_0}(\zeta_T) - d_{g_0}(Tv_+,-Tv_-)\leq 2c$ for all $T$. 
On the plane $P:=\mathrm{span}\{v_-,v_+\}$ the Euclidean line $[Tv_+,-Tv_-]$ intersects both lines 
$\R v_\pm$ with an angle $\alpha\in(0,\pi/2)$ independent of $T$. Therefore
\begin{align*}
T|v_++v_-|=d_{g_0}(Tv_+,-Tv_-) 
&\leq
\cos(\alpha) \big( d_{g_0}(0,Tv_-)+d_{g_0}(0,Tv_+) \big)\\
&\leq \cos(\alpha) \big( 2c + \ell_{g_0}(\gamma|_{[-T,0]})  + 2c+ \ell_{g_0}(\gamma|_{[0,T]}) \big)\\
&\leq \cos(\alpha)4c + \cos(\alpha)\ell_{g_0}(\gamma|_{[-T,T]}).
\end{align*}
By combining these estimates,
\begin{align*}
\ell_{g_0}(\gamma|_{[-T,T]}) - \ell_{g_0}(\zeta_T)
& \geq
\ell_{g_0}(\gamma|_{[-T,T]}) - d_{g_0}(Tv_+,-Tv_-) - 2c \\
& \geq (1-\cos(\alpha)) \ell_{g_0}(\gamma|_{[-T,T]}) -4c\cos(\alpha)-2c.
\end{align*}
Since $1-\cos(\alpha)>0$ and $\ell_{g_0}(\gamma|_{[-T,T]})\to\infty$ as $T\to\infty$, we infer that $\ell_{g_0}(\gamma|_{[-T,T]}) - \ell_{g_0}(\zeta_T)\to\infty$ as $T\to\infty$. 

\vspace{5pt}

We now have all the ingredients to conclude the proof of the Lemma as follows. The $g$-length of the curve $\zeta_T$ can be bounded from above as
\begin{align*}
\ell_g(\zeta_T) & \leq \ell_{g_0}(\zeta_T) + C &\mbox{(by Step 2)}\\
& = \ell_{g_0}(\gamma|_{[-T,T]}) + \big(\ell_{g_0}(\zeta_T)- \ell_{g_0}(\gamma|_{[-T,T]})\big) + C
\\
& \leq \ell_{g}(\gamma|_{[-T,T]}) + \big(\ell_{g_0}(\zeta_T)- \ell_{g_0}(\gamma|_{[-T,T]})\big) + C & \mbox{(by Step 2)}
\end{align*}
By Step 3, $\ell_{g_0}(\zeta_T)- \ell_{g_0}(\gamma|_{[-T,T]})\to-\infty$ as $T\to\infty$. Therefore, for all $T>0$ large enough, we have that $\ell_g(\zeta_T)<\ell_{g}(\gamma|_{[-T,T]})$. This is impossible, since $\gamma$ is a geodesic, and on any complete Riemannian manifold without conjugate points a geodesic segment is the (unique) shortest curve joining its endpoints.
\end{proof}

We define the compact subsets
\begin{align}
\label{B_R}
B_R:=\big\{x\in \R^n\ \big|\ \rho(x)\geq 1/R\big\},
\end{align} 
which are strictly convex with respect to both $g$ and $g_0$ for all $R>0$ large enough.

\begin{lemma}
\label{asymp of dist function}
There exists an $R_0>0$ and $C>0$ such that for all $R\geq R_0$, 
\[|d_g(p,q) - d_{g_0}(p,q)| \leq CR^{-m+1}\]
for all 
$p,q\in \rr^n$ such that the Euclidean line segment $[p,q]$ does not intersect $B_R$.
\end{lemma}
\begin{proof}
Let $\gamma(t)$ be the complete unit-speed $g$-geodesic joining $p$ and $q$, with boundary data $(y_-,\eta)\in \pl_-S^*\R^n$ as $t\to -\infty$ (obtained using Lemma \ref{behavgeod}). 
We then have that
\begin{equation}
\label{d_g(p,q) < length[p,q]}
d_g(p,q) \leq \ell_g([p,q]) = \int_{t_0}^{t_0+d_{g_0}(p,q)} \sqrt{g(\dot\gamma_0(t),\dot \gamma_0(t))}dt
\end{equation}
where $\gamma_0(t)$ is the complete Euclidean unit-speed geodesic containing the points $p$ and $q$ and $[p,q]$ the Euclidean segment joining $p,q$.  
Using the coordinates \eqref{projcoord}, we can write 
\[\dot \gamma_0(t) = \bbar{\xi}_0(t) \rho^2(t) \partial_\rho + \rho^2(t)\eta(t)^\sharp\]
with $\eta(t)^\sharp$ the dual of the covector $\eta(t)\in T^*S^{n-1}$ with respect to $h_0$, satisfying the condition
$\bbar{\xi}_0(t)^2 + \rho(t)^2 |\eta(t)|_{h_0}^2 = 1$. Due to the assumption on the Riemannian metric $g$ we have that
\begin{equation} \label{g of dot gamma_e}
g(\dot\gamma_0(t), \dot \gamma_0(t)) = 1 + \rho^{2+m}(t) h_{\rho(t)}(\eta(t)^\sharp, \eta(t)^\sharp).
\end{equation}
Using \eqref{g of dot gamma_e} we can expand
\[\sqrt{g(\dot\gamma_0(t), \dot \gamma_0(t))} = 1 + \frac{1}{2}\rho^{2+m}(t) h_{\rho(t)} (\eta(t)^\sharp,\eta(t)^\sharp) + \mc{O}(\rho^{4+2m}(t)|\eta(t)|_{h_0}^4)\]
Substituting this into \eqref{d_g(p,q) < length[p,q]} there is $C>0$ depending only on $g$ such that
\begin{equation}
\label{dg-de}
d_g(p,q) -   d_{g_0}(p,q) \leq C\int_{t_0 }^{t_0 + d_{g_0}(p,q)} 
\rho(t)^{2+m}|\eta(t)|_{h_0}^2 dt.
\end{equation}
Suppose the complete geodesic $\gamma_0$ intersects $B_R$, we choose a parametrization of $\gamma_0$ such that $\gamma_0(0)\in \partial B_R$ and $[p,q] \subset \{t\geq 0\}$ (then $\rho(\gamma(t))<1/R$ for all $t>0$). For $t>0$, 
since $\rho(t)|\eta(t)|_{h_0}\leq 1$ we have $|\eta(0)|_{h_0} \leq R$, and 
by Lemma \ref{lemmeGLT}, we deduce that there is $C>0$ such that $|\eta(t)|_{h_0} \leq CR$ for all $t \geq 0$. 
By \eqref{firstboundrho}, $\rho(t) \leq \frac{C}{R + t}$ for all $t\geq 0$ for some $C>0$ depending only on $g$. Inserting these into \eqref{dg-de}, there is $C>0$ such that
\begin{equation}\label{diffofdist}
d_g(p,q) -   d_{g_0}(p,q) \leq C R^2\int_{0 }^{\infty} \left (\frac{1}{R + t} \right)^{2+m}dt \leq CR^{-m+1}.
\end{equation}
Suppose that the complete geodesic $\gamma_0$ never intersect $B_R$. Then we can find a unit-speed parametrization for $[p,q]$ such that $\gamma_0(0) \in [p,q]$ with $\rho(\gamma_0(0)) = \epsilon \leq R^{-1}$ and $|\eta(0)| = \epsilon^{-1}$. By Lemma \ref{lemmeGLT}, there is $C>0$ depending only on $g$ such that
\[\rho(\gamma_0(t)) \leq \frac{C\epsilon}{1 + \epsilon |t|},\ \ |\eta(\gamma_0(t))| \leq C\epsilon^{-1}.\]
Inserting these estimates into \eqref{dg-de} we obtain again the estimate \eqref{diffofdist}.
Combining these two cases shows $d_g(p,q) - d_{g_0}(p,q) \leq CR^{-m+1}$ with $C$ uniform amongst points $p$, $q$ such that the segment $[p,q]$ does not intersect $B_R$. Reversing the role of $g$ and $g_0$ in the above argument shows the estimate $d_{g_0}(p,q)-d_g(p,q)\leq CR^{-m+1}$ for $R>0$ sufficiently large and the proof is complete.
\end{proof}

\begin{lemma}
\label{dg0leqdg}
There are constants $R_0>0$ and $C>0$ such that
\begin{align*}
d_{g_0}(x,y) \leq d_g(x,y) + CR^{-m+1},
\qquad\forall R\geq R_0,\  x,y\not\in B_R.
\end{align*}
\end{lemma}

\begin{proof}
Let us consider $R\geq R_0$, two distinct points $x,y\not\in B_R$, and an auxiliary vector $w\in\R^{n}$ such that $g_0(w,y-x)=0$ and whose norm $|w|_{g_0}$ is large enough so that the Euclidean geodesic $\zeta(t)=w+(y-x)t$ will not intersect $B_R$ for any $t\in\R$. We are going to use the elementary fact that 
\begin{align*}
\lim_{a\to+\infty} \Big( \sqrt{a^2+b^2}-a\Big) = 0.
\end{align*}
This, together with Lemma~\ref{asymp of dist function}, implies that, for some constants $R_0>0$ and $C>0$, and for all $R\geq R_0$, we have
\begin{align*}
d_{g_0}(x,y) & = \lim_{t\to\infty}  \Big( d_{g_0}(\zeta(-t),\zeta(t))-d_{g_0}(\zeta(-t),x)-d_{g_0}(y,\zeta(t)) \Big)\\
& \leq
\lim_{t\to\infty}  \Big( d_{g}(\zeta(-t),\zeta(t))-d_{g}(\zeta(-t),x)-d_{g}(y,\zeta(t))\Big) + CR^{-m+1}\\
& \leq d_g(x,y) + CR^{-m+1}.
\end{align*}
Here we used that $[\zeta(-t),x]\cap B_R=\varnothing$ and 
$[\zeta(t),y]\cap B_R=\varnothing$ for all $t>0$ large enough.
\end{proof}

\begin{lemma}
\label{l:dist_from_asymptotic_Euclidean_geodesics}
There exists a constant $C>0$ such that, for all $R>0$ large enough and for all unit-speed $g$-geodesics $\gamma:(-\infty,0]\to \R^n\setminus B_R$ asymptotic as $t\to-\infty$ to the Euclidean geodesic $\gamma_0^-$ $($as in~\eqref{limitinggeod}$)$, we have
\begin{align*}
 \sup_{t\in(-\infty,0]} \Big( d_{g_0}(\gamma(t),\gamma_0^-(t)) + |\dot\gamma(t)-\dot\gamma_0^-(t)|_{g_0} \Big)\leq CR^{-m+1}.
\end{align*}
\end{lemma}

\begin{proof}
Consider the two geodesics $\gamma$ and $\gamma_0^-$ as in the statement, they satisfy the equations
\begin{equation}\label{Gammagamma}
\ddot\gamma(t)=\sum_{i,j=1}^n \Gamma_{ij}(\gamma(t))\dot\gamma_i(t)\dot\gamma_j(t)=:\Gamma_{\gamma(t)}(\dot\gamma(t),\dot\gamma(t)),
\qquad
\ddot\gamma_0^-(t)=0.
\end{equation}
Since $g$ is asymptotically Euclidean, there exists a constant $k\geq 1$ such that
\[k^{-1}|\cdot|_{g_0}\leq |\cdot|_g\leq k|\cdot|_{g_0}.\] Since the geodesic ray $\gamma$ does not intersect $B_R$, Lemmas~\ref{curvdecay} and~\ref{lemmeGLT} imply that
\begin{align*}
|\Gamma_{\gamma(t)}(\dot\gamma(t),\dot\gamma(t))|_{g_0} \leq \|\Gamma_{\gamma(t)}\|_{g_0} |\dot\gamma(t)|_{g_0}^2 \leq k^2 \rho(\gamma(t))^{m+1}
\leq
\frac{C}{(R+t)^{m+1}},
\qquad
\forall t\in(-\infty,0]
.
\end{align*}
for some uniform $C>0$.
Since $d_{g_0}(\gamma(t),\gamma_0^-(t))\to0$ and $\dot\gamma(t)-\dot\gamma_0^-(t)\to0$ as $t\to-\infty$, by integrating the geodesic equations we have that, for all $t\leq0$,
\[\begin{gathered}
|\dot\gamma(t)-\dot\gamma_0^-(t)|_{g_0}
=
\bigg| \int_{-\infty}^t
\Gamma_{\gamma(s)}(\dot\gamma(s),\dot\gamma(s))\,ds
\bigg|_{g_0} \leq
C \int_{-\infty}^t (R+s)^{-m-1} ds
\leq
C (R+t)^{-m},\\
|\gamma(t)-\gamma_0^-(t)|_{g_0}
\leq
\int_{-\infty}^t
|\dot\gamma(s)-\dot\gamma_0^-(s)|_{g_0}ds
\leq
C (R+t)^{-m+1}
\end{gathered}\]
and the proof is complete.
\end{proof}

Proposition~\ref{samescat} is a direct consequence of the following lemma.
\begin{lemma}
\label{same_scattering}
Let $g$ be an asymptotically Euclidean metric $g$ to order $m>2$ on $\R^n$ with no conjugate points.
For any unit-speed $g$-geodesic $\gamma:\R\to\R^n$ there exist unique $x_0,v_0\in\R^n$ with $|v_0|_{g_0}=1$, which define the Euclidean geodesic $\gamma_0(t)=x_0+tv_0$, such that
\begin{align*}
\lim_{t\to\pm\infty} d_{g_0}(\gamma(t),\gamma_0(t))=0,
\qquad
\lim_{t\to\pm\infty} \dot\gamma(t)=v_0.
\end{align*} 
\end{lemma}
\begin{proof}
That $x_0, v_0$ are uniquely determined is clear so we only prove existence. Let $\gamma$ be a complete unit-speed 
$g$-geodesic, and consider its two asymptotic Euclidean geodesics $\gamma_0^\pm(t)=x_\pm+tv_\pm$. According to Lemma~\ref{l:asymptotic_slope}, we have $v_0:=v_-=v_+$. Namely, there exists $a\in\R^n$ with $g_0(a,v_-)=0$ and $t_0\in\R$ such that
\begin{align*}
\gamma_0^+(t)=\gamma_0^-(t+t_0)+a.
\end{align*}
Let $R>0$ be large enough so that $\partial B_R$ intersects $\gamma$ in two points $x_R=\gamma(s_R)$ and $y_R=\gamma(t_R)$, for some $s_R<t_R$. 
We fix a vector $w\in\R^{n}$ such that $a\in\mathrm{span}\{w\}$ and $|w|_{g_0}$ is large enough so that 
\[
\big\{\gamma_0^-(t)+w\ \big|\ t\in\R\big\}\cap B_R=\varnothing.
\] 
We set $z:=\gamma_0^-(s_R)+w$. For $t>0$ large enough, both the Euclidean geodesic joining $\gamma(-t)$ and $z$ and the one joining $z$ and $\gamma(t)$ will not intersect $B_R$. Therefore, by Lemmas~\ref{asymp of dist function} and~\ref{l:dist_from_asymptotic_Euclidean_geodesics}, if $R>0$ is large enough we have
\[\begin{split}
d_g(x_R,y_R)
=\, & 
d_g(\gamma(-t),\gamma(t))-d_{g}(\gamma(-t),x_R)-d_g(y_R,\gamma(t))\\
\leq\, & 
d_g(\gamma(-t),z)+d_g(z,\gamma(t))-d_{g}(\gamma(-t),x_R)-d_g(y_R,\gamma(t))\\
\leq\, & 
d_{g_0}(\gamma(-t),z)+d_{g_0}(z,\gamma(t))-d_{g_0}(\gamma(-t),x_R) -d_{g_0}(y_R,\gamma(t)) +\mc{O}(R^{-m+1})\\
 \leq\, & 
d_{g_0}(\gamma_0^-(-t),z)+d_{g_0}(z,\gamma_0^+(t))-d_{g_0}(\gamma_0^-(-t),\gamma_0^-(s_R))-d_{g_0}(\gamma_0^+(t_R),\gamma_0^+(t))\\
 &
 +2d_{g_0}(\gamma_0^-(-t),\gamma(-t)) +2d_{g_0}(\gamma_0^+(t),\gamma(t))+\mc{O}(R^{-m+1})
\end{split}\]
As $t\to\infty$, the sum in the last two lines converges to 
\[(d_{g_0}(\gamma_0^+(t_R),\gamma_0^-(s_R))^2-|a|_{g_0}^2)^{1/2}+\mc{O}(R^{-m+1}).\] 
Summing up and using Lemma \ref{l:dist_from_asymptotic_Euclidean_geodesics} we get
\[\begin{split}
d_{g}(x_R,y_R) &\leq (d_{g_0}(\gamma_0^+(t_R),\gamma_0^-(s_R))^2-|a|_{g_0}^2)^{1/2}+\mc{O}(R^{-m+1})\\
&
\leq 
d_{g_0}(x_R,y_R) \Big(1-\frac{|a|_{g_0}^2}{2d_{g_0}(\gamma_0^+(t_R),\gamma_0^-(s_R))^2} +\mc{O}(\frac{1}{d_{g_0}\big(\gamma_0^+(t_R),\gamma_0^-(s_R))^4}\big)\Big)\\
& \quad+\mc{O}(R^{-m+1}).
\end{split}\]
By choosing $R>0$ large enough so that $d_{g_0}(x_R,y_R)\geq1$, this inequality and Lemma~\ref{dg0leqdg} imply that
\[|a|^2_{g_0}\leq C(d_{g_0}(x_R,y_R)^{-2}+R^{-m+2}).\] 
Since $m>2$ and $d_{g_0}(x_R,y_R)\to \infty$ as $R\to \infty$, we conclude $a=0$, and thus
\begin{equation}
\label{e:equivalent_distances}
d_{g}(x_R,y_R)\leq d_{g_0}(x_R,y_R)+\mc{O}(R^{-m+1}).
\end{equation}
Lemma~\ref{l:asymptotic_slope} implies that $|v_0|_{g_0}=1$. This, together with~\eqref{e:equivalent_distances} and Lemma~\ref{l:dist_from_asymptotic_Euclidean_geodesics}, implies that
\begin{align*}
t_R-s_R 
& = d_g(x_R,y_R)
\leq  d_{g_0}(x_R,y_R) + \mc{O}(R^{-m+1})
\leq d_{g_0}(\gamma_0^+(t_R),\gamma_0^-(s_R)) + \mc{O}(R^{-m+1})\\
& = d_{g_0}(\gamma_0^-(t_0 + t_R),\gamma_0^-(s_R)) + \mc{O}(R^{-m+1}) = t_R-s_R + t_0 + \mc{O}(R^{-m+1}).
\end{align*}
By taking the limit for $R\to\infty$, we conclude that $t_0=0$. This proves that $\gamma$ is asymptotic to the same Euclidean geodesic $\gamma_0(t):=\gamma_0^-(t)=\gamma_0^+(t)$ for both $t\to\infty$ and $t\to-\infty$. By Corollary~\ref{limitgeod}, this also implies that $\dot\gamma(t)\to v_0$ for $t\to\pm\infty$.
\end{proof}

\begin{lemma}
\label{dg0=dg_at_infinity}
There are constants $R_0>0$ and $C>0$ such that
\begin{align*}
d_{g}(x,y) \leq d_{g_0}(x,y) + CR^{-m+1},
\qquad\forall R\geq R_0,\  x,y\not\in B_R.
\end{align*}
\end{lemma}

\begin{proof}
We consider $R>0$ large enough so that the assertion of Lemma~\ref{l:dist_from_asymptotic_Euclidean_geodesics} holds.
For each pair of distinct points $x,y\not\in B_R$, we consider the unique complete unit-speed $g$-geodesic $\gamma$ such that $\gamma(0)=x$ and $\gamma(\tau)=y$ for $\tau:=d_g(x,y)$. By Lemma~\ref{same_scattering}, there exists an Euclidean geodesic $\gamma_0(t):=x_0+t v_0$ that is asymptotic to $\gamma$ at $t\to\pm\infty$. As in the proof of Lemma~\ref{same_scattering}, we fix $w\in\R^n$ such that the Euclidean geodesic $w+\gamma_0$ does not intersect $B_R$, and set $z:=w+\gamma_0(0)$. By Lemmas~\ref{asymp of dist function} and~\ref{l:dist_from_asymptotic_Euclidean_geodesics}, for all $t>0$ large enough we have
\begin{align*}
d_g(x,y) & = d_g(\gamma(-t),\gamma(t)) - d_g(\gamma(-t),\gamma(0)) - d_g(\gamma(\tau),\gamma(t))\\
& \leq d_g(\gamma(-t),z)+ d_g(z,\gamma(t)) - d_g(\gamma(-t),\gamma(0)) - d_g(\gamma(\tau),\gamma(t))\\
& \leq d_{g_0}(\gamma_0(-t),z)+ d_{g_0}(z,\gamma_0(t)) - d_{g_0}(\gamma_0(-t),\gamma_0(0)) - d_{g_0}(\gamma_0(\tau),\gamma_0(t)) + CR^{-m+1}.
\end{align*}
As $t\to\infty$, the sum in the latter line converges to 
\[\tau +CR^{-m+1}=d_{g_0}(\gamma_0(0),\gamma_0(\tau)) +CR^{-m+1},\] 
and thus
\begin{align*}
& d_g(x,y) \leq  d_{g_0}(\gamma_0(0),\gamma_0(\tau)) +CR^{-m+1} \leq d_{g_0}(x,y) + 2C R^{-m+1}.
\qedhere
\end{align*}
\end{proof}

\section{Comparison of asymptotic volumes}

Let $g$ be a Riemannian metric on $\R^n$ that is asymptotically Euclidean to order $m>2$ and  without conjugate points. As before, we denote by $g_0$ the Euclidean metric on $\R^n$, and by $\varphi_t^{g}$ and $\varphi_t^{g_0}$ the geodesic flows of $g$ and $g_0$ respectively on the whole tangent bundle $T\R^n=\R^n\times\R^n$.  By Lemma~\ref{same_scattering}, for every unit-speed $g$-geodesic $\gamma:\R\to\R^n$ there exists a unit-speed Euclidean geodesic $\gamma_0(t)=x_0+tv_0$ such that $| \gamma(t) - \gamma_0(t)|_{g_0} \to 0 $ and $|\dot \gamma(t) - \dot \gamma_0(t)|_{g_0} \to 0$ as $t\to \pm\infty$. 

We extend the map $\theta : S^0\R^n \to S\R^n$ defined in Corollary \ref{limitgeod} to $T\R^n\setminus\{0\mbox{-section}\}$ as follows:
\begin{align*}
\theta:T\R^n\setminus\{0\mbox{-section}\}\to T\R^n\setminus\{0\mbox{-section}\},
\qquad
\theta(x_0,v_0) = (x,v),
\end{align*}
where $x_0=\gamma_0(0)$, $v_0=\lambda\dot\gamma_0(0)$, $x=\gamma(0)$, $v=\lambda\dot\gamma(0)$, $\lambda>0$, and $\gamma$ is a unit-speed $g$-geodesic asymptotic to the Euclidean one $\gamma_0$ as above. {In Corollary~\ref{limitgeod}, we showed that it is a diffeomorphism. }
By its very definition, $\theta$ conjugates the geodesic flows $\varphi_t^{g}$ and $\varphi_t^{g_0}$ as
\begin{eqnarray}
\label{flow conjugation}
\theta^{-1}\circ\varphi_t^{g} =  \varphi_t^{g_0} \circ \theta^{-1}.
\end{eqnarray}
Moreover, since the corresponding geodesics $\gamma$ and $\gamma_0$ as above are asymptotic for $t\to\pm\infty$, we have that $\theta(x_0,v_0)=(x,v)$ if and only if $\theta(x_0,-v_0)=(x,-v)$. We now show that $\theta$ converges to the identity in the $C^1$-topology at infinity. 

\begin{lemma}
\label{l:theta}
For each compact subset $K\subset\R^n\setminus\{0\}$ there exists $C_K>0$ such that, for all $R>0$ large enough, we have
\begin{align*}
\| \theta - \Id \|_{C^1((\R^n\setminus B_R)\times K)} \leq C_K R^{-m+1}.
\end{align*}
Here, the $C^1$ norm is computed with respect to the Euclidean metric $G_0$ on $\R^n\times K\subset T\R^n$.
\end{lemma}

\begin{proof}
Notice that, if $\theta(x_0,v_0)=(x,v)$, then $\theta(x_0,\lambda v_0)=(x,\lambda v)$ for all $\lambda\in\R$. This, together with the compactness of $K$ and Lemma~\ref{l:dist_from_asymptotic_Euclidean_geodesics}, implies that $\| \theta - \Id \|_{C^0((\R^n\setminus B_R)\times K)} \leq C_K R^{-m+1}$. 

It is proved in~\cite[Lemma 3.6]{Guillarmou:2019aa} that there exists a constant $C>0$ such that, for each Jacobi field $J$ over a geodesic $\gamma$ for the metric $g$ and for all $t\in\R$,
\begin{equation}
\label{e:growth_Jacobi_fields}
\begin{split}
|\dot J(t)|_{g_0} & \leq C|(J(0),\dot{J}(0))|_{G_0},\\
|J(t)|_{g_0} & \leq C(|J(0)|_{g_0} + |(J(0),\dot{J}(0))|_{G_0}|t|).
\end{split}\end{equation}
We consider $R>0$ large enough so that $B_R$ is convex for $g$ and $\theta(\R^n\setminus B_R)\cap B_{R/2}=\varnothing$. We consider arbitrary $(x_0,v_0)\in(\R^n\setminus B_R)\times K$ and $w_0:=(x_0',v_0')\in \R^n\times\R^n\simeq T_{(x_0,v_0)}(T\R^n)$. Since $K$ is compact, $|v_0|_{g_0}$ is uniformly bounded when $v_0\in K$. We need to show 
\begin{align*}
|d\theta_{(x_0,v_0)}w_0-w_0|_{G_0}\leq C_K R^{-m+1} |w_0|_{G_0},
\end{align*}
for some uniform constant $C_K\geq 0$. We set $\gamma_{0,s}(t):=x_0+sx_0'+t(v_0+sv_0')$ and $(\gamma_s(t),  \dot \gamma_s(t)):=\theta(\gamma_{0,s}(t), \dot \gamma_{0,s}(t))$. To simplify indices we also write $\gamma : = \gamma_s|_{s= 0}$. Up to replacing $v_0$ with $-v_0$, we can assume that $(\gamma_s(t))\not\in B_{R/2}$ for all $t\leq0$. We set 
\[w=(x',v'):=d\theta_{(x_0,v_0)}(x_0',v_0') = \pl_s[\theta(x_0 + s x_0' , v_0 + sv_0')]|_{s = 0}  = \pl_s(\gamma_s(0), \dot \gamma_s(0))|_{s=0}\]
where $  (\gamma_s(0), \dot \gamma_s(0))$.
Notice that $J_0(t):=\partial_s\gamma_{0,s}(t)|_{s=0}$ and $J(t):=\partial_s\gamma_{s}(t)|_{s=0}$ are Jacobi fields for $g_0$ and $g$ respectively such that $(J_0(0),\dot{J}_0(0))=(x'_0,v'_0)$ and $(J(0),\dot{J}(0))=(x',v')$. The curves $\gamma_{s}(t)$ and $\gamma_{0,s}(t)$ are asymptotic for $t\to-\infty$ together with their first derivatives. 
Using \eqref{Gammagamma} and the notation used there, we can write
\begin{align}
\label{e:difference_geodesics}
\dot\gamma_s(t)-\dot\gamma_{0,s}(t) =  \int_{-\infty}^t \Gamma_{\gamma_s(r)}(\dot\gamma_s(r),\dot\gamma_s(r))dr.
\end{align}
We have $\dot{J}(t)=(\nabla_{\dot \gamma}J)(t)= (\nabla_{J} \dot \gamma)(t)= 
\pl_s \dot{\gamma}_s(t)|_{s=0} + \Gamma_{\gamma(t)}(\dot{\gamma}(t),J(t))$.
Combining this expression with \eqref{e:growth_Jacobi_fields}, Lemma \ref{curvdecay}, and Corollary~\ref{rholeq1/R}, we see that
\begin{equation}
\label{partial_s dot gamma}
|\pl_s \dot \gamma_s(t)|_{s = 0}| \leq C|(J(0),\dot{J}(0))|_{G_0}
\end{equation}
By taking the derivatives with respect to the parameter $s$ in~\eqref{e:difference_geodesics}, and using~\eqref{partial_s dot gamma}, Lemma~\ref{curvdecay}, and Corollary~\ref{rholeq1/R}, 
we obtain the estimate 
\begin{align*}
|\partial_s (\dot\gamma_s(t)-\dot\gamma_{0,s}(t))|_{s = 0}|_{g_0} 
& \leq 
\int_{-\infty}^t
\Big|
d \Gamma_{\gamma(r)}(J(r))(\dot\gamma(r),\dot\gamma(r))
+
2\Gamma_{\gamma(r)}(\pl_{s}\dot{\gamma}_s(r)|_{s=0},\dot\gamma(r))
\Big|
dr\\
& \leq
C|w|_{G_0}\int_{-\infty}^t
\Big(
(R+|r|)^{-m-2} (1+r)
+
(R+|r|)^{-m-1}
\Big)
dr \\
& \leq
C |w|_{G_0}
\int_{-\infty}^t
(R+|r|)^{-m-1}
dr
\end{align*}
for all $t \leq 0$. Integrating once more,
\[
| \pl_{s}(\gamma_s(0) - \gamma_{0,s}(0))|_{s=0}|_{g_0} 
\leq \int_{-\infty}^0
 |\pl_s (\dot\gamma_s(t)-\dot\gamma_{0,s}(t))|_{s=0}|_{g_0}
 dt
\leq C R^{-m+1} |w|_{G_0}.
\]
Therefore
\[
|d\theta^{-1}_{(x,v)}w-w|_{G_0} \leq C R^{-m+1} |w|_{G_0}.
\]
This implies in particular that, for some $c_0>0$, we have
\begin{align*}
|d\theta^{-1}_{(x,v)}w|_{G_0}\geq c_0 |w|_{G_0},
\qquad\forall (x,v)\in\theta(\R^n\setminus B_R\times K),\ w\in T_{(x,v)}T\R^n,
\end{align*}
and therefore
\[\begin{split}
|d\theta_{(x_0,v_0)}w_0-w_0|_{G_0} \leq & C R^{-m+1} |d\theta_{(x_0,v_0)}w_0|_{G_0}
\leq C R^{-m+1} \|d\theta_{(x_0,v_0)}\|_{G_0}|w_0|_{G_0} \\
\leq & C R^{-m+1}c_0^{-1} |w_0|_{G_0},
\end{split}\]
concluding the proof.
\end{proof}
We now construct geodesic cylinders for the metric $g$ and $g_0$. We choose Cartesian coordinates $\{x_1,\dots, x_n\}$ and $H:=\{x_n=0\}$.
By Lemma \ref{behavgeod} there exist $y_\pm\in \partial \R^n$ such that every Euclidean unit geodesic $\gamma^0(t)$ satisfying $\gamma^0(0)\in H$ and $\dot{\gamma}^0(t)=\pl_{x_n}$  will have limit 
\begin{equation}\label{limgamma0}
\lim_{t\to -\infty}(\gamma^0(t),{\dot{\gamma}^0(t)}^\flat)=(y_-,\tfrac{d\rho}{\rho^2}+ \eta)\in \pl_-S^*\R^n\end{equation}
for some $\eta\in T_{y_-}^*\mathbb{S}^{n-1}$. We shall now denote 
$\gamma_\eta^0(t)$ and 
for the Euclidean geodesic satisfying $\gamma_\eta^0(0)\in H$ and \eqref{limgamma0}.
Now define 
\begin{equation}
\label{FR0 and F_R}
\begin{gathered}
F_R^0 := \big\{ \gamma_\eta^0(t)\in \R^n\ \big|\ |t|<R, |\eta|< R \big\},\quad
\ F_R := \pi_0 \circ\theta\left( \{ (x,\partial_{x_n})\in T\R^n\ |\  x\in F_R^0 \}\right).
\end{gathered}\end{equation}

\begin{proposition}\label{PropvolgFR}
Let $g$ be a Riemannian metric on $\R^n$ that is asymptotically Euclidean to order $m>2$ and  without conjugate points.
Then the following estimate holds as $R\to\infty$
\[|\Vol_g(F_R) - \Vol_{g_0}(F_R^0)|\leq\mc{O}(R^{n-m+1}).\] 
\end{proposition}
\begin{proof}
In order to shorten the notation, we set $\|\cdot\|_{\infty,R}:=\|\cdot\|_{L^\infty(\partial SF_R)}$. When this norm will be applied to 1-forms, it will be the one corresponding to the Euclidean metric $g_0$. We consider the Liouville forms $\lambda_0$ and $\lambda$ on $T\R^n$ associated with $g_0$ and $g$, i.e.
\begin{align*}
\lambda_0(w)=g_0(v,d\pi(x,v)w),\qquad
\lambda(w)=g(v,d\pi(x,v)w),\qquad
\forall w\in T_{(x,v)}T\R^n.
\end{align*}
Their restrictions $\alpha_0=\lambda_0|_{S^0\R^n}$ and $\alpha=\lambda|_{S\R^n}$ are the contact forms on the unit tangent bundles associated with $g_0$ and $g$ respectively. It readily follows from their definition and from the fact that $g$ is asymptotically Euclidean to the order $m$ that
\begin{align}
\label{e:lambda-lambda_0}
\|(\lambda-\lambda_0)|_{S^0\R^n}\|_{\infty,R}
+
\|(\lambda-\lambda_0)|_{S\R^n}\|_{\infty,R}
\leq 
\mc{O}(R^{-m})\quad \mbox{as }R\to\infty.
\end{align}
We set $\alpha_1:=\theta^*\alpha=\theta^*\lambda|_{S^0\R^n}$, which is a contact form on $S^0\R^n$. The bound in~\eqref{e:lambda-lambda_0}, together with Lemma~\ref{l:theta}, implies 
\begin{align}
\label{e:alpha1-alpha0}
\|\alpha_1-\alpha_0\|_{\infty,R}
\leq
\|\theta^*\lambda-\lambda\|_{\infty,R} + \|\lambda-\lambda_0\|_{\infty,R}
\leq
\mc{O}(R^{-m+1})\quad\mbox{as }R\to\infty.
\end{align}
Moreover 
\begin{equation}
\label{e:det_dtheta}
\begin{split}
\Vol_{G_0}(\theta^{-1}(\partial SF_R))
& \leq \Vol_{G_0}(\partial SF_R) +
\|\det d\theta^{-1}-1\|_{\infty,R} \Vol_{G_0}(\partial SF_R)\\
& \leq R^{n-1}(1+\mc{O}(R^{-m+1})) \quad\mbox{as }R\to\infty. 
\end{split}
\end{equation}

For all $R>0$ large enough, the map $\kappa:F_R^0\to F_R$, $\kappa(x)=\pi_0\circ\theta(x,\partial_{x_n})$ restricts to a diffeomorphism of a neighborhood of the boundary $\partial F_R^0$ onto a neighborhood of the boundary $\partial F_R$. Lemma~\ref{l:theta} readily implies that $\|\kappa-\Id\|_{C^1(\partial F_R^0)}\leq\mc{O}(R^{-m+1})$. This, together with Stokes Theorem, implies
\begin{equation}
\label{e:volume_estimate_1}
\begin{split}
\big|\Vol_{g_0}(F_R)-\Vol_{g_0}(F_R^0)\big|
& =
\Big|
\int_{F_R} dx_1\wedge...\wedge dx_n - \int_{F_R^0} dx_1\wedge...\wedge dx_n \Big|\\
&=
\Big| \int_{\partial F_R^0} \kappa^*(x_1\,dx_2\wedge...\wedge dx_n) - \int_{\partial F_R^0} x_1\,dx_2\wedge...\wedge dx_n \Big|\\
& \leq
 C \|\kappa-\Id\|_{C^1(\partial F_R^0)} (1+R)\,\Vol_{g_0}(\partial F_R^0)\\
& \leq \mc{O}(R^{-m+1}) (1+R)R^{n-1} \leq \mc{O}(R^{n-m+1})\quad\mbox{as }R\to\infty.
\end{split}
\end{equation}
For each $R>0$ large enough, the map 
$\chi:S^0F_R\to\theta^{-1}(SF_R)$,
$\chi(x,v)=\theta^{-1}(x,\|v\|_{g}^{-1}v)$
is a diffeomorphism. Since $g$ is asymptotically Euclidean to the order $m$, Lemma~\ref{l:theta}  implies that
$\|\chi-\Id\|_{\infty,R}+\|d\chi-\Id\|_{\infty,R}\leq\mc{O}(R^{-m+1})$.
Therefore, by proceeding as in the last estimate,
\begin{align}
\label{e:volume_estimate_2}
\big|\Vol_{G_0}(\theta^{-1}(SF_R))-\Vol_{G_0}(S^0F_R)\big|
\leq \mc{O}(R^{n-m+1})\quad\mbox{as }R\to\infty.
\end{align}
We are now going to apply an argument due to Croke and Kleiner \cite[Lemma~2.1]{Croke:1994aa} in order to compare the volumes $\Vol_{G_0}(\theta^{-1}(SF_R))$ and $\Vol_G(SF_R)$. We set
$\alpha_t:=(1-t)\alpha_0 + t\alpha_1$, where $t\in[0,1]$.
Since $\theta$ conjugates the geodesic flows of $g$ and $g_0$, the geodesic vector field $X_0$ of  $g_0$ is the Reeb vector field of both $\alpha_0$ and $\alpha_1$. Therefore $\alpha_t(X_0)\equiv1$ and $d\alpha_t(X_0,\cdot)\equiv0$ for all $t\in[0,1]$, which readily implies that $(\alpha_1-\alpha_0)\wedge(d\alpha_t)^{n-1}=0$ for all $t\in[0,1]$. This, together with Stokes Theorem, \eqref{e:alpha1-alpha0}, and~\eqref{e:det_dtheta}, allows to estimate as $R\to \infty$
\begin{align*}
\Big|\frac{d}{dt} \int_{\theta^{-1}(SF_R)}\!\!\!\!\!\! \alpha_t\wedge(d\alpha_t)^{n-1}\Big|
= \,&
(n-1)\Big|\int_{\theta^{-1}(SF_R)}\!\!\!\!\!\! \alpha_t\wedge d(\alpha_1-\alpha_0)\wedge(d\alpha_t)^{n-2}\Big|\\
= \,& (n-1) \Big|\int_{\theta^{-1}(\partial SF_R)}\!\!\!\!\!\! \alpha_t\wedge(\alpha_1-\alpha_0)\wedge(d\alpha_t)^{n-2}\Big|\\
\leq \,& C \Vol_{G_0}(\theta^{-1}(\partial SF_R)) \|\alpha_1-\alpha_0\|_{\infty,R}\\
\leq \,& R^{n-1}(1+\mc{O}(R^{-m+1}))\mc{O}(R^{-m+1})
= \mc{O}(R^{n-m})
\end{align*}
We recall that the Riemannian volume forms of $G$ and $G_0$ are respectively 
$\tfrac{1}{(n-1)!}\alpha\wedge d\alpha^{n-1}$ and  $\tfrac{1}{(n-1)!}\alpha_0\wedge d\alpha_0^{n-1}.$ 
Therefore, the latter estimates imply that as $R\to \infty$
\begin{align*}
\big| \Vol_G(SF_R) - \Vol_{G_0}(\theta^{-1}(SF_R)) \big| = 
\frac{1}{(n-1)!} \Big|\int_0^1 \!\!\frac{d}{dt}\! \int_{\theta^{-1}(SF_R)}\!\!\!\!\!\! \alpha_t\wedge(d\alpha_t)^{n-1} dt\Big|
= \mc{O}(R^{n-m}).
\end{align*}
This, together with \eqref{e:volume_estimate_2} and ~\eqref{e:volume_estimate_1} allows to conclude
\begin{align*}
|\Vol_g(F_R) - \Vol_{g_0}(F_R^0)|
&=
\big|\Vol(S^{n-1})^{-1}\Vol_G(SF_R) - \Vol_{g_0}(F_R^0)\big|\\
&
=
\mc{O}(R^{n-m}) + \big|\Vol(S^{n-1})^{-1}\Vol_{G_0}(\theta^{-1}(SF_R)) - \Vol_{g_0}(F_R^0)\big|\\
&
=
\mc{O}(R^{n-m+1})
+
\big|\Vol(S^{n-1})^{-1}\Vol_{G_0}(S^0F_R)-\Vol_{g_0}(F_R^0)\big|
\\
&
=
\mc{O}(R^{n-m+1})
+
\big|\Vol_{g_0}(F_R)-\Vol_{g_0}(F_R^0)\big|=
\mc{O}(R^{n-m+1})
\end{align*}
which ends the proof.
\end{proof}

\section{Rigidity for perturbations of $\R^n$ decaying fast}

In this section, we will show that asymptotically Euclidean Riemannian metrics to large enough order are flat, provided they have no conjugate points.
The proof follows some ideas developed by Croke \cite{Croke:1991aa}, and based on methods of Berger \cite{Besse:1978aa} used for the Blaschke conjecture for spheres.

\subsection{Geodesic normal coordinates}

Let $g$ be an asymptotically Euclidean metric to order $m>2$ on $\R^n$ without conjugate points. 
The goal of this section is to consider a coordinate system on $\R^n$ obtained from 
geodesics $\gamma(t)$ of $g$ that are all converging to a point $p_-\in \pl\bbar{\R^n}$ as $t\to -\infty$: by Proposition \ref{samescat}, they are asymptotically parallel as $t\to \pm\infty$. This can be viewed as some geodesic normal coordinates from the point $p_-$ at infinity.
 
We consider the diffeomorphism $\theta:S^0\R^n\to S\R^n$ introduced in Corollary~\ref{limitgeod}, which conjugates the geodesic flows of the Euclidean metric $g_0$ and $g$. Consider the Euclidean geodesic lines parallel to $\pl_{x_n}$ in $\R^n$. In the compactification they intersect $\pl\bbar{\R^n}=\mathbb{S}^{n-1}$ in two points $p_-,p_+$, where $p_\pm$ correspond to the limit $x_n\to \pm\infty$. By Lemma \ref{behavgeod} (applied to the Euclidean metric $g_0$), these Euclidean geodesics are parametrized by $\eta\in T_{p_-}^*\mathbb{S}^n$. We will thus denote them  by $\gamma^0_\eta$, and 
we will fix their parametrization so that $\gamma^0_\eta(0)$ belongs to the hyperplane $x_n=0$. The curve $\gamma_\eta(t) :=\pi_0(\theta(\gamma_\eta^0(t),\partial_{x_n}))$ is the unit-speed $g$-geodesic with the same limiting conditions as $\gamma_\eta^0(t)$ for $t\to\pm\infty$, see Lemma \ref{same_scattering}.

Another parametrization (not of unit length) of $\gamma_\eta$ is obtained by considering the rescaled flow $\bbar{\varphi}_\tau(p_-,\eta)$ where, following the proof of Corollary \ref{limitgeod}, we have $\gamma_\eta(t)=\pi_0(\bbar{\varphi}_{\tau(t,\eta)}(p_-,\eta))$ with $\tau(t,\eta)$ defined by \eqref{formulataut}. We will denote the rescaled geodesic 
\begin{equation}\label{defbargamma}
\bbar{\gamma}_\eta(\tau)=\pi_0(\bbar{\varphi}_{\tau}(p_-,\eta)).
\end{equation}
We first prove a lemma about the {\em uniform} growth of the Jacobi fields along geodesics emanating from a fixed point $p\in \R^n$.
\begin{lemma}
\label{uniform jacobi growth}
For all $R>0$ there exists a $T_R>0$ such that for any orthogonal Jacobi field $J$ along any unit-speed $g$-geodesic $\gamma$ with $\gamma(0) = p$ satisfying $J(0) = 0$, $\|\dot{J}(0) \| = 1$, we have
$\|J(t)\| \geq R$ for all $t \geq T_R$.
\end{lemma}
\begin{proof}
Let $v,w \in S_p\R^n$ satisfy $g_p(v,w) = 0$ and denote by $J(t;v,w)$ to be the value of the Jacobi field at time $t$ along the geodesic $\gamma(\cdot; v)$ with initial condition $J(0;v,w) = 0$, $\nabla_{\dot\gamma} J(0;v,w) = w$. We set \[D:= \{(v,w) \in S_p\R^n \times S_p\R^n\ |\ g(v,w) = 0\} \cong S^{n-1} \times S^{n-2}.\]
Notice that, for each fixed $t$, the map
\[
D\to\R^2,\qquad (v,w)\mapsto (\|J(t;v,w)\|, \|\dot{J}(t;v,w)\|)
\]
is continuous. 

Now fix $\epsilon >0 $ small. For any $(v_0,w_0) \in D$ there exists a $T_{v_0,w_0}$ such that 
\begin{eqnarray}
\label{jacobi at large enough time}
\rho(\gamma(t;v_0)) \leq \epsilon,\quad  \dot{J}(T_{v_0,w_0}; v_0,w_0) \neq 0
\end{eqnarray}
with $\gamma(t;v_0):=\pi_0(\varphi_{t}(p,v_0))$.
This is guaranteed by \cite[Prop 2.9]{Eberlein:1973aa}. 

By continuity, there is a small neigbourhood $D_0 \subset D$ containing $(v_0,w_0)$ such that for 
$$ \| \dot{J}(T_{v_0,w_0}; v,w)\| \geq c_0>0\quad \forall (v,w) \in D_0.$$
Using compactness there exists $c>0$ and finitely many $(v_j,w_j)\in D$, $T_j >0$, $j = 1,\dots N$ such that 
\begin{eqnarray}
\label{compactness jacobi}
\rho(\gamma(T_j;v_j)) \leq \epsilon,\quad \bigcup\limits_{j = 1}^N D_j = D,\quad  \| \dot{J}(T_j; v,w)\| \geq c >0,\quad  \forall (v,w)\in D_j.
\end{eqnarray}
By finiteness it suffices to prove the uniform growth of Jacobi fields for $(v,w) \in D_j$ for each fixed $j$.
Without loss of generality we may choose $D_j$ in our construction such that $\{v\ |\ (v,w) \in D_j\}\subset S_p\R^n$ is contained in a neighbourhood of $S_p\R^n$ which is topologically equivalent to a ball of dimension $n-1$. Using this construction
$$\{ \gamma(T_{j}; v)\ |\ (v,w) \in D_j\}$$
is contained in a simply connected portion of the geodesic sphere of radius $T_{j}$ which admits a orthonormal frame of tangent vectors $\{X_1,\dots X_{n-1}\}$. Each tangent vector $X_j$ is orthogonal to $\dot\gamma$ by the Gauss Lemma. We may extend these frames along geodesic flows for $t\geq T_{j}$ by parallel transport.
Let $(v,w) \in D_j$ and write, for $t\geq T_{j}$, 
\[J(t; v,w) = \sum\limits_{k = 1}^{n-1}J_k(t;v,w) X_k(t;v).\] 
Writing the Jacobi equation in these components and integrating we have that for $t \geq T_j$
\begin{eqnarray}
\label{integrate jacobi field term by term}
J_k(t;v,w) = J_k(T_j;v,w) + \dot J_k(T_j;v,w) (t- T_j) + \sum_{l=1}^{n-1}\int_{T_j}^t \int_{T_j}^s  J_l(r;v,w) R_l^k(r;v,w) dr ds
\end{eqnarray}
where $R_l^k(t;v,w):= \langle R_{\gamma(t;v)}(\dot \gamma(t;v), X_l(t;v,w)) \dot \gamma(t;v), X_k(t;v,w)\rangle$. By \cite[Lemma 3.5]{Guillarmou:2019aa} for $t \geq T_j$
\[\left |\sum_{l=1}^{n-1}  J_l(t;v,w) R_l^k(t;v,w)\right| \leq C \rho^4(\gamma(t;v)) \sum_{k=1}^{n-1} |J_k(t;v,w)|\]
which, by applying \eqref{e:growth_Jacobi_fields} and Corollary \ref{rholeq1/R} yields
\[\left |\sum_{l=1}^{n-1}  J_l(t;v,w) R_l^k(t;v,w)\right| \leq C \rho^4(\gamma(t;v)) (1+t) \leq Ct^{-4}(1+t)\]
and all constants are uniform amongst $(v,w) \in D_j$. Inserting this estimate into \eqref{integrate jacobi field term by term},
\[ |J_k(t;v,w)| \geq -C + |\dot J_k(T_j;v,w)| (t-T_j)\]
Condition \eqref{compactness jacobi} then allows us to conclude that for all $(v,w) \in D_j$ and $t\geq T_j$
$$\| J(t;v,w)\| \geq -C + c (t-T_j)$$
for constants $C$ and $c$ independent of the choice of $(v,w) \in D_j$.
\end{proof}

We set 
\[{\mathcal D}:= \{ (\tau,\eta)\in (0,\infty)\x T^*_{p_-}\pl\bbar{\R^n}\ |\  \tau\in(0,\tau_+(\eta))\}\] 
where as before $\tau^+(\eta)$ is define by $\rho(\bbar{\gamma}_\eta(\eta))=0$ using the definition \eqref{defbargamma}.

\begin{lemma}\label{foliationbygamma}
The  map  
$\psi:{\mathcal D} \to \R^n$, $\psi (\tau,\eta)=\pi_0(\bbar{\varphi}_\tau(p_-,\eta))$
is a smooth diffeomorphism.
\end{lemma}
\begin{proof}
First, $\psi$ is smooth by definition. We will show that it is a local diffeomorphism which is proper. We first show it is proper. Let $K\subset \R^n$ be compact, then there exists $\eps>0$ so that $K\subset \{\rho\geq \eps\}$. Take a sequence $(\tau_n,\eta_n)\in \mc{D}$ such that $\psi(\tau_n,\eta_n)\to z\in K$. By Corollary \ref{rholeq1/R}, there is $R>0$ such that for all $|\eta|_{h_0}>R$, the geodesic $\gamma_\eta$ satisfies $\rho(\gamma_\eta(t))<\eps$. Thus $|\eta_n|_{h_0}\leq R$ and $\eta_n$ has a converging subsequence to $\eta$. Up to extraction, we can assume that the entire sequence converges. Since $\bbar{\gamma}_{\eta_n}$ stays in a compact set of $\bbar{S^*\R^n}$ and $\tau_n\leq \tau^+(\eta_n)$ is uniformly bounded, $\tau_n$ has a converging subsequence to $\tau$. Assume that $\tau=0$ or $\tau=\tau^+(\eta)$, then $\rho(\psi(\tau_n,\eta_n))\to 0$ but since $\psi^{-1}(K)$ is closed in $\mc{D}$ one gets $\psi(\tau,\eta)\in K$, which is not possible as $K\subset \{\rho\geq \eps\}$. This shows that $\psi^{-1}(K)$ is compact.

 We next show that $\psi$ is a local diffeomorphism: its derivatives are given by 
\begin{equation} \label{derivativespsi}
\begin{gathered}
\pl_\tau \psi(\tau,\eta)= d\pi_0(\bbar{X}(\bbar{\varphi}_\tau(p_-,\eta)))=\pl_\tau \bbar{\gamma}_\eta(\tau)\\
\pl_\eta \psi(\tau,\eta)=d\pi_0.\pl_\eta \bbar{\varphi}_\tau(p_-,\eta)=:\bbar{Q}_\eta(\tau). 
\end{gathered}\end{equation}
We identify $T(T^*_{p_-}\pl\bbar{\R^n})\simeq T^*_{p_-}\pl\bbar{\R^n}$. 
For each $\bbar{w}\in T^*_{p_-}\pl\bbar{\R^n}$, we define $\bbar{J}_{\bbar{w}}(\tau):=\bbar{Q}_\eta(\tau)\bbar{w}$.
Recall from \eqref{formulataut} that we can rewrite $\tau(t,\eta)$ and its inverse $t(\tau,\eta)$ to relate the flow $\varphi_t$ and $\bbar{\varphi}_\tau$: in particular we have
 $\gamma_\eta(t) = \pi_0(\varphi_t(\bbar{\varphi}_{\tau(0,\eta)}(p_-,\eta))$ and  $\bbar\varphi_\tau(p_-,\eta) = \varphi_{t(\tau, \eta)}(\bbar\varphi_{\tau(0,\eta)}(p_-,\eta))$. 
 Thus,
 \[\begin{split} 
d\pi_0.d\bbar{\varphi}_{\tau}(p_-,\eta)\bbar{w} & = d\pi_0.d\varphi_{t(\tau,\eta)}(\bbar{\varphi}_{\tau(0,\eta)}(p_-,\eta))d\bbar{\varphi}_{\tau(0,\eta)}(p_-,\eta)\bbar{w}+d\pi_0(X(\bbar{\varphi}_{\tau}(p_-,\eta)))\pl_\eta t(\tau,\eta)\bbar{w}\\
&= J_{w}(t(\tau,\eta))+\dot{\gamma}_\eta(t(\tau,\eta))\pl_\eta t(\tau,\eta)\bbar{w}
\end{split}\]
where $J_w$ is the Jacobi field with $(J_w(0),\dot{J}_w(0))=w:=d\bbar{\varphi}_{\tau(0,\eta)}(p_-,\eta)\bbar{w}$. Therefore, if $\perp$ denotes the orthogonal projection on $\dot{\gamma}_\eta^\perp$, we obtain 
\[ \bbar{J}_{\bbar{w}}(\tau)^\perp=J_w(t(\tau,\eta))^\perp.\]
To show that $\psi$ is a local diffeomorphism, assume now that $d\psi(\tau_0,\eta_0)W=0$ for some $(\tau_0,\eta_0)$ and $W=w_0\pl_\tau+\bbar{w}$ with $\bbar{w}\in T_{p_-}^*\pl\bbar{\R^n}$. We obtain 
\[ J_w(t(\tau_0,\eta_0))^\perp=0, \quad w_0\pl_\tau \bbar{\gamma}_{\eta_0}(\tau_0)+
\dot{\gamma}(t(\tau_0,\eta_0))\pl_\eta t(\tau_0,\eta_0)\bbar{w}=0. \]
Thus $J_w(t)^\perp$ is an orthogonal Jacobi field along $\gamma_{\eta_0}$ vanishing at $t(\tau_0,\eta_0)$. Now let us show that this field has bounded norm with respect to $g$.
By \eqref{formrhoy}, for each $\eta\in T^*_{p_-}\pl\bbar{\R^n}$ the integral curve $\bbar{\gamma}_\eta(\tau)=\pi_0(\bbar{\varphi}_\tau(p_-,\eta))=(\rho(\tau),y(\tau))$ near $\pl \bbar{\R^n}$ is of the form 
\[
 \rho(\tau)=\tau+\mc{O}(\tau^3),  \quad y(\tau)=p_-+\eta^\sharp \tau+\mc{O}(\tau^2).
 \]
 locally uniformly in $\eta$ where $\eta^\sharp=\sum_{i,j}h_0^{ij}\eta_j \pl_{y_i}$ is the dual to $w$ with respect to $h_0$. Thus, taking $\eta_s=\eta_0+s\bbar{w}$, we obtain (with $\bbar{w}=\sum_{j}\bbar{w}_jdy_j$) as $\tau\to 0$
\[\begin{split}
\bbar{J}_{\bbar{w}}(\tau)=\pl_s\bbar{\gamma}_{\eta_s}(\tau)|_{s=0} & = \mc{O}(\tau^3)\pl_\rho+\sum_{ij}h_0^{ij}\bbar{w}_j \tau\pl_{y_i}+\mc{O}(\tau^2)\pl_y \\
&= \mc{O}(\tau)\rho^2\pl_\rho+\sum_{ij}h_0^{ij}\bbar{w}_j (\rho\pl_{y_i})+\mc{O}(\tau^2)\pl_y 
\end{split}\]
This implies that $\sup_{\tau\in[0,\tau_0]}\|\bbar{J}_{\bbar{w}}(\tau)\|_g<\infty$. Therefore $J_w(t)^\perp$ is bounded in $[-\infty,t(\tau_0,\eta_0)]$. By \cite[Proposition 2.9]{Eberlein:1973aa}, since the curvature is bounded below uniformly and there is no conjugate points we get $J_w(t)^\perp=0$, which means $w=0$ or equivalently $\bbar{w}=0$.  This implies $w_0=0$ and thus $W=0$, showing that $\psi$ is a local diffeomorphism.
Its image is then open, but also it is closed since it is a proper map by the argument above. 

Finally we show that $\psi$ is injective. Assume by contradiction that 
\[\psi_{p_-}(\tau_1,\eta_1)  =\psi_{p_-}(\tau_2,\eta_2) =: p\] 
for some $(\tau_1,\eta_1) \neq (\tau_2,\eta_2)$. By the fact that there are no conjugate points, we infer that $\eta_1\neq \eta_2$ (if they were equal then there would exist a geodesic trajectory with self-intersections). This means that there are two distinct unit-speed $g$-geodesics $\gamma_1,\gamma_2$ originating from $p$ with backward limiting conditions $(p_-,\eta_1)$ and $(p_-,\eta_2)$ respectively. Lemma~\ref{same_scattering} implies that there exist two unit-speed Euclidean geodesics $\gamma^0_j$, $j = 1,2$, such that $d_g(\gamma_j(t),\gamma_j^0(t)) \to 0$ as $t\to \pm\infty$. Since $\gamma_1$ and $\gamma_2$  land at $p_-$ in the backward limit, $\gamma^0_1$ and $\gamma^0_2$ are parallel. This implies that 
\begin{eqnarray}
\label{geodesics don't diverge}
\sup_{t\in\R}d_g(\gamma_1(t),\gamma_2(t)) <\infty.
\end{eqnarray}
Lemma~\ref{uniform jacobi growth} implies that 
\[\lim\limits_{t\to \infty} \inf_J \|J(t)\| = \infty,\]
where the infimum is taken over all Jacobi fields $J$ such that $\pi_0(J(0))=p$, $J(0)=0$, and $\|\nabla_tJ(0)\|_g=1$. This, together with \cite[Proposition 6]{Eschenburg:1976aa}, implies that any two unit-speed $g$-geodesics intersecting transversely at $p$ must diverge asymptotically in positive and negative time, contradicting \eqref{geodesics don't diverge}.
\end{proof}

We will next consider Jacobi fields alongs the geodesics $\gamma_\eta$. It will be convenient to have a global parallel frame asymptotic to the canonical frame $\pl_{x_j}$ of the Euclidean metric $(\R^n,g_0)$ near $p_-\in \pl\bbar{\R^n}$. This is provided by the following lemma.
\begin{lemma}\label{parallelbasis}
For each orthonormal basis $(W_1,\dots, W_{n-1})$ of $T_{p_-}\pl\bbar{\R^n}$ with respect to $h_0$,
there is a smooth orthonormal frame $(Y_1,\dots,Y_n)$ of $T\R^n$ for the Riemannian metric $g$ that is parallel along each $\gamma_\eta$, such that $Y_n=\dot{\gamma}_\eta$, $\psi^*Y_j$ extends smoothly down to $\tau=0$ and $\tau=\tau_+$ and is of the form 
\[
\psi^*Y_j=\left\{\begin{array}{ll}
\displaystyle\sum_{i=1}^{n-1}a_{ji}(\tau,\eta)\tau\pl_{y_i}+a_{jn}(\tau,\eta)\tau^2\pl_\tau, & \textrm{ near }\tau=0, 
\vspace{5pt}\\
\displaystyle\sum_{i=1}^{n-1}a_{ji}(\tau,\eta)(\tau^+(\eta)-\tau)\pl_{y_i}+a_{jn}(\tau,\eta)(\tau^+(\eta)-\tau)^2\pl_\tau, &
\textrm{ near }\tau=\tau_+(\eta),
\end{array}\right.\]
for some $a_{ji}\in C^\infty(\bbar{\mc{D}})$, 
with 
\[\psi^*(\rho^{-1}Y_j)|_{\tau=0}=\sum_{i=1}^{n-1}a_{ji}(0,\eta)\pl_{y_i}= W_{j},\ \  j = 1,\dots, n-1.\]
If we choose $W_j:= (\rho^{-1} \partial_{x_j})|_{T_{p_-} \pl\bbar{\R^n}}$ and denote $Y_j(t,\eta):=Y_j(\gamma_\eta(t))$ then there is $C>0$ such that for $|\eta|$ large enough 
\begin{eqnarray}
\label{almost euclidean frame}
\|Y_j(t,\eta) - \partial_{x_j}\|_g \leq C|\eta|^{-1}_{h_0},\ \ \ t \in(-\infty,\infty).
\end{eqnarray}
\end{lemma}
\begin{proof}
Let us fix an $\eta\in T^*_{p_-}\pl\bbar{\R^n}$ and call $p_+=\lim_{t\to +\infty}\gamma_\eta(\tau)$. We take an orthonormal basis of  $^{\rm sc}T\bbar{\R^n}$ near the (resclaled) geodesic $\bbar{\gamma}_\eta$ of the form 
\[ Z_0=\pm\rho^2\pl_\rho \textrm{ near }p_\mp, \quad Z_j=\rho \bbar{Z}_j, \quad d\rho(Z_j)=0 \textrm{ for j}>0\]
with $\bbar{Z}_j(p_-)=W_j$, and we write vector fields under the form $Y=\sum_jf_jZ_j$.
The equation for being parallel with respect to $\dot{\gamma}_\eta$ is 
\[ 0=\nabla_{\dot{\gamma}_\eta}Y=\sum_{j}\dot{f}_j(t)Z_j+f_j(t)\nabla_{\dot{\gamma}_\eta}Z_j.\]
For each $Z\in C^\infty(\bbar{\R^n};{^{\rm sc}T}\bbar{\R^n})$, Koszul formula gives directly that $\Omega_{ij}(Z):=\cjg \nabla_{Z}Z_i,Z_j\cjd$ are smooth functions on $\bbar{\R^n}$. We use the change of coordinates \eqref{formulataut} and \eqref{formulattau} related $\gamma_\eta(t)$ and $\bbar{\gamma}_\eta(\tau)$: we have $\pl_\tau t(\tau,\eta)=\bbar{\rho}(\tau)^{-2}$ with $\bbar{\rho}(\tau)=\rho(\bbar{\gamma}_\eta(\tau))$. Therefore we obtain 
\[  0=\sum_{j} \left(\bbar{\rho}(\tau)^{2}\pl_\tau \bbar{f}_j(\tau)+\sum_i\bbar{f}_i(\tau)\Omega_{ij}(\dot{\gamma}_\eta)\right)Z_j\]  
for $\bbar{f}_j(\tau):= f_j(t(\tau,\eta))$
and we rewrite $\dot{\gamma}_\eta(t(\tau,\eta))=\bbar{\rho}(\tau)^2\pl_\tau\bbar{\gamma}_\eta(\tau)$ to deduce that the system becomes 
\[\forall j, \quad 0=\pl_\tau \bbar{f}_j(\tau)+\sum_i\bbar{f}_i(\tau)\Omega_{ij}(\pl_\tau\bbar{\gamma}_\eta).\]
This is an ODE with smooth coefficients so we can prescribe the boundary value at $\tau=0$ to be 
$f_j(0)=\delta_{ij}$ to obtain that $Y_i=\sum_i f_iZ_i$ are parallel and satisfying the required properties at $p_-$. Note that this is a frame near $\bbar{\gamma}_\eta$ and it is uniquely defined by the boundary value at $\tau=0$.
Now we need to check they are orthonormal: since they are parallel it suffices to use that the limit $g(Y_i,Y_j)\circ \bbar{\gamma}_\eta(\tau)\to \delta_{ij}$ as $\tau\to 0$. Since $\eta$ is chosen arbitrarily, this defines the frame $(Y_j)_j$ everywhere in $\R^n$ by Lemma \ref{foliationbygamma}. The behaviour at $\tau=\tau^+(\eta)$ is obtained the same way.

We now verify \eqref{almost euclidean frame} for $Y = Y_1$ and the rest follows by the same argument. 
We write $Y(t) = \sum_{j=1}^{n} f_j(t) \partial_{x_j}$, the parallel transport transport equation yields
$0 = \sum_{j=1}^n\dot f_j \partial_{x_j} + f_j \nabla^g_{\dot\gamma} \partial_{x_j}$
which implies
\begin{equation}
\label{parallel ODE}
0 = \dot f_j(t) +\sum_{k=1}^n \Gamma^k_{j}(\dot{\gamma}_\eta(t)) f_k(t),
\end{equation}
where $\Gamma^j_{k}(Z) := dx_j (\nabla_{Z} \partial_{x_k})=g(\nabla_{Z} \partial_{x_k},\pl_{x_j})+\mc{O}(\rho^2)$. 
By Lemma \ref{curvdecay} we have that 
\begin{equation}
\label{kristof decay}
|\Gamma^k_{j}(Z)| \leq C\rho^{m+1}|Z|_g.
\end{equation}
By the fact that $\lim\limits_{t\to -\infty} f_j(t) = \delta_{1j}$ we can write
\[f_j(t) - \delta_{1j} =\int_{-\infty}^t \Gamma^k_{j}(\dot{\gamma}_\eta(s)) f_k(s) ds.\]
The vector field $Y$ is parallel so the coefficients $f_j(t)$ are all bounded. Therefore we may estimate the integral using \eqref{kristof decay} to get
\[|f_j(t) - \delta_{1j} |\leq\int_{-\infty}^t \rho^2(\gamma_\eta(s)) ds.\]
Using  \eqref{uniformboundrho} we conclude that $|f_j(t) - \delta_{1j} 
| \leq C|\eta|_{h_0}^{-1}$ for $|\eta|_{h_0}$ large. Note that 
\eqref{firstboundrho} also prove that for each $\eta$
\begin{equation}\label{Y_j-pl_xj}
\|Y_j(t)-\pl_{x_j}\|_g=\mc{O}(t^{-m}).
\end{equation}
\end{proof}

\subsection{Bounds on stable Jacobi tensors}
Next, we will study the bounded Jacobi fields along the geodesics $\gamma_\eta$, they arise from varying $\gamma_\eta$ with respect to $\eta$. We call them stable Jacobi fields.

We define a family of $(n-1)\x (n-1)$ matrices $\bbar{A}_\eta(\tau)$ by setting
\begin{equation}\label{defbarA}
\bbar{A}_\eta(\tau)(\alpha_1,\dots,\alpha_{n-1})=\Big(
g\Big( \bbar{Q}_\eta(\tau)\sum_{j=1}^{n-1}\alpha_j W_j^*,Y_k\Big)\Big)_{k=1,\dots,n-1}
\end{equation}
where $(W_j^*)_{j=1,\dots,n-1}\in T^*_{p_-}\pl \bbar{\R^n}$ is the dual basis to $(W_j)_{j=1,\dots,n-1}$ and $\bbar{Q}_\eta$  is defined in \eqref{derivativespsi}. We also define the $(n-1)\x (n-1)$ matrix
\begin{equation}\label{defAeta}
A_\eta(t):=\bbar{A}_\eta(\tau(t,\eta))
\end{equation} 
with $\tau(t,\eta)$ the rescaled time defined by \eqref{formulataut} (satisfying $\tau(t,\eta)\to 0$ as $t\to -\infty$.
The matrix $A_\eta(t)$ is a matrix solution of the Jacobi equation for $g$ along $\gamma_{\eta}$: indeed, it satisfies  
\begin{equation} \label{Jacobieq}
\pl_t^2A_\eta(t)+\mc{R}_\eta(t)A_\eta(t)=0
\end{equation}
with $\mc{R}_\eta(t)$ a smooth family of matrices so that (with the notations of the proof of Lemma \ref{foliationbygamma})
\[\mc{R}_\eta(t)(\alpha_1,\dots,\alpha_{n-1}):=
\Big(g\Big(R_{\gamma_\eta(t)}(\sum_{j = 1}^{n-1}\alpha_j Y_j,\dot{\gamma}_\eta(t),\dot{\gamma}_\eta(t)),Y_i\Big)\Big)_{i=1,\dots,n-1}.\]

Moreover, by \eqref{behavbargamma} we have $\bbar{Q}_\eta(\tau)\bbar{w}=\tau \bbar{w}^\sharp+\mc{O}(\tau^2)\pl_y+\mc{O}(\tau^3)\pl_\rho$ with $\bbar{w}^\sharp$ the dual to $\bbar{w}$ with respect to $h_0$, thus using Lemma \ref{parallelbasis} we infer
\begin{equation}\label{limA-infty}
\lim_{t\to-\infty}A_{\eta}(t) =\lim_{\tau\to 0}\bbar{A}_\eta(\tau)={\rm Id}.
\end{equation} 
Recall from Lemma \ref{curvdecay} that for $W_j$ of $g$-norm $1$
\begin{equation}
\label{curvature decays to order m}
|R_{x}(W_1,W_2,W_3,W_4)|_g \leq C\rho(x)^{m+2}.
\end{equation}
Using this estimate, we thus get from \eqref{Jacobieq} and \eqref{firstboundrho}
that $\|\pl_tA_\eta(t)\|\to 0$ as $t\to -\infty$, and there is $C>0$ independent of $\eta$ 
such that for all $t$
\[ \|A_\eta(t)\|=\Big\|{\rm Id}+\int_{-\infty}^t \int_{-\infty}^s \mc{R}_\eta(r)A_\eta(r)dr ds \Big\|
\leq 1+ C\int_{-\infty}^t \int_{-\infty}^s  \rho^{m+2}(r)\|A_\eta(r)\|dr.\]
Now if $M_t=\sup_{s\leq t}\|A_\eta(s)\|$, we obtain using the uniform bound \eqref{uniformboundrho} and \eqref{firstboundrho} that there is $C>0$ independent of $t$ and $\eta$ such that
\[ M_t \leq 1+ \frac{CM_t}{|t|^{m}}\]
and thus there is $T>0$ independent of $\eta$ such that 
\[ \forall t\leq -T,\quad \|A_\eta(t)\|\leq 2.\]
This also tell us that there is $C>0$ and $T>0$ such that for all $\eta$
and $t<-T$ 
\begin{equation}\label{limderA-infty}
\|\pl_tA_\eta(t)\|\leq  \frac{C}{|t|^{m+1}},
\end{equation}
\begin{equation}\label{A-Id}
\|A_\eta(t)-{\rm Id}\|\leq \frac{C}{|t|^{m}}.
\end{equation}

\begin{lemma}\label{lim+infty}
Assuming the scattering map of $g$ agrees with that of $g_0$, there is a smooth family $H_\eta\in {\rm SO}(n-1)$, $C>0$ and $T>0$ large such for all $\eta$ and $t\geq T$
\begin{equation}\label{asymptofAeta+infty} 
\|A_\eta(t)-H_\eta\|\leq \frac{C}{t^{m}}.
\end{equation}
Moreover there is $C>0$ and $R>0$ such that for all $\eta$ with $|\eta|>R$, we have 
\begin{equation}\label{Beta-Id}
\|H_\eta-{\rm Id}\|<C/|\eta|
\end{equation} 
\end{lemma}
\begin{proof}
We choose the frame $W_j$ of Lemma \ref{parallelbasis} to be 
$(\rho^{-1} \partial_{x_j})|_{T_{p_-}\pl\bbar{\R^n}}$. Recall that $\pl_{x_j}$ is parallel for the Euclidean metric $g_0$. 
Assuming that the limit of $A_\eta(t)$ exists, then it satisfies
\[ \lim_{t\to \infty}A_\eta(t)\alpha=\Big(\lim_{t\to +\infty} g\Big(\bbar{Q}_\eta(\tau(t,\eta))\sum_{j=1}^{n-1}\alpha_j W_j^*,Y_i\Big)(\gamma_\eta(t))\Big)_{i=1,\dots,n-1}.\] 
To prove the existence of the limit, 
we shall show that, if $S_g=S_{g_0}$, then for each $\nu\in T_{p_-}^*\pl\bbar{\R^n}$
\begin{equation}\label{egaliteQ}
\lim_{\tau \to \tau^+_{g}(p_-,\eta)} \Big(\bbar{Q}^{g}_\eta(\tau)\nu\Big)^{\perp_g}=\lim\limits_{\tau \to\tau^+_{g_0}(p_-,\eta)}  \Big(\bbar{Q}^{g_0}_\eta(\tau) \nu\Big)^{\perp_{g_0}} 
\end{equation} 
where $\bbar{Q}_\eta^{g}$ and $\bbar{Q}_\eta^{g_0}$ are the maps \eqref{derivativespsi}
for the Riemannian metrics $g$ and $g_0$, and as before $\perp_{g}$ denotes the projection orthogonal to  
$\dot{\gamma}_\eta$ and similarly for $g_0$.
To simplify notation we set $(p_+ ,\eta_+(s)) = S_{g_0}(p_-,\eta+ s\nu)$. Observe that $p_+$ is the image of $p_-$ under the antipodal map $\Theta: \mathbb{S}^{n-1} \to \mathbb{S}^{n-1}$ and 
$(\eta_+(s))^\sharp=d\Theta_{p_-}(\eta+ s\nu)^\sharp$.
We compute for $\tau <\tau^+(p_-,\eta):=\tau_g^+(p_-,\eta)$
\[\bbar{Q}^g_\eta(\tau) \nu = \partial_s\pi_0 (\bbar{\varphi}_\tau(p_-, \eta+ s\nu))|_{s=0}  = 
\partial_s \pi_0 (\bbar{\varphi}_{\tau^+(p_-,\eta+s\nu)-\tau}(p_+, -\eta_+(s)))|_{s=0}.\]
Note that since $\bbar{X}^{\perp_g} = 0$, 
\[(\bbar{Q}^g_\eta(\tau) \nu)^{\perp_g} = (\partial_s \pi_0(\bbar{\varphi}_{\tau^+(p_-,\eta)-\tau}(p_+, -\eta_+(s)))|_{s=0})^{\perp_g} = ( \partial_s \pi_0 (\bbar{\varphi}_{u}(p_+, -\eta_+(s))|_{s=0}))^{\perp_g}\]
with $u:=\tau^+(p_-,\eta)-\tau$.
The expression of $\pi_0( \bbar\varphi_{u}(p_+, \eta_+(s)))$ near $\pl\bbar{\R^n}$ is given by
\begin{equation}
\label{in coordinates}
\pi_0 (\bbar{\varphi}_{u}(p_+,- \eta_+(s)) = \bbar{\gamma}_{u}(p_+,-\eta_+(s))=:
(\rho(u), y(u)), 
\end{equation}
and by \eqref{behavbargamma} we have that 
\[\rho(u) = u + u^3 f_s(u), \quad y(u) = p_+ - u\left(\eta_+(s)\right)^\sharp + u^2 r_s(u)\] 
for some $f_s(u)$ and $r_s(u)$  smooth in $(s,u)$. Observe that since $S_g=S_{g_0}=\Theta$, we have 
$\left(\eta_+(s)\right)^\sharp = d\Theta (\eta^\sharp) + s d\Theta(\nu^\sharp).$
Substituting these into \eqref{in coordinates} and differentiating with respect to $s$ at $s= 0$ we have
\[\begin{split}
\partial_s\bbar{\gamma}_{u}(p_+,-\eta_+(s))|_{s=0}= & \mc{O}(u^3)
\partial_\rho-ud\Theta(\nu^\sharp)+ \mc{O}(u^2)\partial_{y} 
= \mc{O}_g(u) -\rho d\Theta(\nu^\sharp).
\end{split}\]
The above expression shows that as $u \to 0$, $(\partial_s\bbar{\gamma}_{u}(p_+,-\eta_+(s)))|_{s=0}^{\perp_g}$ converges to the element of the scattering bundle given by $(\rho d\Theta(\nu^\sharp))|_{\partial \R^n}$ in the topology given by $g$. By the analogous computation $(\bbar{Q}_\eta^{g_0}(\tau) \nu )^{\perp_{g_0}}$ converges to the same vector as $\tau \to \tau^+_{g_0}(p_-,\eta)$ 
in the topology given by the Euclidean metric. This shows \eqref{egaliteQ} and the existence of $\lim_{t\to +\infty}A_\eta(t)$. We write $Y_i=\sum_{k}L_{ik} \pl_{x_k}$ for some 
matrix $L_{ik}$: since $(Y_j)_j$ is orthonormal for $g$ and $(\pl_{x_j})_j$ is orthonormal for $g_0$, 
the fact that $g-g_0\in \rho C^\infty(\bbar{\R^n};S^2(^{\rm sc}T^*\bbar{\R^n}))$ implies that 
$(L_{ik}(\gamma_\eta(t)))_{ik}$ converges to some $H_\eta\in {\rm SO}(n-1)$ matrix as $t\to \infty$ (the existence of the limit being guaranteed by Lemma \ref{parallelbasis}).
Now we get for $\nu\in T^*_{p_-}\pl\bbar{\R^n}$
\begin{equation}\label{limA}
\lim_{t\to +\infty} g\Big(\big(\bbar{Q}^{g}_\eta(\tau)\nu \big)^{\perp_{g_0}} ,Y_i\Big)(\gamma_\eta(t))
= - h_0\Big( d\Theta(\nu^\sharp) ,\sum_{k}(H_\eta)_{ik} (\rho^{-1}\pl_{x_k})|_{p_+}\Big)\end{equation} 
where we recall that $\rho^{-1}\pl_{x_k}$ extends smoothly at $\pl\bbar{\R^n}$ since $\pl_{x_k}$ is a smooth scattering vector field.
Applying the same reasoning to $g_0$, we get the same result with $H_\eta$ replaced by ${\rm Id}$.
and $A^{g_0}_\eta(t)={\rm Id}$. Since for $j=1,\dots,n-1$,
\[d\Theta_{p_-}W_j=d\Theta_{p_-}(\rho^{-1}\pl_{x_j})|_{T_{p_-}\pl\bbar{\R^n}}=-(\rho^{-1}\pl_{x_j})|_{T_{p_+}\pl\bbar{\R^n}}\]
we deduce from \eqref{limA} that $\lim_{t\to +\infty}A_\eta(t)=H_\eta$. In order to get the bound on the remainder in \eqref{asymptofAeta+infty}, we proceed as for the estimate on $A_\eta(t)$ for $t<-T$: since $\|H_\eta\|=1$, there is $T>0$ independent of $\eta$ such that for all $t>T$, $\|A_\eta(t)\|\leq 2$, then since $\pl_tA_\eta(t)\to 0$ as $t\to +\infty$ from \eqref{Jacobieq}, we get for all $t\geq T$
\[\|A_\eta(t)-H_\eta\|\leq 2 \int_t^\infty\int_s^\infty \|\mc{R}_\eta(r)\| drds\]
which gives the desired result since
\begin{equation}\label{Retabound}
\|\mc{R}_\eta(t)\|\leq \frac{C}{(|t|+\cjg\eta\cjd_{h_0})^{m+2}}
\end{equation} 
holds true by using \eqref{curvature decays to order m}, \eqref{firstboundrho} and \eqref{uniformboundrho}.

Finally  the proof of \eqref{Beta-Id} is obtained by combining \eqref{egaliteQ} with \eqref{almost euclidean frame}.
\end{proof}

\begin{corollary}\label{unifboundonA}
There is $C>0$ such that for all $t\in \R$ and all $\eta\in T_{p_-}^*\pl\bbar{\R^n}$,
\[ \|A_\eta(t)\|\leq C.\]
\end{corollary}
\begin{proof}
The bound follows from \eqref{A-Id} and Lemma \ref{lim+infty} when $t>T$, with $T$ uniform in $\eta$. For $t\in [-T,T]$, we can use that $A_\eta(t)$ solves the ODE \eqref{Jacobieq}, the uniform curvature bound \eqref{Retabound} and Gr\"onwall inquality to obtain that $A_\eta(t)$ is bounded uniformly with respect to $\eta$ for $t\in [-T,T]$.
\end{proof}

\begin{lemma}\label{uniformA}
Assuming $g$ has no conjugate points, then $A_\eta(t)$ is invertible for all $(\eta,t)$ and there exists $C>0$ such that for all $(t,\eta)\in \R\x T^*_{p_-}S^{n-1}$ we have 
\[ \|A_\eta^{-1}(t)\| \leq C.\]
\end{lemma}
\begin{proof}
The main part of the proof is to show that there is $C>0$ such that for $R>0$ large, 
for all $t\in \R$ and all $|\eta|_{h_0}>R$
\begin{equation}\label{boundlargeeta} 
\|A_{\eta}(t)-{\rm Id}\|<C/|\eta|.
\end{equation}
Recall that the matrix $A_\eta(t)$ solves the Jacobi equation
\[\ddot A_\eta(t) +{\mathcal R}_\eta(t) A_\eta(t) = 0.\]
Setting $E_\eta(t) := A_{\eta}(t)-{\rm Id}$,  observe that
\[\ddot E_\eta(t)  = -{\mathcal R}_\eta(t) A_\eta(t).\]
Thanks to \eqref{Retabound} and Corollary  \ref{unifboundonA} we have that 
\[\| {\mathcal R}_\eta(t) A_\eta(t)\| \leq \frac{C}{(|t|+\cjg \eta\cjd)^{m+2}} \]
for some $C>0$ uniform in $\eta,t$. Using that $E_\eta(t)\to 0$ as $t\to -\infty$ and $\pl_tE_\eta(t)\to 0$ as $t\to -\infty$, we have 
\[ \begin{split}
\|E_\eta(t)\|=  \Big\| \int_{-\infty}^t\int_{-\infty}^s {\mathcal R}_\eta A_\eta(r) dr ds\Big\| \leq  C \int_{-\infty}^t\int_{-\infty}^s  \frac{1}{(|r|+\cjg \eta\cjd)^{m+2}} dr ds 
\end{split}\]
for all $t$. This gives that for $t\leq 1/2$
\[\ \|E_\eta(t)\|\leq 
\frac{C}{(\cjg \eta\cjd+|t|)^{m}}.\]
Now for $t\geq 0$, we can also write, using $\dot{E}_\eta(t)=\dot{A}_\eta(t)\to 0$ as $t\to +\infty$ and $\lim_{t \to +\infty}E_\eta(t)=H_\eta-{\rm Id}$,
\[ E_\eta(t)=  H_\eta-{\rm Id}-\int_{t}^\infty \int_{s}^\infty {\mathcal R}_\eta A_\eta(r) dr ds
.\]
The same argument as above shows that for $t\geq 0$
\[\|E_\eta(t)\|\leq \|H_\eta-{\rm Id}\|+\frac{C}{(\cjg \eta\cjd+|t|)^{m}}.
\]
Using \eqref{Beta-Id}, we get that $\|E_\eta(t)\|\leq C/\cjg \eta\cjd$ for all $t$, showing \eqref{boundlargeeta}. The bound on $\|A_\eta^{-1}(t)\|$ for $\eta$ bounded and $t$ bounded follows from continuity of $A_\eta(t)^{-1}$ provided we know $A_\eta(t)$ is invertible. This last fact follows from the obsvervation that if $v\in \ker A_\eta(t_0)$, then $J_v(t):=\sum_{j=1}^{n-1}(A_\eta(t)v)_jY_j$ 
is an orthogonal bounded Jacobi field along $\gamma_\eta(t)$ vanishing at $t=t_0$, which by
\cite[Proposition 2.9]{Eberlein:1973aa} is not possible since $g$ has no conjugate points and curvature bounded below. Finally,  the bound on $\|A_\eta(t)^{-1}\|$ for large $|t|$ and $\eta$ bounded follows from  $A_\eta(t)\to {\rm Id}$ as $t\to -\infty$ and  $A_\eta(t)\to H_\eta$ as $t\to +\infty$. 
\end{proof}

\subsection{Estimates on unbounded Jacobi fields}

Next, we will consider the other Jacobi fields, namely those that grow at $t=-\infty$ along $\gamma_\eta(t)$. They can be obtained as before by variations of geodesics but rather by moving the endpoint $p_-$. Let $\mc{H}_{p_-,\eta}\subset T_{p_-,\eta}(T^*\pl \bbar{\R^n})$ be a horizontal subbundle such that $d\pi_0:\mc{H}_{p_-,\eta}\to T_{p_-}\pl \bbar{\R^n}$ is an isomorphism. For example, we can fix a local system of coordinates $y_1,\dots, y_{n-1}$ on $\pl\bbar{\R^n}=\mathbb{S}^{n-1}$ near $p_-$ so that  $h_0|_{p_-}=\sum_{i=1}^{n-1}dy_i^2$ and $\pl_{y_j}h_0|_{p_-}=0$ for all $j$, and covectors can be written 
$\eta=\sum_{i=1}^{n-1}\eta_idy_i$. We then take 
$\mc{H}_{p_-,\eta}={\rm span}(\pl_{y_1},\dots,\pl_{y_{n-1}})$ in these coordinates. 
Note that the coordinates $y_j$ can be chosen so that $(\rho^{-1}\pl_{x_j})|_{p_-}=\pl_{y_j}=W_j$, using Lemma \ref{parallelbasis} and  
$d\pi_0:\mc{H}_{p_-,\eta}\to T_{p_-}\pl\bbar{\R^n}$ is an isomorphism.
Now, define the linear map
\[ \bbar{P}_\eta(\tau):= d\pi_0.d\bbar{\varphi}_\tau(p_-,\eta)|_{\mc{H}_{p_-,\eta}}: \mc{H}_{p_-,\eta} \to T\bbar{\R^n}\]
and define the family of $(n-1)\x (n-1)$ matrices $\bbar{B}_\eta(\tau)$ by setting:
\begin{equation}\label{defbarB}
\bbar{B}_\eta(\tau)(\alpha_1,\dots,\alpha_{n-1})=\Big(
g\Big( \bbar{P}_\eta(\tau)\sum_{j=1}^{n-1}\alpha_j \pl_{y_j},Y_k\Big)\Big)_{k=1,\dots,n-1}
\end{equation}
where $Y_k$ are the parallel fields of Lemma \ref{parallelbasis}. 
We also define the $(n-1)\x (n-1)$ matrix
\begin{equation}\label{defBeta}
B_\eta(t):=\bbar{B}_\eta(\tau(t,\eta))
\end{equation} 
with $\tau(t,\eta)$  defined by \eqref{formulataut}. Just like for $A_\eta(t)$, $B_\eta(t)$ solves \eqref{Jacobieq}, but as we shall see, it grows linearly at $\pm\infty$.
Note that 
$\bbar{P}_\eta(\tau)=d\pi_0+\tau d\pi_0(\pl_y \bbar{X})+\mc{O}(\tau^2)$, and 
since $\pl_{y_j}h_0|_{p_-}=0$ we get that $d\pi_0(\pl_y \bbar{X})=0$ at $p_-$, which implies using \eqref{formrhoy}, \eqref{rote} and \eqref{Y_j-pl_xj} that as $t\to -\infty$
\[ 
B_{\eta}(t) =t\, {\rm Id}+G+\mc{O}(1/t). 
\]
for some matrix $G$ constant in $(t,\eta)$.
Notice that the corresponding matrix for the Riemannian metric $g_0$ is simply $B^{g_0}_\eta(t)=t\,{\rm Id}+G$: since $g-g_0$ decays to order $m$, it has to agree to order $\mc{O}(1/t)$ with $B_{\eta}(t)$ at $-\infty$ and to solve $\pl_t^2B_\eta^{g_0}(t)=0$. Up to adding $-GA_\eta(t)$ to $B_\eta(t)$ (so that it remains a solution of \eqref{Jacobieq}), we can assume  
\begin{equation}\label{limBetat}
B_{\eta}(t) =t\, {\rm Id}+\mc{O}(1/t) , \quad \textrm{ as }t\to -\infty
\end{equation}
and similarly $B_\eta^{g_0}(t)=t\, {\rm Id}$. 
To get a uniform in $\eta$ bound near $-\infty$, we proceed as for $A_\eta(t)$: since $B_\eta(t)$
solve \eqref{Jacobieq}, if $M_t=\sup_{s\leq t}\|s^{-1}B_\eta(s)\|$, then using \eqref{Retabound}
\[ \begin{split}
\|t^{-1}B_\eta(t)\| \leq & 1+ |t|^{-1}M_t\int_{-\infty}^t\int_{-\infty}^s \|r\mc{R}_\eta(r)\|drds
\leq  1+ C|t|^{-m}M_t
\end{split}\]
which means that there is $T>0$ independent of $\eta$ such that for all $t\leq -T$ and all $\eta$
\[ \|B_\eta(t)\|\leq 2|t|.\]
Now we can estimate using again \eqref{curvature decays to order m} that for $t\leq -T$ and all $\eta$
\begin{equation}\label{uniformBt<T}
\| B_\eta(t)-t\,{\rm Id}\|\leq 2\int_{-\infty}^t\int_{-\infty}^s \|r\mc{R}_\eta(r)\|drds\leq 
\frac{C}{|t|^{m-1}}.
\end{equation}
Let us now prove that $B_\eta$ has a controlled behavior as $t\to +\infty$.
\begin{lemma}\label{limB}
There is $C>0$ and $T>0$ such that for all $\eta$ and all $t\geq T$
\[ \|B_\eta(t)- t H_\eta\|\leq \frac{C}{t^{m-1}}\]
where $H_\eta\in {\rm SO}(n-1)$ is the matrix of Lemma \ref{lim+infty}.
\end{lemma}
\begin{proof}
We proceed as in Lemma \ref{lim+infty}. If $S_g=S_{g_0}$, then we claim that for 
$\tau>0$ small and $j=1,\dots,n-1$
\begin{equation}\label{egaliteP}
 \Big(\bbar{P}^{g}_\eta(\tau^+_{g}(p_-,\eta)-\tau)\pl_{y_j}\Big)^{\perp_g}=  \Big(\bbar{P}^{g_0}_\eta(\tau^+_{g_0}(p_-,\eta)-\tau) \pl_{y_j}\Big)^{\perp_{g_0}} +\mc{O}(\tau).
\end{equation} 
where $\bbar{P}^g$ and $\bbar{P}^{g_0}$ are the map defined above, with respect to the respective Riemannian metrics $g$ and $g_0$. Let us prove \eqref{egaliteP}. For $s\in\R$ small, we set $(p_+(s) ,\eta_+(s)) = S_{g_0}(p_-(s),\eta)$ with $\eta=\sum_{j=1}^{n-1}\eta_jdy_j$ in the fixed coordinate system chosen above, and 
$p_-(s)$ is a curve in $\mathbb{S}^{n-1}$ such that $p_-(0)=p_-$, $\dot{p}_-(0)=\pl_{y_j}$. 
Now $p_+(s)=\Theta(p_-(s))$ is the image of $p_-$ under the antipodal map $\Theta: \mathbb{S}^{n-1} \to \mathbb{S}^{n-1}$ and 
$(\eta_+(s))^\sharp=d\Theta_{p_-(s)}(\eta^\sharp)$.
We compute for $\tau <\tau^+(p_-,\eta):=\tau_g^+(p_-,\eta)$
\[\bbar{P}^g_\eta(\tau) \partial_{y_j} = \partial_s\pi_0 (\bbar{\varphi}_\tau(p_-(s), \eta))|_{s=0}  = 
\partial_s \pi_0 (\bbar{\varphi}_{\tau^+(p_-(s),\eta)-\tau}(p_+(s), -\eta_+(s)))|_{s=0}.\]
and
\[(\bbar{P}^g_\eta(\tau) \partial_{y_j})^{\perp_g} = (\partial_s \pi_0(\bbar{\varphi}_{\tau^+(p_-,\eta)-\tau}(p_+(s), -\eta_+(s)))|_{s=0})^{\perp_g} = ( \partial_s \pi_0( \bbar{\varphi}_{u}(p_+(s), -\eta_+(s))|_{s=0}))^{\perp_g}\]
with $u:=\tau^+(p_-,\eta)-\tau$. The curve $\pi_0( \bbar{\varphi}_{u}(p_+(s),-\eta_+(s)))$ near $\pl\bbar{\R^n}$ is of the form $(\rho(u),y(u))$ satisfying 
\[\rho(u) = u + u^3 f_s(u), \quad y(u) = \Theta(p_-(s)) - u\left(\eta_+(s)\right)^\sharp + u^2 r_s(u)\] 
for some $f_s(u)$ and $r_s(u)$  smooth in $(s,u)$. Differentiating in $s$ at $s=0$, 
\[\begin{split}
(\bbar{P}^g_\eta(\tau) \partial_{y_j})^{\perp_g}=& \mc{O}(u^3)\pl_\rho+d\Theta_{p_-}\pl_{y_j}-uD^2\Theta_{p_-}(\eta^\sharp,\pl_{y_j})+\mc{O}(u^2)\pl_y\\
& =d\Theta_{p_-}\pl_{y_j}-uD^2\Theta_{p_-}(\eta^\sharp,\pl_{y_j})+\mc{O}_g(u)
\end{split}\]
Assuming that $S_g=S_{g_0}$, we observe that we get the same exact formulas if we replace $g$ by $g_0$ in the estimate above, which implies \eqref{egaliteP}. Now, notice that from \eqref{parallel ODE} and \eqref{kristof decay}, we have that $Y_j(\gamma_\eta(t))=
\sum_{k=1}^{n-1}(H_\eta)_{jk}\pl_{x_k}+\mc{O}(|t|^{-1})$ as $t\to +\infty$. Thus, using the definition of $B_\eta(t)$ (more precisely its modification so that \eqref{limBetat} holds), we obtain that $B_\eta(t)-tH_\eta=\mc{O}(t^{-1})$ at $+\infty$ by using that $B_\eta^{g_0}(t)=t\,{\rm Id}$, and the uniform estimate on the remainder is obtained exactly as what we did for \eqref{uniformBt<T}.
\end{proof}

\subsection{Volume estimates using Jacobi tensors and rigidity under decaying conditions}

 Let $\{Y_j^*\}_{j=1}^n$ be an orthonormal frame defined on all of $\R^n$ such that along each geodesic $\gamma_\eta(t) = \theta(\gamma_\eta^0(t))$ with $\eta \in T^*_{p_-} \partial \R^n$, $Y_n^* = \dot \gamma_\eta^\flat$. Define the volume form $\omega_g := Y_1^* \wedge \dots \wedge Y_n^*$. The pullback of this volume form by $\psi$ (from Lemma \ref{foliationbygamma}) is a volume form on $\mathcal D$ given by 
\[\til{\psi}^* \omega_g = \det (A_\eta) W_1 \wedge \dots W_{n-1}\wedge dt = \det(A_\eta) {\rm dv}_{h_0}(\eta) \wedge dt\] 
where $\til \psi(t,\eta) := \psi( \tau(t,\eta),\eta) $, ${\rm dv}_{h_0}$ is the Riemannian volume form on $T_{p_-}^*\pl\bbar{\R^n}$, and $\tau(t,\eta)$ is given by the expression \eqref{formulataut}. 
Recall the definition of $F_R$ and $F_R^0$ in \eqref{FR0 and F_R}. By definition of the map $\theta$ and \eqref{formulataut} we have that if $\gamma_\eta(t) = \theta(\gamma_\eta^0(t),\partial_{x_n})$ 
\begin{align}
\label{F_R in tpsi}
F_R =\til{\psi}(\{ |t|\leq R, |\eta|_{h_0}\leq R\}).
\end{align}  

Using the change of variable $\til\psi$, the volume of $F_R$ for the Riemannian metric $g$ is given by 
\begin{equation}\label{VolFR}
{\rm Vol}_g(F_R) = \int_{|\eta|_{h_0}\leq R}\int_{|t|\leq R} \det(A_\eta) {\rm dv}_{h_0}(\eta)  dt.
\end{equation}
We next write two technical estimates that will be useful below.
\begin{lemma}\label{jensen}
Let $(\Omega,P)$ be a probability space and $X:\Omega\to \R^n$ be a continuous bounded random variable. Let $F:\R^n\to \R$ be a smooth map with positive definite Hessian. Then 
there is $C>0$ depending on $F$ and $\|X\|_{L^\infty}$ such that 
\[ \mathbb{E}(F(X))-F(\mathbb{E}(X))\geq C|X-\mathbb{E}(X)|_{L^2}^2.\]
\end{lemma}
\begin{proof}
We have for all $X\in \R^n$, $a\in \R^n$ and $K>0$
\[ F(X)-F(a)\geq DF(a)(X-a)+ C|X-a|^2 \]
where $C=\min_{|X|\leq K}\min \la_j(D^2F_X)$ where $\la_j(D^2F_X)$ are the eigenvalues of the Hessian $(D^2F_X)$ of $F$ at $X$, viewed as a quadratic form. We integrate this inequality and taking $K=\|X\|_{L^\infty}$ and $a=\mathbb{E}(X)$, we get using that $DF$ is linear
\[ \mathbb{E}(F(X))-F(\mathbb{E}(X))\geq C\mathbb{E}(|X-a|^2),\]
concluding the proof.
\end{proof}

\begin{corollary}\label{Jensendet}
Let $(A_t)_{t\in [-R,R]}$ be a family of real valued invertible $(n-1)\x(n-1)$ matrices and assume there is $C>0$ such that $\sup_{t\in [-R,R]}\|A^{-1}_t\|\leq C$ and $\det(A_t)>0$. Then there is $C'>0$ depending on $C$ such that for all $R>0$
\begin{equation}\label{jensenfinal}
\begin{gathered}\Big(\int_{-R}^R  f(t)^{-\frac{2}{n-1}}dt\Big)^{-\frac{n-1}{2}} 
- \det\Big(\int_{-R}^{R}A_t^{-1}(A_{t}^{-1})^*dt\Big)^{-\frac{1}{2}}\\
\geq 
C\Big(\int_{-R}^Rf(t)^{-\frac{2}{n-1}}dt\Big)^{-\frac{n+1}{2}}\int_{-R}^{R} \Big\|A_t^{-1}(A_{t}^{-1})^*-\frac{\int_{-R}^RA_s^{-1}(A_{s}^{-1})^*ds}{f(t)^{\frac{2}{n-1}}\int_{-R}^{R}f(s)^{-\frac{2}{n-1}}ds}\Big\|^2 f(t)^{\frac{2}{n-1}}dt
\end{gathered}
\end{equation}
where $f(t)=\det(A_t)$.
\end{corollary}
\begin{proof}
Let ${\rm Sym}_0(n-1)$ be the set of symmetric $(n-1)\times (n-1)$ matrices which are invertible 
and let $F: {\rm Sym}_0(n-1)\to \R$ the function $F(X)=\det(X)^{-1/2}$. This smooth function 
is strictly convex in the sense that it has positive definite Hessian at every point. 
Let $X_t:=A_t^{-1}(A_{t}^{-1})^*f(t)^{\frac{2}{n-1}}$ with $f(t):=\det(A_t)$ and 
consider the probability space $[-R,R]$ with density $p_t:=\frac{1}{M(T)}f(t)^{-\frac{2}{n-1}}dt$
with $M(R):=\int_{-R}^Rf(t)^{-\frac{2}{n-1}}dt$. 
We can then apply Lemma \ref{jensen} with $F$: using that $\det(X_t)=1$ we get 
\[\begin{gathered}
1- \Big(\int_{-R}^Rf(t)^{-\frac{2}{n-1}}dt\Big)^{\frac{n-1}{2}}\det\Big(\int_{-R}^{R}A_t^{-1}(A_{t}^{-1})^*dt\Big)^{-\frac{1}{2}} \geq \\
 C'\Big(\int_{-R}^Rf(t)^{-\frac{2}{n-1}}dt\Big)^{-1}\int_{-R}^R\Big\|A_t^{-1}(A_{t}^{-1})^*-\frac{\int_{-R}^RA_s^{-1}(A_{s}^{-1})^*ds}{f(t)^{\frac{2}{n-1}}\int_{-R}^Rf(s)^{-\frac{2}{n-1}}ds}\Big\|^2 f(t)^{\frac{2}{n-1}}dt
\end{gathered}\] 
which can be rewritten under the form \eqref{jensenfinal}.
\end{proof}
We can finally prove the 
\begin{proposition}
Let $g$ be an asymptotically Euclidean metric to order $m>n+1$ on $\R^n$ with no conjugate points and such that 
$S_g=S_{g_0}$. Then the curvature of $g$ is zero and $g$ is thus isometric on $g_0$.
\end{proposition}
\begin{proof}
First we use the stability estimate in H\"older equality (see \cite[Theorem 2.2]{Aldaz:2008uq} applied with 
the pair of functions $\det(A_\eta)^{\frac{2}{n+1}}$ and $\det(A_\eta)^{-\frac{2}{n+1}}$ and with the exponent $p=(n+1)/2$ and $q=(n+1)/(n-1)$) and we get with $f_\eta(t)=\det A_{\eta}(t)$
\begin{equation}\label{holderstab}
\int_{-R}^R f_\eta \geq (2R)^{\frac{n+1}{2}}\Big(\int_{-R}^R f_\eta^{-\frac{2}{n-1}}\Big)^{-\frac{n-1}{2}}\left(1+
c_n  \left\|\frac{f_\eta^{\frac{1}{2}}}{(\int_{-R}^R f_\eta)^{\frac{1}{2}}}
-\frac{f_\eta^{-\frac{1}{n-1}}}{(\int_{-R}^Rf_\eta^{-\frac{2}{n-1}})^{\frac{1}{2}}}\right\|_{L^2([-R,R])}^{2}\right)
\end{equation}
for some $c_n>0$ depending only on $n$.
We consider the volume of the region $F_R$ from \eqref{VolFR}: using \eqref{holderstab} and 
 Corollary \ref{Jensendet} with   $A_t := A_\eta (t)$, we get
\begin{equation}\label{VolF_R}
\begin{split}
{\rm Vol}_g(F_R) = & \int_{|\eta|\leq R}\int_{-R}^R \det A_\eta(t)\, dt {\rm dv}_{h_0}(\eta) \\
\geq & (2R)^{\frac{n+1}{2}}\int_{|\eta|\leq R}\det \Big(\int_{-R}^R A_\eta(t)^{-1}(A_\eta(t)^{-1})^* dt\Big)^{-\frac{1}{2}} {\rm dv}_{h_0}(\eta)\\
 & + c_n \int_{|\eta|\leq R}G_{\eta}^1(R){\rm dv}_{h_0}(\eta) +\int_{|\eta|<R}G^2_\eta(R){\rm dv}_{h_0}(\eta)
\end{split}
\end{equation}
where we have denoted 
\[G_\eta^1(R):=2R\Big(\frac{1}{2R}\int_{-R}^R f_\eta^{-\frac{2}{n-1}}\Big)^{-\frac{n-1}{2}} \int_{-R}^R\Big|\frac{f_\eta(t)^{\frac{1}{2}}}{(\int_{-R}^R f_\eta)^{\frac{1}{2}}}
-\frac{f_\eta(t)^{-\frac{1}{n-1}}}{(\int_{-R}^Rf_\eta^{-\frac{2}{n-1}})^{\frac{1}{2}}}\Big|^2dt\]
\[\begin{split}
G_{\eta}^2(R):=\Big(\frac{1}{2R} & \int_{-R}^Rf_\eta(t)^{-\frac{2}{n-1}}dt\Big)^{-\frac{n+1}{2}}
 \\
& \times \int_{-R}^{R} \Big\|A_\eta(t)^{-1}(A_\eta(t)^{-1})^*-\frac{\int_{-R}^RA_\eta(s)^{-1}(A_\eta (s)^{-1})^*ds}{f_\eta(t)^{\frac{2}{n-1}}\int_{-R}^{R}f_\eta(s)^{-\frac{2}{n-1}}ds}\Big\|^2 f_\eta(t)^{\frac{2}{n-1}}dt.\end{split}
\]
Now let $B^R_\eta (t):=A_\eta(-R+t)\Big(\int_{-R}^{-R+t} A_\eta(s)^{-1}(A_\eta(s)^{-1})^* ds\Big)(A_\eta(-R))^*$: it is direct to check that $B^R_\eta(t)$ solves the Jacobi equation
\eqref{Jacobieq} with conditions $B^R_\eta(0)=0$ and 
\begin{equation}\label{valueat-R}
  \pl_tB^R_\eta(0)= {\rm Id}.
\end{equation}
Let us compute the asymptotic behavior of $B_\eta^R(t)$ as $t\to -\infty$: using \eqref{limA-infty} and \eqref{A-Id}, we obtain for $|\eta|\leq R$ and $t\to -\infty$
\[\begin{split} 
B_\eta^R(t)= & \Big({\rm Id}+ \mc{O}\Big(\frac{1}{|t-R|^{m}}\Big)\Big)\int_{-R}^{-R+t}
\Big({\rm Id}+\mc{O}\Big(\frac{1}{|s|^{m}}\Big)\Big)ds A_\eta(-R)^*\\
=& tA_\eta(-R)^*+ C_R(\eta)A_\eta(-R)^*+\mc{O}(t^{-1}) 
\end{split} \]
for some matrix $C_R(\eta)$ such that $\|C_R(\eta)\|\leq CR^{-m-1}$ for some $C$ independent of $R,\eta$. 
We conclude that 
\[ B_\eta^R(t)=(B_\eta(t)+A_\eta(t)C_R(\eta))A_{\eta}(-R)^*.\]
Using Lemma \ref{lim+infty} and \ref{limB}, there is $C>0$ 
such that for all $R$ 
\begin{equation}\label{boundonBR}
\|B_\eta^R(2R)-2R H_\eta\|\leq CR^{-m+1}.
\end{equation}
This implies that
\begin{equation}\label{estimateAA*}
\begin{split}
\int_{-R}^{R} A_\eta(s)^{-1}(A_\eta(s)^{-1})^* ds= 
& 2RA_\eta(R)^{-1}H_\eta (A_\eta(-R)^{-1})^*+\mc{O}(R^{-m+1})\\
= & 2R(H_\eta^{-1}+\mc{O}(R^{-m}))H_\eta ({\rm Id}+\mc{O}(R^{-m})))\\
=& 2R\, {\rm Id}+{\mc O}(R^{-m+1})
\end{split}
\end{equation}
where the $\mc{O}$ are uniform with respect to $|\eta|\leq R$. Now we plug this estimate in 
\eqref{VolF_R} and obtain that there is $c_0>0$ independent of $R$ such that
\[\begin{split}
{\rm Vol}_g(F_R)+c_0R^{n-m} \geq &  {\rm Vol}_{g_0}(F_R^0)+ c_n \int_{|\eta|\leq R}G_{\eta}^1(R){\rm dv}_{h_0}(\eta) +\int_{|\eta|<R}G^2_\eta(R){\rm dv}_{h_0}(\eta).
\end{split}
\]
where $F_R^0$ is defined in \eqref{FR0 and F_R} and has volume $2R^n\int_{|\eta|_{h_0}\leq 1}{\rm dv}_{h_0}(\eta)$. 
Assuming now that $m>n+1$ and using Proposition \eqref{PropvolgFR}, we deduce that
\begin{equation}\label{limG_1G_2}
\lim_{R\to \infty}\int_{|\eta|\leq R}G_{\eta}^1(R){\rm dv}_{h_0} =0\quad \textrm{ and }
\lim_{R\to \infty}\int_{|\eta|\leq R}G_{\eta}^2(R){\rm dv}_{h_0} =0.
\end{equation}
Now, using \eqref{A-Id} and \eqref{asymptofAeta+infty}, we see that for all $\eta$ and for $R$ large $f_\eta(t)=\det A_\eta(t)=1+\mc{O}(1/|t|^{m})$ uniformly in $\eta$ if $|t|\geq T$, thus for large $R$
\begin{equation}\label{asymptoff}
\frac{1}{2R}\int_{-R}^Rf_\eta(t)dt= 1 +\mc{O}\big(\frac{1}{R}\big),\quad 
\frac{1}{2R}\int_{-R}^Rf_\eta^{-\frac{2}{n-1}}dt=1 +\mc{O}\big(\frac{1}{R}\big)
\end{equation}
where $\mc{O}(1/R)$ does not depend on $\eta$. Let us check that $f_\eta(t)=1$:
we get by \eqref{asymptoff}
\[ \begin{split}
\int_{|\eta|\leq R}G^1_\eta(R) & \geq \frac{1}{2} \int_{|\eta|\leq R}\int_{-R}^R \Big|\frac{f_\eta(t)^{\frac{1}{2}}}{(\tfrac{1}{2R}\int_{-R}^R f_\eta)^{\frac{1}{2}}}
-\frac{f_\eta(t)^{-\frac{1}{n-1}}}{(\tfrac{1}{2R}\int_{-R}^Rf_\eta^{-\frac{2}{n-1}})^{\frac{1}{2}}}\Big|^2dt
\\ 
& \geq \frac{1}{4} \int_{|\eta|\leq R}\int_{-R}^R \Big|{\big(\frac{1}{2R}\int_{-R}^Rf_\eta^{-\frac{2}{n-1}}\big)^{\frac{1}{2}}}
-f_\eta(t)^{-\frac{n+1}{2(n-1)}}\Big|^2f_\eta(t)dt.
\end{split}\]
Assume that $|f_\eta(t)-1|>\delta>0$  for some small $\delta$ on a bounded open set $U\subset \{|\eta|\leq R, |t|\leq R\}$ 
independent of $R$, then 
\[\begin{split}
\int_{|\eta|\leq R}G^1_\eta(R)\geq & \frac{1}{4}\int_{U}\Big|{\big(\tfrac{1}{R}\int_{-R}^Rf_\eta^{-\frac{2}{n-1}}\big)^{\frac{1}{2}}}
-f_\eta(t)^{-\frac{n+1}{2(n-1)}}\Big|^2f_\eta(t)dt\\
\geq & \frac{1}{4}\int_{U}|\mc{O}(\delta)+\mc{O}(\tfrac{1}{R})|^2f_\eta(t)dt\geq c_1\delta 
\end{split}
\]
for some uniform $c_1>0$ if $R$ is chosen large enough depending on $\delta$. Here we have used that $f_\eta(t)$ is bounded below by some positive constant in the relatively compact set $U$. Letting $R\to 0$, we obtain a contradiction with \eqref{limG_1G_2}. We conclude that $f_\eta(t)=1$ everywhere. We now get that 
\[G_{\eta}^2(R)=\int_{-R}^{R} \Big\|A_\eta(t)^{-1}(A_\eta(t)^{-1})^*-\frac{1}{R}\int_{-R}^RA_\eta(s)^{-1}(A_\eta(s)^{-1})^*ds\Big\|^2 dt.\]
We can use the same argument as above: assume that 
$\|A_\eta(t)^{-1}(A_\eta(t)^{-1})^*-{\rm Id}\|>\delta>0$ in some bounded open set $U$, then for $R$ large enough we obtain using \eqref{estimateAA*},
\[ \int_{|\eta|\leq R}G^2_\eta(R){\rm dv}_{h_0}(\eta)\geq c_1\delta\]
for some uniform $c_1>0$ if $R$ is chosen large enough depending on $\delta$. This contradicts  \eqref{limG_1G_2} as $R\to \infty$, and therefore $A_\eta(t)^{-1}(A_\eta(t)^{-1})^*={\rm Id}$ for all $\eta,t$. Now, just as for $B_\eta^R(t)$, it is direct to check that
$A_\eta(t)\Big(\int_{0}^{t} A_\eta(s)^{-1}(A_\eta(s)^{-1})^* ds\Big)(A_\eta(0))^*$ is a solution of \eqref{Jacobieq}, thus $tA_\eta(t)$ is also a solution, which gives
\[ \pl_t^2(tA_\eta(t))=2\pl_tA_\eta(t)+t\mc{R}_{\eta}(t)A_\eta(t)=\mc{R}_\eta(t)(tA_\eta(t)) \]
and thus $A_\eta(t)$ is constant with respect to $t$. Letting $t\to -\infty$ and using \eqref{limA-infty} we conclude that $A_\eta(t)={\rm Id}$, and therefore $\mc{R}_\eta=0$ for all $\eta$. Since 
the role of $p_-$ is arbitrary on $\mathbb{S}^{n-1}$, we conclude that the curvature is flat everywhere.
\end{proof}

\section{Determination of the jets of $g$ at infinity}

In this section, we will assume that $g$ is \emph{asymptotically Euclidean to order $m\geq 1$}, 
written under the normal form of Lemma \ref{normalform}, and we will show that if $m\geq 3$, then the scattering map
$S_g$ determines the jets of the metric $h_\rho$ appearing in the normal form. 
For that purpose, we will consider the geodesics that are very far near infinity. In the coordinates $(\rho,y,\bbar{\xi}_0,\eta)$, they correspond to the regime $|\eta|_{h_0}\to +\infty$. 

It is convenient to fix $(y_0,\eta_0)\in T^*\pl\bbar{\R^n}$
with $|\eta_0|_{h_0}=1$ and consider the flow line $\tau\mapsto \bbar{\varphi}_{\tau}(y_0,\eps^{-1}\eta_0)$ where $\eps>0$ is a small parameter tending to $0$. 
We can work on the manifold $\mc{M}$ defined in \eqref{defmcM} since the trajectory stays in the region $\{\rho\leq C\eps\}$ by Corollary \ref{rholeq1/R}. We let $c_\eps(s):=(\til{\rho}(s),\til{\xi}_0(s),\til{y}(s),\til{\eta}(s))\in 
[0,1]\x [-1,1]\x T^*\pl\bbar{\R^n}$ with
\[ \til{\rho}(s):=\eps^{-1}\rho(\eps s), \quad \til{\xi}_0(s):=\bbar{\xi}_0(\eps s),\quad \til{\eta}(s):=\eps\eta(\eps s),\quad \til{y}(s):=y(\eps s)\]
where $(\rho(\tau),\xi_0(\tau),y(\tau),\eta(\tau)$ is the coordinate representation of 
$\bbar{\varphi}_{\tau}(y,\eps^{-1}\eta)$. 
Using \eqref{formbarX} and the fact that $\bbar{\varphi}_{\eps s}(y,\eps^{-1}\eta)$ is an integral curve of $\bbar{X}$, it is direct to check that the curve $c_\eps(s)$ is an integral curve with initial condition $c_\eps(0)=(0,1,y_0,\eta_0)$ of the following smooth vector field 
\begin{equation}\label{tilXeps} 
\til{X}_\eps:= \til{\xi}_0\pl_{\til{\rho}}-\til{\rho}\big(|\til{\eta}|^2_{h_{\eps\til{\rho}}}+\tfrac{\eps\til{\rho}}{2}h'_{\eps\til{\rho}}(\til{\eta},\til{\eta})\big)\pl_{\til{\xi}_0}+
H_{\eps\til{\rho}}
\end{equation}
where the variables $\til{\rho},\til{\xi}_0,(\til{y},\til{\eta})$ live in the manifold $[0,1]\x [-1,1]\x T^*\pl\bbar{\R^n}$, we have denoted $h'_\rho:=\pl_{\rho}h_{\rho}$ and $H_{\eps\til{\rho}}$ is the Hamiltonian vector field of the Hamiltonian $(\til{y},\til{\eta})\mapsto \frac{1}{2}h_{\eps\til{\rho}}(\til{\eta},\til{\eta})$
on $T^*\pl\bbar{\R^n}$ with respect to the canonical symplectic structure. Its form in local coordinates is given by
\[H_{\rho}=\sum_{j,k=1}^{n-1}h_{\rho}^{jk}\til{\eta}_k\pl_{\til{y}_j}-\frac{1}{2}\sum_{j=1}^{n-1}\pl_{\til{y}_j}|\til{\eta}|^2_{h_\rho}\pl_{\til{\eta}_j}.\]
We see that $\til{X}_\eps$ is a 1-parameter family of smooth vector fields on a subset of $[0,1]\x [-1,1]\x T^*\pl\bbar{\R^n}$ depending smoothly on $\eps\in[0,\eps_0)$ and $\til{X}_\eps$ is a priori only defined on (and tangent to) the energy surface $\{\til{\xi}_0^2+\til{\rho}^2|\til{\eta}|^2_{h_{\eps\til{\rho}}}=1\}$. However the expression \eqref{tilXeps} extends smoothly to $[0,1]\x [-1,1]\x T^*\pl\bbar{\R^n}$ with smooth dependence on $\eps\in [0,\eps_0)$, we can thus view it as a smooth vector field defined everywhere on $[0,1]\x [-1,1]\x T^*\pl\bbar{\R^n}$. We shall then apply perturbation theory to describe the integral curve $c_\eps(s)$ as $\eps\to 0$. 

First, we can compute the integral curves for $\eps=0$, that is for $\til{X}_0$: writing $e^{tH_0}$ for the Hamiltonian flow of the Hamiltonian $\tfrac{1}{2}h_0(\til{\eta},\til{\eta})$ on $T^*\pl\bbar{\R^n}$ (which in turn is the geodesic flow for the unit sphere), we get
\[ c_0(s)=(\sin(s),\cos(s),e^{sH_0}(y,\eta)).\]
Here we have used the fact $|\til{\eta}(s)|_{h_0}=|\eta_0|_{h_0}$ for all $s$, a fact which follows directly from the expression of $\til{X}_0$ and $H_0|\til{\eta}|^2=0$.
For the moment, we assume that $g$ is asymptotically Euclidean to order $m\geq 1$ (later we shall take $m\geq 3$ but here $m\geq 1$ is sufficient), we can then write in normal form near $\pl\bbar{\R^n}$
\[ g=\frac{d\rho^2}{\rho^4}+\frac{h_\rho}{\rho^2}, \quad h_\rho^{-1}=h_0^{-1}+\rho^mh_m+\mc{O}(\rho^{m+1}).\] 
for some symmetric tensor $h_m\in C^\infty(\pl\bbar{\R^n};S^2T\pl\bbar{\R^n})$; here $h_\rho^{-1}$ denotes the metric on $T^*\pl\bbar{\R^n}$ dual to $h_\rho$ and similarly for $h_0$.  
By perturbation theory of ODEs, we know that $c_{\eps}(s)$ is a smooth family in $\eps$
, and using \eqref{tilXeps} we can write 
\[ \begin{gathered}
c_\eps(s)=c_0(s)+\eps^{m} c_m(s)+\mc{O}(\eps^{m+1}), \quad \til{X}_\eps=\til{X}_0+\eps^{m}\til{X}_m+\mc{O}(\eps^{m+1}),\\
\dot{c}_0(s)=\til{X}_0(c_0(s)), \quad \dot{c}_m(s)=d\til{X}_0(c_0(s))c_m(s)+\til{X}_m(c_0(s))
\end{gathered}\]
with initial condition $c_m(0)=(0,0,0,0)$ and $c_0(0)=(0,1,y_0,\eta_0)$ in local coordinates $(\til{\rho},\til{\xi}_0,\til{y},\til{\eta})$ near $c_0(0)$. Here the vector field $\til{X}_m$ is given by 
\begin{equation}\label{formulaXtilm}
\til{X}_m=-\left(\tfrac{m}{2}+1\right)\til{\rho}^{m+1}h_m\pl_{\bbar{\xi}_0}+\til{\rho}^m H_m
\end{equation} 
with $H_m$ being the Hamiltonian vector field of $\tfrac{1}{2}h_m(\til{\eta},\til{\eta})$ on $T^*\pl\bbar{\R^n}$.
We can also write the first component of $c_\eps(s)$ under the form 
\[\til{\rho}(s)=\sin(s)+\eps^m\til{\rho}_m(s)+\mc{O}(\eps^{m+1})\]
with $\til{\rho_m}$ the $\til{\rho}$ component of $c_m(s)$.
Let $R(s)$ be the the matrix valued solution of the linear differential equation
\begin{align}
\label{eq_for_R}
\dot{R}(s)=d\til{X}_0(c_0(s))R(s), \quad R(0)={\rm Id}.
\end{align}
We will use in $T^*\pl\bbar{\R^n}$ the variables $E:=|\til{\eta}|^2_{h_0}$ and set $\hat{\eta}:=\til{\eta}/\sqrt{E}$ for the cotangent vectors.  
 
Note that in the $(\til{\rho},\til{\xi}_0,E, \til{y},\hat{\eta})$ coordinates, we can write $d\til{X}_0$ under the matrix form
\[d\til{X}_0=\left(\begin{array}{c c c c}
0 & 1 & 0 & 0\\
-E & 0 & -\til{\rho} & 0\\
 0 & 0 & 0 & 0 \\
 0 & 0 & \pl_E H_0 & d_{\til{y},\hat{\eta}}H_0
\end{array}
\right).
\]
Thus, integrating the equation~\eqref{eq_for_R} at the energy layer $E=1$ where $c_0(s)$ lives, 
\[R(s)=\left(\begin{array}{c c c c}
\cos(s) & \sin(s) & a_1(s) & 0\\
-\sin(s) & \cos(s) & a_2(s) & 0\\
 0 & 0 & 1 & 0 \\
 0 & 0 & K(s) & L(s)
\end{array}
\right).
\]
Here, $L(s)$ solves $\dot{L}(s)=dH_{0}(e^{sH_0}(y,\hat{\eta}))L(s)$ with $L(0)={\rm Id}$ on the energy layer $\{E=1\}=S^*\pl\bbar{\R^n}$ and corresponds to the linearised geodesic flow of the unit sphere, $a_1(s),a_2(s)$ are functions satisfying
\[(\dot{a_1}(s),\dot{a_2}(s))=(a_2(s),-a_1(s)-\sin(s)) , \quad (a_1(0),a_2(0))=(0,0),\]
 and $K(s)$ is some vector solving $\dot{K}(s)=\pl_EH_0(c_0(s))+d_{(\til{y},\hat{\eta})}H_0(c_0(s))K(s)$ whose value is irrelevant.
The following result follows directly from \cite[Lemma 4.6]{Guillarmou:2019aa}, but we provide its proof (which is slightly simpler than the one in \cite{Guillarmou:2019aa}) for the reader's convenience.
\begin{lemma}\label{solutionlinearised}
Assume that $g$ is asymptotically Euclidean to order $m\geq 1$ and the scattering map $S_g$ coincides with the Euclidean scattering map $S_{g_0}$. 
Then, using the above notation, the following equation holds true at each $(y_0,\eta_0)\in S^*\mathbb{S}^{n-1}$
\begin{equation}\label{equationsurc0(s)}
R(\pi)\int_0^\pi R(s)^{-1}\til{X}_m(c_0(s))ds+\til{\rho}_m(\pi)\til{X}_0(c_0(\pi))=0.
\end{equation}
In particular, viewing $h_m$ as a function on $T^*\pl\bbar{\R^n}$, we have 
\begin{equation}\label{energyvar} 
\int_0^\pi \sin(s)^m H_0h_m(e^{sH_0}(y_0,\eta_0))ds=0.
\end{equation}
Assuming in addition that $H_0h_m=0$, we get
\begin{gather}
\label{rhom}
\til{\rho}_m(\pi)=-\left(\tfrac{m}{2}+1\right)\int_0^\pi \sin(s)^{m+2}h_m(e^{sH_0}(y_0,\eta_0))ds,\\
\label{eqcos}
\int_0^\pi \cos(s)\sin(s)^{m+1}h_m(e^{sH_0}(y_0,\eta_0))ds=0,\\
\label{eqdirectionH0}
\int_0^\pi \big(\sin(s)^m-\left(\tfrac{m}{2}+1\right)\sin(s)^{m+2}\big)h_m(e^{sH_0}(y_0,\eta_0))ds=0.
\end{gather}
\end{lemma}
\begin{proof}
First, by standard ODE argument, we directly obtain 
\[c_m(s)=R(s)\int_0^s R(t)^{-1}\til{X}_m(c_0(t))dt.\]
Next, $S_g(y_0,\eta_0/\eps)=S_{g_0}(y_0,\eta_0/\eps)$ is equivalent to having $c_\eps(\tau_\eps)=c_0(\pi)$
where $\tau_\eps>0$ is the travel time for the flow of $\til{X}_\eps$ defined by the equation $\til{\rho}(\tau_\eps)=0$. The asymptotic expansion as $\eps\to 0$ of this equation implies that
\[\tau_\eps=\pi+\eps^m\tau_m+\mc{O}(\eps^{m+1}) \textrm{ with } \tau_m d\til{\rho}(\til{X}_0)(c_0(\pi))+ \til{\rho}_m(\pi)=0.\] 
Since $d\til{\rho}(\til{X}_0)(c_0(\pi))=-1$, this implies that $\tau_m = \til{\rho}_m(\pi)$.
Writing the expansion as $\eps\to 0$ of $c_\eps(\tau_\eps)=c_0(\pi)$ and keeping the $\eps^m$ coefficient directly gives the  identity
\begin{equation}\label{eqsystemc}
 0=c_m(\pi)+\tau_m \dot{c}_0(\pi)=c_m(\pi)+\til{\rho}_m(\pi)\til{X}_0(c_0(\pi)),
\end{equation}
which can be rewritten as \eqref{equationsurc0(s)}.
Identifying the $\pl_E$ component of the equation, using that $dE.\til{X}_0=0$ and that $dE(\til{X}_m)=\til{\rho}^mdE.H_m=-2\til{\rho}^mH_0(h_m)$ (where we view 
$h_m$ as a homogeneous function of degree $2$ on $T^*\pl\bbar{\R^n}$), 
we obtain 
\[  \int_0^\pi \sin(s)^m H_0h_m(e^{sH_0}(y_0,\eta_0))ds=0.\]
Identifying the $\pl_{\til{\xi}_0}$ component, we get 
\[ \int_0^\pi \left(\tfrac{m}{2}+1\right)\sin^{m+1}(s)\cos(s)h_m(e^{sH_0}(y,\eta))+b(s)\sin(s)^{m}H_0h_m(e^{sH_0}(y,\eta)) ds=0\]
for some function $b(s)$, implying equation \eqref{eqcos} if $H_0h_m=0$.
Identifying the $\pl_{\til{\rho}}$ component of \eqref{equationsurc0(s)} and using $\til{X}_0(c_0(\pi))=-\pl_{\til{\rho}}+H_0$ and \eqref{formulaXtilm}, we obtain \eqref{rhom} if $H_0h_m=0$.
To obtain equation \eqref{eqdirectionH0}, we consider the $H_0$ component of \eqref{eqsystemc}. Since $d(e^{sH_0})_{y,\eta}.H_0(y,\eta)=H_0(e^{sH_0}(y,\eta))$, the direction $H_0$ is preserved by the linearisation of the geodesic flow $H_0$ on the unit sphere $\pl\bbar{\R^n}=\mathbb{S}^{n-1}$, and the same holds for 
$\ker \la$ if $\la:=\sum_{j}\til{\eta}_jdy_j$ is the Liouville $1$-form on $T^*\pl \R^n$. 
We then have $L(s)H_0(c_0(0))=H_0(c_0(s))$ and $L(s)^{-1}H_0(c_0(s))=H_0(c_0(0))$.
There is a natural projection on this direction by applying $\la$ to 
in \eqref{eqsystemc}, this gives 
\[\til{\rho}_m(\pi)+ \int_0^\pi \sin(s)^{m}\la(H_m(e^{sH_0}(y,\eta)))ds =0\]
But $\la(H_m)=h_m$, so we conclude that \eqref{eqdirectionH0} holds.
\end{proof}

\begin{corollary}\label{intagainstp}
Assume that $g$ is asymptotically Euclidean to order $m$ and the scattering map $S_g$ coincides with the Euclidean scattering map $S_{g_0}$. 
Then for each point $(y,\eta)\in S^*\pl\bbar{\R^n}$, 
\begin{equation}\label{parity}
(H_0h_m)(e^{\pi H_0}(y,\eta))=(-1)^mH_0h_m(y,\eta).
\end{equation}  
If $m\geq 3$ is odd (resp.\ even), for each homogeneous polynomial $p(y)$ of degree $3$ (resp.\ degree~$4$) 
and each point $(y,\eta)\in S^*\mathbb{S}^n$ we have 
\[ \int_{0}^{2\pi}p(e^{sH_0}(y,\eta))(H_0h_m)(e^{sH_0}(y,\eta))ds=0.\]
\end{corollary}
\begin{proof}
We apply $H_0$ to \eqref{energyvar} and integrate by parts in $s$ to get for each $(y,\eta)\in S^*
\mathbb{S}^{n-1}$ 
\[0=\int_0^\pi \sin(s)^{m}H_0^2h_m(e^{sH_0}(y,\eta))ds=
-m\int_0^\pi \cos(s)\sin(s)^{m-1}H_0h_m(e^{sH_0}(y,\eta))ds.\]
Applying this trick one more time and using \eqref{energyvar}, we have when $m>1$
\[0=\int_0^\pi \cos(s)^2\sin(s)^{m-2}H_0h_m(e^{sH_0}(y,\eta))ds,\]
which in turn, using \eqref{energyvar} gives $\int_0^\pi\sin(s)^{m-2}H_0h_m(e^{sH_0}(y,\eta))ds=0$. 
Repeating the operation iteratively, we end up in the case $m$ even with the identity 
$\int_0^\pi H_0h_m(e^{sH_0}(y,\eta))ds=0$, which gives \eqref{parity}.
In the case $m$ odd, we end up with $\int_0^\pi \sin(s)H_0h_m(e^{sH_0}(y,\eta))ds=0$. Applying again $H_0$ twice and integrating by parts, we get \eqref{parity} in the case $m$ odd.
Moreover the argument above shows that for the case $m\geq 4$ even
\begin{equation}\label{intcossin}
\begin{gathered}
\forall j\in \NN\cap [0,4], \quad \int_0^\pi \sin(s)^j\cos(s)^{4-j}H_0h_m(e^{sH_0}(y,\eta))ds=0\end{gathered}
\end{equation}
while if $m\geq 3$ is odd,
\begin{equation}\label{intcossinodd}
\forall j\in \NN\cap [0,3], \quad \int_0^\pi \sin(s)^j\cos(s)^{3-j}H_0h_m(e^{sH_0}(y,\eta))ds=0, 
\end{equation}
It is a direct consequence of \eqref{parity} that the same vanishing 
as \eqref{intcossin} (resp. \eqref{intcossinodd}) hold with $\int_0^{2\pi}$ instead of $\int_0^\pi$ when $m$ is even (resp. odd).
Assume $m$ odd: if $p(y)$ is a homogeneous polynomial of degree $3$ on $\mathbb{S}^{n-1}$, for each unit-speed geodesic $\gamma$ on $\mathbb{S}^{n-1}$ 
\[ \int_0^{2\pi}p(\gamma(s))(H_0h_m)(\gamma(s),\dot{\gamma}(s)^\flat)ds=\sum_{j=0}^3
a_i\int_0^{2\pi}\sin(s)^{j}\cos(s)^{3-j}(H_0h_m)(\gamma(s),\dot{\gamma}(s)^\flat)ds\]
for some real numbers $a_i$ depending on $\gamma$ and $p$, and this vanishes by \eqref{intcossin}. The case $m\geq 4$ even is similar. 
\end{proof}

We now employ Corollary~\ref{intagainstp} together with an argument related to that of Joshi-Sa Barreto in \cite[Proposition 3.2]{Joshi:1999aa} to obtain the following proposition.
\begin{proposition}\label{hmkilling}
Assume that $g$ is asymptotically Euclidean to order $m\geq 3$ and assume that the scattering map $S_g$ coincides with the Euclidean scattering map $S_{g_0}$. 
Then $H_0h_m=0$ on $S^*\mathbb{S}^{n-1}$, that is, $h_m$ is a Killing $2$-tensor for the Riemannian metric $h_0$ of the unit sphere $\mathbb{S}^{n-1}$.
\end{proposition}

Before proving Proposition~\ref{hmkilling}, we need a preliminary lemma. We denote by $\Sym^q(T^*\mathbb{S}^{n-1})$ the vector bundle over $\mathbb{S}^{n-1}$ of symmetric $q$-tensors. We recall that the symmetrized covariant derivative of symmetric tensor fields is the operator  
\begin{gather*}
D:C^\infty(\mathbb{S}^{n-1};\Sym^q(T^*\mathbb{S}^{n-1}))\to C^\infty(\mathbb{S}^{n-1};\Sym^{q+1}(T^*\mathbb{S}^{n-1})),\\
Dh (v_1,...,v_{q+1}):=\sum_{j=1}^{q+1} (\nabla_{v_j}h) (v_1,...,v_{j-1},v_{j+1},...,v_{q+1}),
\end{gather*}
where $\nabla$ denotes the Levi-Civita connection of $h_0$.

\begin{lemma}
\label{f=0}
Let $f\in C^\infty(\mathbb{S}^{n-1};\Sym^3(T^*\mathbb{S}^{n-1}))$ be such that, for each homogeneous polynomial $p\in\R[x_1,...,x_n]$, there is $k\in C^\infty(\mathbb{S}^{n-1};\Sym^2(T^*\mathbb{S}^{n-1}))$ such that $pf=Dk$. Then $f$ must be identically $0$.
\end{lemma}
\begin{proof} 
We will use a set of natural differential operators on tensors on $\mathbb{S}^{n-1}$ and their commutation relations, following the conventions of \cite{Heil:2016eu}.
The Riemannian metric on the vector bundle $\Sym^q(T^*\mathbb{S}^{n-1})$ is defined by 
\[ g(w_1 \cdot \ldots \cdot w_q, w'_1\cdot \ldots \cdot w'_q):=\sum_{\sigma \in \Sigma_q}
h_0(w_1,w_{\sigma(1)}')\ldots h_0(w_q,w_{\sigma(q)}'),\]
where $\Sigma_q$ denotes the set of permutations of $\{1,...,q\}$, and $\cdot$ denotes the symmetric tensor product
\begin{align*}
w_1\cdot...\cdot w_q=\sum_{\sigma\in\Sigma_q} w_{\sigma(1)}\otimes...\otimes w_{\sigma(q)}.
\end{align*}
The $L^2$-adjoint $D^*:C^\infty(\mathbb{S}^{n-1};\Sym^{q+1}(T^*\mathbb{S}^{n-1}))\to C^\infty(\mathbb{S}^{n-1};\Sym^{q}(T^*\mathbb{S}^{n-1}))$ is the divergence, which is given by 
\[ D^*h:=-\sum_{j=1}^{n-1}\iota_{e_j}\nabla_{e_j}h=-{\rm Tr}(\nabla h). 
\]
Here, $e_1,...,e_{n-1}$ is an $h_0$-orthonormal basis, $\iota_{e_j}$ denotes the interior product with $e_j$, and the trace map is defined as usual by 
\[{\rm Tr}(h)=\sum_{j=1}^{n-1}\iota_{e_j}\iota_{e_j}h.\]
Notice that $D=d$ is the exterior derivative on $0$-tensors, and $D^*=d^*$ on $1$-tensors.
The Lichnerowicz Laplacian on symmetric $p$-tensors is defined by $\Delta_L=\nabla^*\nabla+q(R)$ where $R$ denotes the curvature tensor and $q(R)$ is an endomorphism constructed out of $R$, see \cite[Section 6]{Heil:2016eu} for the exact formula. The following formulas, which actually hold on general Riemannian manifolds, relate the operators introduced so far. For each symmetric $q$-tensor field $h$, we have 
\begin{align}
\label{commuttrace_1}
{\rm Tr}(Dh) & =D{\rm Tr}(h)-2D^*h,\\ 
\label{commuttrace_2}
\quad {\rm Tr}D^*h & =D^*{\rm Tr}(h), \\
\label{DeltaL}
\Delta_Lh & =D^*Dh-DD^*h+2q(R)h,
\end{align}
and, if $q=1$,
\begin{equation}\label{formulaDeltaL1}
\Delta_Lh=\nabla^*\nabla h +{\rm Ric}_{h_0}h=\Delta_Hh:=(dd^*+d^*d)h 
\end{equation}
where $\Delta_H$ denotes the Hodge Laplacian, $d$ the exterior derivative and $d^*$ its $L^2$-adjoint. On the unit sphere $(\mathbb{S}^{n-1},h_0)$, we have
\begin{equation}\label{constcurvqR}
q(R)h=\left\{\begin{array}{@{}ll}
(n-2)h, & \textrm{if }h \textrm{ is a 1-tensor},\\
2(n-1)h^{\rm tf}=2(n-1)(h-\frac{1}{n-1}{\rm Tr}(h)h_0), & \textrm{if }h\textrm{ is a 2-tensor}.
\end{array}\right.
\end{equation}
By applying the argument in \cite[Lemma 5.2]{Hadfield:2017} (originally developed for the hyperbolic space) to the unit sphere $(\mathbb{S}^{n-1},h_0)$, we obtain 
\begin{align}
\label{commutDeltaD_1}
D^*\Delta_L & =\Delta_LD^*,\\
\label{commutDeltaD_2}
\quad \Delta_LD & =D\Delta_L.
\end{align}
Using first \eqref{commuttrace_1}, \eqref{commutDeltaD_1}, and \eqref{DeltaL}, the identity $pf=Dk$ implies 
\[ \begin{split}
\Delta_L {\rm Tr}(pf)=& -2\Delta_LD^*k+\Delta_LD{\rm Tr}(k)= -2D^*\Delta_Lk+d\Delta{\rm Tr}(k)\\
=& -2D^*(D^*Dk-DD^*k+2q(R)k)+d\Delta{\rm Tr}(k).
\end{split}\]
Using, in order, \eqref{commuttrace_1}, \eqref{constcurvqR} (first line),  \eqref{DeltaL}
\eqref{formulaDeltaL1}, \eqref{constcurvqR}, \eqref{commuttrace_2},
and finally \eqref{commutDeltaD_2}, we obtain
\[
\begin{split}
\Delta_L ({\rm Tr}(pf))=& -2D^*D^*(pf)-D^*D({\rm Tr}(Dk)-d{\rm Tr}(k))
 -8(n-1)D^*\big(k-\tfrac{{\rm Tr}(k)h_0}{n-1}\big)+ d\Delta {\rm Tr}(k)\\
=& -2D^*D^*(pf)-D^*D{\rm Tr}(pf)+(\Delta_Hd+d\Delta){\rm Tr}(k)-2(n-2)d{\rm Tr}(k)
\\
 & +4(n-1){\rm Tr}(pf)-4(n+1)d{\rm Tr}(k)+d\Delta{\rm Tr}(k)\\
=& -2D^*D^*(pf)-D^*D{\rm Tr}(pf)+4(n-1){\rm Tr}(pf) +(3\Delta-6n)d{\rm Tr}(k).
\end{split}\]
We next apply the exterior derivative $d$ to this identity to deduce the following  
\begin{equation}\label{eqontrace}
\begin{split}
dd^*d ({\rm Tr}(pf))& =-2dD^*D^*(pf)-dD^*D{\rm Tr}(pf)+4(n-1)d{\rm Tr}(pf)\\
&= -2dD^*D^*(pf)-dd^*d{\rm Tr}(pf)+2(3n-4)d{\rm Tr}(pf)
\end{split}\end{equation}
where we used again \eqref{DeltaL} and \eqref{formulaDeltaL1} for the last line.
The important point to observe is that this is an equation purely on $f$ (which does not involve $k$), with many possible choices for the polynomials $p$.
We shall now choose particular polynomials $p$. We fix a point $y_0\in \mathbb{S}^{n-1}$ and, up to using a rotation on the sphere, we can choose coordinates $x=(x_1,\dots,x_{n-1})\mapsto (x_1,\dots,x_{n-1},\sqrt{1-|x|^2})$ near $y_0$ so that $y_0=\{x=0\}$. It is direct to see that they are normal coordinates in the sense that $h_0=\sum_{j=1}^{n-1}dx_j^2+\mc{O}(|x|^2)$ near $x=0$. We will choose $p(x)=x_rx_sx_t$ vanishing to order $3$ at $x=0$, so that \eqref{eqontrace} gives us near $x=0$
\begin{equation}\label{lastidentTrpf} 
2dd^*d ({\rm Tr}(pf))+2dD^*D^*(pf)=\mc{O}(|x|).
\end{equation}
We write $f$ near $y_0=\{x=0\}$ in local coodinates under the form 
\[f=\sum_{l,i,j=1}^{n-1}f_{ijl}(x)\, dx_i\cdot dx_j\cdot dx_l\]
with $f_{ijl}$ symmetric in $(i,j,l)$.
We then compute that near $x=0$
\begin{equation}\label{dD*D*}
dD^*D^*(pf)= 12\sum_{j=1}^{n-1}f_{stj}(0) dx_r\wedge dx_j+f_{rtj}(0)dx_s\wedge dx_j+f_{srj}(0)dx_t\wedge dx_j +\mc{O}(|x|).
\end{equation}
We also have 
\begin{equation}\label{dDeltaTr}
\begin{split}
dd^*d{\rm Tr}(pf)=\, &6 (d\Delta p)\wedge \sum_{u,j=1}^{n-1}f_{uuj}(0)\,dx_j+\mc{O}(|x|)\\
 =\, &- 12 \sum_{u,j=1}^{n-1}f_{uuj}(0)(\delta_{rs}dx_t\wedge dx_j+\delta_{rt}dx_s\wedge dx_j+\delta_{st}dx_r\wedge dx_j)+\mc{O}(|x|) 
\end{split}
\end{equation}
Let us first choose $s= t\neq r$. In this case we get 
\[dd^*d{\rm Tr}(pf)=-12\sum_{u,j=1}^{n-1}f_{uuj}(0)dx_r\wedge dx_j+\mc{O}(|x|)\]
and \eqref{lastidentTrpf} yields 
\[0=-\sum_{u,j}f_{uuj}(0)dx_r\wedge dx_j+\sum_{j}\Big(f_{ssj}(0)dx_r\wedge dx_j+ 2
f_{rsj}(0)dx_s\wedge dx_j\Big).\]
Taking the component $dx_s\wedge dx_j$ of this equation with $j\notin\{r,s\}$, this implies that
\begin{equation}\label{frsj=0}
\forall j\notin \{r,s\},\,\, f_{rsj}(0)=0.
\end{equation}
Taking the component $dx_s\wedge dx_r$ gives 
\begin{equation}\label{fsss} 
\forall s\not=r, \,\,\sum_{u=1}^{n-1}f_{uus}(0)=f_{sss}(0)-2f_{rrs}(0).
\end{equation}
To obtain more information we make another choice for $p$, namely we choose the indices such that $r\neq s\neq t$ (three distinct indices, assuming $n\geq 4$). In this case, \eqref{dDeltaTr} vanishes at $x=0$, and \eqref{lastidentTrpf} implies 
\[ 0=\sum_{j=1}^{n-1}f_{stj}(0) dx_r\wedge dx_j+f_{rtj}(0)dx_s\wedge dx_j+f_{srj}(0)dx_t\wedge dx_j \]
Looking at the $dx_s \wedge dx_t$ component brings us the relation 
\[
\forall r\not\in\{s,t\},\,\, f_{ssr}(0) = f_{ttr}(0).
\]
Combining with \eqref{fsss}, we obtain for each $r\not=s$
\begin{equation}\label{frrs=0} 
(n-2)f_{rrs}(0)+f_{sss}(0)=f_{sss}(0)-2f_{rrs}(0), \textrm{ thus }f_{rrs}(0)=0.
\end{equation}
Finally let us choose the polynomial $p(x)=x_{r}^3$, i.e. $r = s = t$. With this choice of $p$, \eqref{lastidentTrpf} produces the identity 
\[0=-\sum_{u,j=1}^{n-1} f_{uuj}(0)\, dx_r \wedge dx_j + \sum_{j=1}^{n-1} f_{rrj}(0)\,  dx_r\wedge dx_j.\]
Considering the $dx_r\wedge dx_j$ component for $r\not= j$ we get 
$\sum_{u=1}^{n-1}f_{uuj}(0)=f_{rrj}(0)$ which is equal to $0$ by  \eqref{frrs=0}, leading then to 
$f_{jjj}(0)=0$. Combined with \eqref{frsj=0} and \eqref{frrs=0}, we conclude that $f(y_0)=0$ in dimension $n\geq 4$ and since $y_0$ is arbitrary, then $f=0$ everywhere.

The case $n = 3$ (i.e. $\partial \R^3 = {\mathbb S}^2$) requires a different argument. Going back to \eqref{fsss} and writing down the sum explicitly, we get $f_{rrs}(0) = -2 f_{rrs}(0)$ whenever $r\neq s$. This means that $f_{rrs}(0)= 0$ whenever $r\neq s$. Finally, considering $t=r=s$ as above, we obtain that $f_{jjj}(0)= 0$ for all $j = 1,2$ and this completes the proof for $n = 3$.
\end{proof}

\begin{proof}[Proof of Proposition~\ref{hmkilling}]
Assume first that $m$ is odd. The function $H_0h_m\in C^\infty(S^*\mathbb{S}^{n-1})$ can be expressed by means of the symmetrized derivative as 
\[
H_0h_m(y,\eta)=\tfrac{1}{3} (Dh_m)_y(\eta^\sharp,\eta^\sharp,\eta^\sharp).
\] 
Therefore, it corresponds to a symmetric 3-tensor 
\[f:=\tfrac{1}{3}Dh_m\in C^\infty(\mathbb{S}^{n-1};S^3T^*\mathbb{S}^{n-1}).\]
By Corollary \ref{intagainstp}, $f$ is odd and 
\[\int_{0}^{2\pi}p(\gamma(t))f_{\gamma(t)}(\dot{\gamma}(t),\dot{\gamma}(t),\dot{\gamma}(t))\,dt=0\]
for each closed geodesic $\gamma$ of $(\mathbb{S}^{n-1},h_0)$ and each homogeneous polynomial $p$ of degree $3$. By \cite{Estezet:1988,Goldschmidt:1990}, a symmetric $3$-tensor on $\mathbb{S}^{n-1}$ whose integral on all closed geodesics vanishes must be of the form $k_{\rm odd}+Dk$, where $k_{\rm odd}$ is odd with respect to the antipodal map $\Theta$ on $\mathbb{S}^{n-1}$ and $k$ is a symmetric $2$-tensor. Since $f$ is odd with respect to $\Theta$, for each homogeneous polynomial $p$ of degree $3$ the product $pf$ is even, and therefore $pf=Dk$ for some symmetric $2$-tensor $k$. By Lemma~\ref{f=0}, we conclude that $f=0$.

Assume now that $m$ is even. The tensor field $f':=x_1f$, where $x_1$ is the first coordinate of $\R^n$, is odd with respect to the antipodal map $\Theta$ and satisfies 
\[\int_{0}^{2\pi} p(\gamma(t))f'_{\gamma(t)}(\dot{\gamma}(t),\dot{\gamma}(t),\dot{\gamma}(t))\,dt=0\] 
for each closed geodesic $\gamma$ on $\mathbb{S}^{n-1}$ and each homogeneous polynomial  $p$ of degree $3$. The argument of the previous paragraph shows that $f'=0$, and therefore $f=0$ as well.
\end{proof}

We conclude with the following theorem.
\begin{theorem}\label{jetdeterm}
Assume that $g$ is asymptotically Euclidean to order $m\geq 3$ and the scattering map $S_g$ coincides with the Euclidean scattering map $S_{g_0}$. Then $g$ is asymptotically Euclidean to all order.
\end{theorem}
\begin{proof}
Assume that $m\geq 3$. By Proposition \ref{hmkilling}, we have that $H_0h_m=0$ when we view $h_m$ as a function on $S^*\mathbb{S}^{n-1}$, that is $h_m(e^{sH_0}(y,\eta))=h_m(y,\eta)$ for all $s\in\R$. We use \eqref{eqdirectionH0} to deduce that for each $(y,\eta)\in S^*\mathbb{S}^{n-1}$
\[h_m(y,\eta)\int_0^\pi (\sin(s)^m-(\tfrac{m}{2}+1)\sin(s)^{m+2})ds=0.\]
But $\int_0^\pi \sin(s)^{m+2}ds=\frac{m+1}{m+2}\int_0^\pi \sin(s)^{m}ds$ so we conclude that 
$h_m(y,\eta)=0$ if $m\geq 3$. 
Thus $g$ is asymptotically Euclidean to order $m+1$. Repeating the argument we obtain that $g$ is asymptotically Euclidean to all order. This concludes the proof.
\end{proof}

\bibliography{_biblio}
\bibliographystyle{amsalpha}

\end{document}